\documentclass[a4paper,reqno]{amsart}

\usepackage[french,british]{babel}
\usepackage{amssymb,amsmath,amsthm,amscd,mathrsfs,graphicx,color,multicol,booktabs,bm,manfnt}
\usepackage[cmtip,all,matrix,arrow,tips,curve]{xy}

\usepackage{mathabx}

\usepackage{manfnt}
\usepackage{centernot}   

\newcommand{\CC}{{\mathbb C}}

\newcommand{\cB}{{\mathscr B}}

\newcommand{\cD}{{\mathscr D}}

\newcommand{\cO}{{\mathscr O}}
\newcommand{\cP}{{\mathscr P}}  

\newcommand{\cR}{{\mathscr R}}
\newcommand{\cS}{{\mathscr S}}

\newcommand{\cV}{{\mathscr V}} 

\newcommand{\cX}{{\mathscr X}}

\newcommand{\es}{\emptyset}

\newcommand{\FF}{{\mathbb F}}

\newcommand{\gh}{\mathfrak{h}}

\newcommand{\gU}{\mathfrak{U}}
\newcommand{\gp}{\mathfrak{p}}

\newcommand{\HH}{{\mathbb H}}
\newcommand{\hra}{\hookrightarrow}
\newcommand{\KK}{{\mathbb K}}
\newcommand{\la}{\langle}

\newcommand{\lra}{\longrightarrow}

\newcommand{\n}{\noindent}
\newcommand{\NN}{{\mathbb N}}
\newcommand{\ov}{\overline}
\newcommand{\PP}{{\mathbb P}}
\newcommand{\QQ}{{\mathbb Q}}
\newcommand{\ra}{\rangle}
\newcommand{\RR}{{\mathbb R}}

\newcommand{\wh}{\widehat}
\newcommand{\wt}{\widetilde}
\newcommand{\ZZ}{{\mathbb Z}}
\newcommand{\Gr}{\mathrm{Gr}}

\theoremstyle{plain}

\newtheorem{thm}{Theorem}[section]

\newtheorem{clm}[thm]{Claim}

\newtheorem{crl}[thm]{Corollary}

\newtheorem{lmm}[thm]{Lemma}
\newtheorem{prp}[thm]{Proposition}
\newtheorem{prp-dfn}[thm]{Proposition-Definition}

\theoremstyle{definition}

\newtheorem{dfn}[thm]{Definition}

\newtheorem{ntn}[thm]{Notation}
\newtheorem{kob}[thm]{Key observation}

\theoremstyle{remark}

\newtheorem{expl}[thm]{Example}

\newtheorem*{qst*}{Main Question}
\newtheorem{rmk}[thm]{Remark}

\DeclareMathOperator{\Ann}{Ann}

\DeclareMathOperator{\charact}{char}

\DeclareMathOperator{\Det}{Det}
\DeclareMathOperator{\divisore}{div}

\DeclareMathOperator{\End}{End}

\DeclareMathOperator{\Id}{Id}
\DeclareMathOperator{\im}{Im}

\DeclareMathOperator{\KS}{KS}

\DeclareMathOperator{\Nm}{Nm}

\DeclareMathOperator{\PD}{PD}

\DeclareMathOperator{\Pf}{Pf}

\DeclareMathOperator{\re}{Re}
\DeclareMathOperator{\rk}{rk}

\DeclareMathOperator{\SL}{SL}
\DeclareMathOperator{\Span}{span}

\DeclareMathOperator{\Sym}{Sym}

\DeclareMathOperator{\Tors}{Tors}
\DeclareMathOperator{\Tr}{Tr}
\DeclareMathOperator{\vol}{vol}

\usepackage{ifthen}
\usepackage{rotating}
\usepackage{supertabular}
\newcommand{\cit}[1]{{\rm \textbf{#1}}}
\newcommand{\Ref}[2]{\cit{%
\ifthenelse{\equal{#1}{thm}}{Theorem}{}%
\ifthenelse{\equal{#1}{ass}}{Assumption}{}%
\ifthenelse{\equal{#1}{chp}}{Chapter}{}%
\ifthenelse{\equal{#1}{prp}}{Proposition}{}%
\ifthenelse{\equal{#1}{lmm}}{Lemma}{}%
\ifthenelse{\equal{#1}{cnj}}{Conjecture}{}%
\ifthenelse{\equal{#1}{crl}}{Corollary}{}%
\ifthenelse{\equal{#1}{dfn}}{Definition}{}%
\ifthenelse{\equal{#1}{expl}}{Example}{}%
\ifthenelse{\equal{#1}{hyp}}{Hypothesis}{}%
\ifthenelse{\equal{#1}{rmk}}{Remark}{}%
\ifthenelse{\equal{#1}{clm}}{Claim}{}%
\ifthenelse{\equal{#1}{exe}}{Exercise}{}%
\ifthenelse{\equal{#1}{qst}}{Question}{}%
\ifthenelse{\equal{#1}{sec}}{Section}{}%
\ifthenelse{\equal{#1}{subsec}}{Subsection}{}%
\ifthenelse{\equal{#1}{subsubsec}}{Subsubsection}{}%
\ifthenelse{\equal{#1}{univ}}{Universal Property}{}%
\ifthenelse{\equal{#1}{trm}}{Terminology}{}%
\ifthenelse{\equal{#1}{tbl}}{Table}{}%
\ifthenelse{\equal{#1}{ntn}}{Notation}{}%
\ifthenelse{\equal{#1}{kob}}{Key observation}{}%
\  \ref{#1:#2}%
}}

\usepackage{multirow}

 \makeindex
 
\begin{document}
 \title[Compact tori associated to HK manifolds of Kummer type]{Compact tori  associated to hyperk\"ahler manifolds  of Kummer type}
 \author{Kieran G. O'Grady}
\dedicatory{Dedicato alla piccola Mia}
\date{\today}
\thanks{Partially supported by PRIN 2015, Coll\`ege de France, ENS (PSL) and Universit\'e Paris VII}
  \maketitle
\bibliographystyle{amsalpha}
\section{Introduction}\label{sec:intro}
\subsection{Background and motivation}\label{subsec:retromotivi}
\setcounter{equation}{0}
Let $X$ be a  hyperk\"ahler manifold, i.e.~(for us) a simply connected compact K\"ahler manifold carrying a holomorphic symplectic form whose cohomology class spans $H^{2,0}(X)$. 
The Kuga-Satake construction~\cite{kuga,deligne-k3} associates to $X$ a compact complex torus $\KS(X)$, and an inclusion of Hodge structures 
$H^2(X)\subset H^1(\KS(X))\otimes H^1(\KS(X)^{\vee})$.
The definition of $\KS(X)$ is transcendental: one constructs a weight $1$ H.S.~out of the weight $2$ H.S.~on $H^2(X)$. If $X$ is projective with ample line bundle $L$, the Kuga-Satake construction applied to the primitive cohomology $H^2(X)_{pr}$ produces an abelian variety $\KS(X, L)$  and  an injective homomorphism of H.S.'s 
\begin{equation}\label{dueinuno}
H^2(X)_{pr}\subset H^1(\KS(X,L))\otimes H^1(\KS(X,L)).
\end{equation}
One might wonder whether it is possible to relate the geometry of $X$ and that of $\KS(X)$, or of $\KS(X,L)$. A famous instance of such a relation is provided by Deligne's proof  of the Weil conjectures for (projective) $K3$ surfaces starting from the validity of the Weil conjectures for abelian varieties~\cite{deligne-k3}. In this respect we notice that if $X$ is projective,  the Hodge conjecture predicts the existence of a \emph{Kuga-Satake algebraic cycle} on $X\times \KS(X,L)\times\KS(X,L)$  realizing  the homomorphism of H.S.'s in~\eqref{dueinuno}. 

There are very few families of hyperk\"ahler manifolds for which one has a geometric description of the corresponding Kuga-Satake varieties and a proof of existence of a Kuga-Satake algebraic cycle: Kummer surfaces~\cite{morrison-ks} and $K3$ surfaces obtained as minimal desingularization of the double cover of a plane ramified over 6 lines~\cite{paranjape}. 

The present paper grew out of the desire to understand the Kuga-Satake torus associated to hyperk\"ahler manifolds of Kummer type, i.e.~deformations of the $2n$-dimensional generalized Kummer manifold associated to an abelian surface (for $n\ge 2$).  

Among known examples of 
hyperk\"ahler manifolds, those of Kummer type are distinguished by the fact that they have non zero odd cohomology. Let  $X$ be such a manifold.  Then $b_3(X)=8$, and hence there is an associated 4 dimensional intermediate Jacobian $J^3(X)$. Most of our paper is actually concerned with $J^3(X)$. Our starting point is the proof that there is an analogue of the key cohomological property of the Kuga-Satake torus (see~\eqref{dueinuno}) valid with $J^3(X)$ replacing the Kuga-Satake torus. From this it follows that if $X$ is projective with polarization $L$, then $\KS(X,L)$ is isogenous to $J^3(X)^4$. Thus $J^3(X)$ is a smaller dimensional version of the Kuga-Satake torus.
Moreover, it is easier to relate geometrically  $X$ to   $J^3(X)$ than it is to relate it to $\KS(X)$ (or $\KS(X,L)$),  e.g.~via the Abel-Jacobi map. 

We will give an \emph{explicit} recipe that produces the weight $1$ H.S.~on $J^3(X)$ in terms of the weight $2$ H.S.~on $H^2(X)$. 

One fact that we discovered is that if $X$ is projective, then $J^3(X)$ is an abelian fourfold of Weil type. More precisely, as $(X,L)$ varies in a complete family of polarized hyperk\"ahlers of Kummer type with fixed discrete invariants, the corresponding polarized   intermediate Jacobians $J^3(X)$ sweep out a complete family of polarized 
abelian fourfolds of Weil type with fixed discrete invariants. Notice that the number of moduli for both families is equal to $4$. This result suggests that we will be able to describe \emph{explicitly} locally complete families of projective hyperk\"ahlers of Kummer type starting from the locally complete families of abelian fourfolds of Weil type which are known (see~\cite{schoen-hc}). In this respect, we notice that several  locally complete families of projective hyperk\"ahlers have been explicitly described, but the varieties in those familes are all of $K3^{[n]}$ type (deformations of the Hilbert scheme of length $n$ subschemes of a $K3$ surface). 

There is a series of papers related to the present work. The first one is~\cite{bert-ks}. Following the proof of Theorem 9.2 of that paper, one shows that the Kuga Satake $\KS(X,L)$ of a polarized HK of Kummer type $(X,L)$ is the fourth power of an abelian fourfold of Weil type. Since $\KS(X,L)$ is isogenous to $J^3(X)^4$, it follows that $J^3(X)$ is of Weil type. However we would like to stress that we have precise results on the \emph{integral} Hodge structure on $J^3(X)$, not only up to isogeny. Another paper related to this work is~\cite{lombardo-ks}. Lastly, the recent preprint~\cite{mar-on-weil} is strictly related to our work.

\subsection{Main results}\label{subsec:risulprinc}
\setcounter{equation}{0}
Let $X$ be a hyperk\"ahler  manifold  of dimension at least $4$, deformation equivalent to a generalized Kummer  variety (following established terminology, we say that $X$ is of \emph{Kummer type}). Then $b^3(X)=8$, see~\cite{gothesis}, and of course $H^{3,0}(X)=0$. Thus
\begin{equation}
J^{3}(X)=H^3(X)/(H^{2,1}(X)+ H^3(X;\ZZ))
\end{equation}
 is a  $4$ dimensional  compact complex torus. If $X$ is projective, and $L$ is an ample line bundle on $X$, then $J^{3}(X)$ is an abelian $4$-fold (all of $H^3(X)$ is  primitive because $H^1(X)=0$), and we let  $\Theta_L$ be the polarization defined by $L$.

Recall that, given a HK manifold $X$, there is a class $q^{\vee}_X\in H^{2,2}_{\QQ}(X)$ which corresponds to the Beauville-Bogomolov-Fujiki quadratic form of $X$   (see~\Ref{subsec}{defgenkum} for details). Now assume that  $X$ is of Kummer type, of dimension $2n$.  Then  $\ov{q}_X:=2(n+1)q^{\vee}_X$ is an integral class  (see~\Ref{dfn}{qubarra}). 
Let  
\begin{equation}\label{vanitoso}
\phi\colon \bigwedge^2 H^3(X) \lra H^2(X)^{\vee}
\end{equation}
be the composition of the map 
\begin{equation*}
\begin{matrix}
\bigwedge^2 H^3(X) & \lra & H^{4n-2}(X)\\
\gamma\wedge\gamma' & \mapsto & \gamma\smile \gamma'\smile \ov{q}_X^{n-2}
\end{matrix}
\end{equation*}
and the map $H^{4n-2}(X)\to H^2(X)^{\vee}$ defined by cup product followed by integration. 
\begin{thm}\label{thm:primoteor}
Let $X$ be a HK manifold of Kummer type, of dimension $2n$. 
\begin{enumerate}
\item
The map $\phi$ is surjective, and hence its transpose defines an inclusion of integral Hodge structures
\begin{equation}
H^2(X)\subset \bigwedge^2 H^1(J^3(X)).
\end{equation}
\item
The set
\begin{equation}
{\bf Q}(X):=\{[\gamma]\in \PP(H^3(X)) \mid \phi(\gamma\wedge H^3(X)) \not= H^2(X)^{\vee}\}
\end{equation}
is a smooth quadric hypersurface in $\PP(H^3(X))$. 
\item
The projectivization of $H^{2,1}(X)$ is a maximal linear subspace of ${\bf Q}(X)$.
\end{enumerate}
\end{thm}
If $X$ is a HK manifold of Kummer type, let ${\bf Q}^{+}(X)$ be  the irreducible component of the variety parametrizing maximal dimensional linear subspaces of ${\bf Q}(X)$ containing $\PP(H^{2,1}(X))$ (this definition makes sense by~\Ref{thm}{primoteor}). We recall that  
${\bf Q}^{+}(X)\subset \PP(S^{+}(X))$, where  $S^{+}(X)$ is one of the two spinor representations of $O({\bf Q}(X))$.  Recall also that $S^{+}(X)$ is $8$-dimensional. Since $H^3(X)$ has an integral structure, so does $S^{+}(X)$. There is a unimodular integral quadratic form ${\bf q}^{+}_X$ on 
$S^{+}(X)$ (unique up to multiplication by $\pm 1$) such that ${\bf Q}^{+}(X)$ is the set of zeroes of ${\bf q}^{+}_X$. 
Moreover, if $\pi\colon\cX\to B$ is a family of HK manifolds of Kummer type, the flat connection on $R^3\pi_{*}\ZZ$ induces a flat connection on the fibration $S^{+}(\pi)\to B$ with fiber $S^{+}(\pi^{-1}(b))$ over $b$.
Next, we make following
\begin{kob}\label{kob:chiave}
Let $\phi$ be the map in~\eqref{vanitoso}. Then $\phi(\bigwedge^2 H^{2,1}(X))$  is equal the one dimensional subspace $\Ann F^1 H^2(X)$. 
\end{kob}
In fact  $\phi(\bigwedge^2 H^{2,1}(X))$  is contained in $\Ann F^1 H^2(X)$ because $\phi$ is a morphism of Hodge structures, and   equality follows from   surjectivity of $\phi$. Notice that Item~(3) of~\Ref{thm}{primoteor} follows from the~\Ref{kob}{chiave}.

The result below is motivated by the~\Ref{kob}{chiave}.
\begin{thm}\label{thm:secondoteor}
Let $X$ be a HK manifold of Kummer type, of dimension $2n$. There exists a codimension $1$ subspace $T^{+}(X)\subset S^{+}(X)$ defined over $\ZZ$ such that the following hold:
\begin{enumerate}
\item
Given a $4$-dimensional vector subspace  $\Gamma\subset H^3(X)$, the subspace $\phi(\bigwedge^2 \Gamma)$ has dimension $1$ if and only if 
$\PP(\Gamma)$ is a linear subspace of ${\bf Q}(X)$ parametrized by a point  $[\sigma]\in\PP(T^{+}(X))\cap {\bf Q}^{+}(X)$. If this is the case, then $\phi(\bigwedge^2 \Gamma)=[\sigma]$.
\item
There exist an isomorphism $\iota\colon H^{2}(X)^{\vee}\overset{\sim}{\lra} S^{+}(X)$ defined over $\QQ$,  invariant  up to sign under monodromy, 
and a choice of \lq\lq  sign\rq\rq\ for ${\bf q}^{+}_X$, such 
 that the restriction  of ${\bf q}^{+}_X$ to $H^{2}(X)^{\vee}\subset S^{+}(X)$ is equal to  the dual  of the BBF quadratic form. 
\end{enumerate}
\end{thm}
Item~(1) of~\Ref{thm}{secondoteor} amounts to an \emph{explicit} description of the weight $1$ Hodge structure on $H^1(J^3(X))$ in terms of the 
weight $2$ Hodge structure on $H^2(X)$.

The result below was first proved by Mongardi by other methods. We will show that it is a simple consequence of~\Ref{thm}{secondoteor}. 
\begin{crl}[Mongardi~\cite{mongardimon}]\label{crl:monlim}
Let $X$ be a HK of Kummer type. Let $\rho\in O(H^2X;\ZZ),q_X)$ be a monodormy operator. Then either $\rho$ acts trivially on the discriminant group $H^2(X;\ZZ)^{\vee}/H^2(X;\ZZ)$ (here $H^2(X;\ZZ)$ is embedded into $H^2(X;\ZZ)^{\vee}$ by the BBF quadratic form) and it has determinant $1$, or 
it acts as multiplication by $-1$ on the discriminant group and it has determinant $-1$.
\end{crl}
Below is our last main result.
\begin{thm}\label{thm:terzoteor}
Let $X$ be a hyperk\"ahler variety  of Kummer type, of dimension $2n$, and let $L$ be an ample line bundle on $X$. Then $(J^3(X),\Theta_L)$ is of Weil type, with an inclusion 
$$\QQ\sqrt{-2(n+1)q_X(L)}\subset \End(J^3(X),\Theta_L)_{\QQ},$$
where $q_X(L)$ is the value of the Beauville-Bogomolov-Fujiki (BBF) quadratic form on $c_1(L)$.  
By varying  $(X,L)$, one gets a complete (up to isogeny)  family of $4$ dimensional abelian varieties of Weil type  with associated field $\QQ[\sqrt{-2(n+1)q_X(L)}]$, and trivial determinant.
Moreover, tha Kuga-Satake variety $\KS(X,L)$ is isogenous to  $J^3(X)^4$.
\end{thm}
\subsection{Organization of the paper}
\setcounter{equation}{0}
Most of~\Ref{sec}{genkum} 	is devoted to the proof of results on the cohomology of HK's of Kummer type. After recalling the definition of generalized Kummers, and establishing basic notation, we  compute the constants which enter into the formula for certain integrals on a HK of Kummer type (see~\Ref{prp}{bellaform}). 
In~\Ref{subsec}{intcohom} we describe explicitly the \emph{integral} 3rd cohomolgy group of a generalized Kummer. In dimension 4 this was done by Kapfer and Menet~\cite{cogekum}. We extend their result to arbitrary dimension by adapting arguments of Totaro~\cite{tothilb}. In~\Ref{subsec}{struttura} we show that, by invariance under the monodromy group of compact complex tori, the map $\phi$ in~\eqref{vanitoso} for $2n$-dimensional HK's of Kummer type  has a \lq\lq shape\rq\rq\ which depends on an apriori unknown  $\vartheta(\ov{q}^{n-2})\in\ZZ^3$. 
In~\Ref{subsec}{hilbring},  \Ref{subsec}{grancul} and~\Ref{subsec}{granculdue} we compute  the first two entries of $\vartheta(\ov{q}^{n-2})$ (the third entry will be determined up to sign in~\Ref{subsec}{twoone}). 
Most of the effort goes in a painful computation of the cup product of certain cohomology classes on a generalized Kummer. In order to do this we rely on the explicit description of the cohomology ring of Hilbert schemes of smooth projective surfaces with trivial canonical bundle given by Lehn and Sorger~\cite{mancrist}.  The last subsection of~\Ref{sec}{genkum} contains the proof of~\Ref{thm}{primoteor}.

In~\Ref{sec}{alglin} we prove~\Ref{thm}{secondoteor} and~\Ref{crl}{monlim}.  Actually we discuss an \lq\lq abstract\rq\rq\ map which has the same shape as $\phi$, depending on a choice of 
$\vartheta\in\ZZ^3$ with no vanishing entry. In such a set-up, we have a way of explicitly associating to a weight-$2$ H.S.~of $K3$ type a weight-$1$ H.S. If the weight-$2$ H.S.~is polarized, then the weight-$1$ H.S.~is also polarized.

In the short~\Ref{sec}{polandmon} we compute the elementary divisors of the natural polarization of $J^3(X)$ for a polarized HK fourfold $X$. 

\Ref{sec}{tipoweil} is devoted to the proof of~\Ref{thm}{terzoteor}. Actually we prove, more generally, that the polarized  weight-$1$ H.S.'s constructed in~\Ref{sec}{alglin}  (depending on 
a $\vartheta\in\ZZ^3$ with no vanishing entry) are of Weil type. 
\subsection{Conventions}\label{subsec:ginevra}
\setcounter{equation}{0}

We work over $\CC$:  projective varieties will be \emph{complex} projective varieties.

Throughout the present paper, $A$ is an abelian surface.

Notation: in dealing with cohomology, we omit to mention  the ring of coefficients when we consider complex coefficients. 

Let $\Lambda$ be a lattice. The \emph{divisibility} of a non zero $v\in\Lambda$ is the positive generator of $(v,\Lambda)$; we denote it by $\divisore(v)$.

Let $\n\in\NN_{+}$. The \emph{double factorial} of $n$ is equal to
\begin{equation}
n!!:=n\cdot (n-2)\cdot\ldots\cdot \left(n-2\left\lfloor\frac{n-1}{2}\right\rfloor\right). 
\end{equation}
It is convenient to set $0!!:=1$ and $(-1)!!:=1$.

\subsection*{Acknowledgements}
After giving a talk in Oberwolfach on some of the results of this paper~\cite{kieran-obw-2017}, I had various stimulating conversations related to this work. 
I am grateful to Francois Charles for  sharing an unpublished paper of his on the Kuga-Satake construction~\cite{charles-univ}. 
I am indebted to Eyal Markman, who pointed out the importance of triality - this motivated me to formulate~\Ref{thm}{secondoteor} in terms of spinor representations. 
 Thanks go to Eyal also for sending me his preprint~\cite{mar-on-weil} while I was finishing writing the present paper. 

It is a pleasure to thank Ruggero Bandiera for proving~\Ref{lmm}{ideban}.

\section{Generalized Kummers and their cohomology}\label{sec:genkum}
\subsection{Hilbert schemes parametrizing subschemes of finite length}\label{subsec:ischemia}
\setcounter{equation}{0}
Let $S$ be a smooth projective surface. Let $S^{[n]}$ be the Hilbert scheme parametrizing subschemes  $Z\subset S$ of length $n$, and let  $S^{(n)}$ be the symmetric product of $n$ copies of $S$. Given a point $[Z]\in S^{[n]}$, we let 
\begin{equation*}
|Z|:=\sum_{p\in S} \ell(\cO_{Z,p})p\in S^{(n)}.
\end{equation*}
The Hilbert-Chow map $\wt{\gh}_n\colon S^{[n]}\to S^{(n)}$ associates to $[Z]$ the cycle $|Z|$. Let $\wt{\Delta}_n(S)\subset S^{[n]}$ be the prime divisor parametrizing non reduced schemes.  The divisor class of $\wt{\Delta}_n(S)$ is  divisible by $2$. We let $\wt{\xi}_n(S)\in H^2(S^{[n]};\ZZ)/\Tors$ be characterized by
\begin{equation}\label{deldop}
2\wt{\xi}_n(S)=c_1(\cO_{S^{[n]}}(\wt{\Delta}_n(S))).
\end{equation}
 (In order to simplify notation, we will omit $S$  whenever there is no ambiguity.)

Let $R$ be a (commutative) ring. 
Given $\alpha \in H^m(S;R)$, let $\alpha^{(n)} \in H^m(S^{(n)};R)$ be characterized by the formula
\begin{equation}\label{simmetrizzo}
\pi^{*}\alpha^{(n)} =\sum_{i=1}^{n}p_i^{*}\alpha,
\end{equation}
 where  $\pi\colon S^{n}\to S^{(n)}$ is the quotient map and  $p_i\colon S^{n}\to S$ is the 
$i$-th projection. Let  
\begin{equation}
\begin{matrix}
 H^m(S;R) & \overset{\wt{\mu}_m}{\lra} &  H^m(S^{[n]};R) \\
 \alpha & \mapsto & \wt{\gh}_n^{*}\alpha^{(n)}.
\end{matrix}
\end{equation}
For $n\ge 2$ let
\begin{equation}
\Gamma_n(S):=\{(W,Z)\in  S^{[2]}\times S^{[n]} \mid W\in\wt{\Delta}_2\quad W\subset Z\}.
\end{equation}
Then $\Gamma_n(S)$ is irreducible of (complex) dimension $2n-1$. 
Let $p\colon \Gamma_n(S)\to S$ be the map sending $(W,Z)$ to the support of $W$, and let  $q\colon \Gamma_n(S)\to S^{[n]}$ be the projection. We let
\begin{equation}\label{nakmap}
\begin{matrix}
H^{m-2}(S;R) & \overset{\wt{\nu}_m}{\lra} & H^m(S^{[n]};R) \\
\beta & \mapsto & \PD(q_{*}([\Gamma_n(S)]\cap p^{*}\beta),
\end{matrix}
\end{equation}
where $\PD$ means Poincar\'e dual.

\subsection{Generalized Kummers}\label{subsec:defgenkum}
\setcounter{equation}{0}
Let  $A$ be an abelian surface.  Let $\sigma_r\colon A^{(r)}\to A$ be the summation map (in the group $A$). The $n$-th  generalized Kummer variety  is
\begin{equation*}
 K_n(A):=\{[Z] \in  A^{[n+1]} \mid \sigma_{n+1}(|Z|)=0\}.
\end{equation*}
Beauville~\cite{beaucy} proved that  $K_n(A)$ is a hyperk\"ahler variety of dimension $2n$. 
Let 
\begin{equation}\label{deldop}
 \Delta_n(A):=\wt{\Delta}_{n+1}(A)\cap K_n(A),\qquad  \xi_n(A):=\wt{\xi}_{n+1}(A)|_{K_n(A)},
\end{equation}
and
\begin{equation}\label{nakmap}
\quad
\begin{matrix}
\scriptstyle H^{m}(A;R) & \scriptstyle  \overset{\mu_m}{\lra} & \scriptstyle H^m(K_n(A);R) \\
\scriptstyle \alpha & \scriptstyle \mapsto & \scriptstyle \wt{\mu}_m(\alpha)|_{K_n(A)}.
\end{matrix}
\qquad \qquad
\begin{matrix}
\scriptstyle H^{m-2}(A;R) & \scriptstyle \overset{\nu_m}{\lra} & \scriptstyle H^m(K_n(A);R) \\
\scriptstyle \beta & \scriptstyle \mapsto & \scriptstyle \wt{\nu}_m(\beta)|_{K_n(A)}.
\end{matrix}
\end{equation}

Now suppose that $n\ge 2$. We have a direct sum decomposition
\begin{equation}\label{accaduekum}
H^2(K_n(A);\ZZ)=\mu_2(H^2(A;\ZZ))\oplus \ZZ\xi_n.
\end{equation}
Moreover, the map $\mu_2$ for $R=\CC$ is a homomorphisms of integral Hodge structures. The Beauville-Bogomolov-Fujiki bilinear form   $(,)$ is given by
\begin{equation}\label{bbfkum}
\scriptstyle
(\mu_2(\alpha)+x\xi_n,\mu_2(\beta)+y\xi_n)=\left(\int_{A}\alpha\wedge\beta\right)-2(n+1)xy,\quad \alpha,\beta\in H^2(A),\ x,y\in\CC,
\end{equation}
and the normalized Fujiki constant of $K_n(A)$
equals $n+1$, i.e.
\begin{equation}\label{duenne}
\int\limits_{K_n(A)}\alpha^{2n}=(n+1)(2n-1)!!(\alpha,\alpha)^n\quad \forall\alpha\in H^2(K_n(A);\CC).
\end{equation}
\begin{rmk}\label{rmk:fujipol}
Let $W$ be a complex vector space, equipped with a bilinear symmetric form $(,)$. 
Let us say that two permutations $\sigma,\tau\in\cS_{2r}$ are $\sim$-equivalent if we have equality of  multilinear symmetric functions
\begin{equation}\label{fujipol}
(\alpha_{\sigma(1)},\alpha_{\sigma(2)})\cdot\ldots\cdot(\alpha_{\sigma(2r-1)},\alpha_{\sigma(2r)})=
(\alpha_{\tau(1)},\alpha_{\tau(2)})\cdot\ldots\cdot(\alpha_{\tau(2r-1)},\alpha_{\tau(2r)}),
\end{equation}
Let $\wt{\cS}_{2r}$ be a set of representatives for  $\sim$-equivalence classes, and let $P\colon W^{2r}\to\CC$ be the multilinear symmetric function defined by
\begin{equation}\label{fujipol}
P(\alpha_{1},\ldots,\alpha_{2r}):=\sum\limits_{\sigma\in\wt{\cS}_{2r}}(\alpha_{\sigma(1)},\alpha_{\sigma(2)})\cdot\ldots\cdot(\alpha_{\sigma(2r-1)},\alpha_{\sigma(2r)}).
\end{equation}
Then  $P$ is the polarization of the  homogeneous polynomial $\alpha\mapsto (2r-1)!!(\alpha,\alpha)^r$, i.e.~$P(\alpha,\ldots,\alpha)=(2r-1)!!(\alpha,\alpha)^r$.
In particular, Equation~\eqref{duenne} is equivalent to the equation
\begin{equation}\label{fujipol}
\int\limits_{K_n(A)}\alpha_1\smile\ldots\smile \alpha_{2n}=(n+1)\sum\limits_{\sigma\in\wt{\cS}_{2n}}(\alpha_{i_1},\alpha_{i_2})\cdot\ldots\cdot(\alpha_{i_{2n-1}},\alpha_{i_{2n}}),
\end{equation}
\end{rmk}
Now let $X$ be a $2n$ dimensional hyperk\"ahler manifold, of Kummer type. The bilinear form $(,)$ defines an isomorphism  $H^2(X)\overset{\sim}{\to} H^2(X)^{\vee}$. The inverse  $H^2(X)^{\vee}\overset{\sim}{\to} H^2(X)$ defines an element in $\Sym^2 H^2(X)$, whose image by the cup-product map $\Sym^2 H^2(X)\to H^4(X)$ is a class in $H^{2,2}_{\QQ}(X)$ that we denote by $q^{\vee}_X$, or $q^{\vee}$ if there is no danger of misunderstanding. An explicit expression for $q^{\vee}_{K_n(A)}$ is obtained as follows. Let $e_1,f_1,e_2,f_2,e_3,f_3$ be a \emph{standard basis} of $H^2(A;\ZZ)$, i.e.
\begin{equation}\label{basestand}
\int\limits_{A}e_i^2=0,\quad \int\limits_{A}e_i\smile f_i=1,\quad \text{$\la e_i,f_i\ra$ is orthogonal to  $\la e_j,f_j\ra$ if $i\not=j$.}
\end{equation}
(Notice that each $\la e_i,f_i\ra$ is a hyperbolic plane.)    Then
\begin{equation}\label{qudukum}
q^{\vee}_{K_n(A)}=2\sum_{i=1}^3 \mu_2(e_i)\smile \mu_2(f_i) -\frac{1}{2(n+1)}\xi_n^2.
\end{equation}
Before proving a result on products of $q^{\vee}_{K_n(A)}$, we need an identity whose proof was kindly provided by Ruggero Bandiera. 
\begin{lmm}[Ruggero Bandiera]\label{lmm:ideban}
Let $k$ and $\ell\le n$ be natural numbers. Then
\begin{equation}\label{ideban}
\sum\limits_{i=0}^{\ell} {\ell\choose i}\frac{(2i+2k)!!}{(2k)!!}\frac{(2n-2i-1)!!}{(2n-2\ell-1)!!}=\frac{(2n+2k+1))!!}{(2n-2\ell+2k+1)!!}.
\end{equation}
\end{lmm}
\begin{proof}
For  fixed natural numbers $k,\ell$ the left and right hand sides of~\eqref{ideban}  are polynomials in $n$ (of degree $\ell$), that we denote 
$p_{\ell}^k$ and $q_{\ell}^k$ respectively. In particular $p_{\ell}^k(x)$ and $q_{\ell}^k(x)$ makes sense for any $x$, not only for $x$ an integer greater than $\ell$. One proves that 
$p_{\ell}^k=q_{\ell}^k$ by induction on $\ell$ arguing as follows. First $p_0^k=q_0^k$, because they are both equal to the constant polynomial $1$. A straightforward computation shows that
\begin{equation*}
p_{\ell}^k(n+1)-p_{\ell}^k(n)=2\ell p_{\ell-1}^k(n),\quad q_{\ell}^k(n+1)-q_{\ell}^k(n)=2\ell q_{\ell-1}^k(n),\quad \ell\ge 1,
\end{equation*}
and hence by the inductive hypothesis the difference operators of $p_{\ell}^k$ and of $q_{\ell}^k$ are equal. Since
\begin{equation*}
p_{\ell}^k(n)\left(\frac{2\ell-1}{2}\right)=q_{\ell}^k(n)\left(\frac{2\ell-1}{2}\right),
\end{equation*}
it follows that $p_{\ell}^k=q_{\ell}^k$.
\end{proof}
\begin{prp}\label{prp:bellaform}
Let $X$  be a $2n$ dimensional hyperk\"ahler manifold, of Kummer type. Then for all $\gamma\in H^2(X)$
\begin{equation*}
\int\limits_{[X]}(q^{\vee})^{\ell}\smile \gamma^{2n-2\ell}=(n+1)\frac{(2n+5)!!}{(2n+5-2\ell)!!}(2n-2\ell-1)!! q(\gamma)^{n-\ell}. 
\end{equation*}
\end{prp}
\begin{proof}
By a theorem of Fujiki~\cite{fujiki}, there exists a rational number $C_n^{\ell}$ (independent of $X$) such that 
\begin{equation}\label{propoli}
\int\limits_{[X]}(q^{\vee})^{\ell}\smile \gamma^{2n-2\ell}=C_n^{\ell} \cdot q(\gamma)^{n-\ell}\qquad \forall \gamma\in H^2(X). 
\end{equation}
In order to determine $C_n^{\ell}$, it suffices to compute the left hand side of~\eqref{propoli} for one $X$ and one $\gamma\in H^2(X)$ such that $q(\gamma)\not=0$. We will do the computation for  $X=K_n(A)$ and  $\gamma=\xi_n$. Let 
\begin{equation*}
\sigma_n:=\sum_{i=1}^3 \mu_2(e_i)\smile \mu_2(f_i)\in H^{2,2}_{\ZZ}(K_n(A)). 
\end{equation*}
Thus
\begin{equation*}
(q^{\vee}_{K_n(A)})^{\ell}=\left(2\sigma_n-\frac{1}{2(n+1)}\xi_n^2\right)^{\ell}=\sum_{i=0}^{\ell}{\ell\choose i}2^i\left(-\frac{1}{2(n+1)}\right)^{\ell-i}\sigma_n^i\smile \xi_n^{2(\ell-i)}. 
\end{equation*}
A straightforward computation shows that 
\begin{equation*}
\int_{K_n(A)}\sigma_n^i\smile\xi_n^{2n-2i}=(n+1)\frac{1}{2}i!(i+2)(i+1)(-2(n+1))^{n-i}(2n-2i-1)!!.
\end{equation*}
With some manipulations, it follows that
\begin{equation*}
\int_{K_n(A)}(q^{\vee}_{K_n(A)})^{\ell}\smile\xi_n^{2n-2\ell}=(n+1)q(\xi_n)^{n-\ell}\cdot 
\sum\limits_{i=0}^{\ell} \frac{\ell!}{(\ell-i)!}(i+2)(i+1)2^{i-1}(2n-2i-1)!!.
\end{equation*}
Thus it remains to show that
\begin{equation*}
\sum\limits_{i=0}^{\ell} \frac{\ell!}{(\ell-i)!}(i+2)(i+1)2^{i-1}(2n-2i-1)!!=\frac{(2n+5)!!}{(2n+5-2\ell)!!}(2n-2\ell-1)!! .
\end{equation*}
The above equality follows at once from the case $k=2$ of~\Ref{lmm}{ideban}.
\end{proof}
\begin{dfn}\label{dfn:qubarra}
If $X$ is a $2n$ dimensional hyperk\"ahler manifold  of Kummer type,  let $\ov{q}_X:=2(n+1)q^{\vee}_X$. 
\end{dfn}
The point of the above definition is that  $\ov{q}\in H^{2,2}_{\ZZ}(K_n(A))$ (by~\eqref{qudukum}).

\subsection{On the integral cohomology  of generalized Kummers}\label{subsec:intcohom}
\setcounter{equation}{0}
  We will prove the following two results.
\begin{prp}[Contained in~\cite{cogekum} for $n=2$.]\label{prp:accatre}
Let $\beta\in H^1(A;\ZZ)$. Then $\nu_3(\beta)$ is divisible by $2$ in $H^3(K_n(A);\ZZ)$. 
\end{prp}
\begin{rmk}
Let  $\beta\in H^1(A;\ZZ)$. By~\Ref{prp}{accatre}, there is a well defined $\nu_3(\beta)/2\in H^3(K_n(A);\ZZ)/\Tors$
\end{rmk}
\begin{thm}[Proposition 6.2 in~\cite{cogekum} for $n=2$.]\label{thm:accatre}
The map 
\begin{equation}\label{bronte}
\begin{matrix}
H^3(A)\oplus H^1(A)(-1) & \overset{\mathsf F}{\lra} & H^3(K_n(A))/\Tors \\
(\alpha,\beta) & \mapsto & \mu_3(\alpha)+\nu_3(\beta)/2
\end{matrix}
\end{equation}
is an isomorphism of integral 
Hodge structures. 
\end{thm}
We recall that $A^{(r)}$ is naturally stratified, with strata indexed by partitions of $r$. The stratification of $A^{(r)}$  defines  a stratification of
  $A^{[r]}$ via pull back by the Hilbert-Chow map. Let $\lambda=(\lambda_1,\ldots,\lambda_s)$ be a partition of $r$, where $\lambda_1\ge \lambda_2\ge\ldots\ge \lambda_s$. The stratum $A^{[r]}_{\lambda}$ is equal to the set of $Z$ such that $|Z|=\lambda_1 a_1+\ldots+\lambda_s a_s$, where the points $a_1,\ldots, a_s\in A$ are pairwise distinct. Each stratum is irreducible, and 
\begin{equation}\label{dimstrat}
\dim A^{[r]}_{\lambda}=r+s.
\end{equation}
  Since  
$\{A^{[r]}_{\lambda}\}_{\lambda\in\cP_{r}}$ is a stratification, the dimension formula~\eqref{dimstrat} shows that 
\begin{equation}\label{urvur}
\scriptstyle
U_r:=A^{[r]}_{(1,\ldots,1)}\sqcup A^{[r]}_{(2,1,\ldots,1)},\quad 
V_r:=A^{[r]}_{(1,\ldots,1)}\sqcup A^{[r]}_{(2,1,\ldots,1)}\sqcup A^{[r]}_{(3,1,\ldots,1)}\sqcup A^{[r]}_{(2,2,1,\ldots,1)}
\end{equation}
are open (dense) subsets of $A^{[r]}$.
\begin{lmm}\label{lmm:restisom}
The restriction map $H^3(A^{[r]};\ZZ)\to H^3(U_r;\ZZ)$ is an isomorphism.
\end{lmm}
\begin{proof}
The complement of the open $V_r\subset A^{[r]}$ has (complex) codimension $3$; it follows by a standard argument  that the map $H^3(A^{[r]};\ZZ)\to H^3(V_r;\ZZ)$ is an isomorphism. Thus it suffices to prove that  the map $H^3(V_r,\ZZ)\to H^3(U_r,\ZZ)$ is an isomorphism. (Notice that $V_r\supset U_r$.) If $r\le 2$, then $V_r=U_r$, and hence we are done. From now on we assume that $r\ge 3$. 
A piece of the long exact sequence of cohomology with $\ZZ$ coefficients for the couple $(V_r,U_r)$ reads
\begin{equation*}
H^3(V_r,U_r;\ZZ)\lra H^3(V_r;\ZZ)\lra H^3(U_r;\ZZ)\lra H^4(V_r,U_r;\ZZ)\overset{\rho_r}{\lra} H^4(V_r;\ZZ)
\end{equation*}
By excision and Thom's isomorphism, $H^3(V_r,U_r;\ZZ)=0$, hence $H^3(V_r;\ZZ)\lra H^3(U_r;\ZZ)$ is injective. On the other hand,  excision and Thom's isomorphism give that
$$H^4(V_r,U_r;\ZZ)\cong
\begin{cases}
\ZZ & \text{if $r=3$,} \\
\ZZ^2  & \text{if $r\ge 4$,}
\end{cases}
\quad 
\im\rho_r=
\begin{cases}
\la [A^{[r]}_{(3,1,\ldots,1)}]\ra, & \text{if $r=3$,} \\
\la [A^{[r]}_{(3,1,\ldots,1)}], [A^{[r]}_{(2,2,1,\ldots,1)}]\ra,  & \text{if $r\ge 4$,}
\end{cases}
$$
where $[A^{[r]}_{(3,1,\ldots,1)}]$ and $[A^{[r]}_{(2,2,1,\ldots,1)}]$ are the fundamental classes of $A^{[r]}_{(3,1,\ldots,1)}$ and $A^{[r]}_{(2,2,1,\ldots,1)}$ respectively.  In order to finish the proof it suffices to show that $\rho_r$ is injective, i.e.~that $[A^{[r]}_{(3,1,\ldots,1)}], [A^{[r]}_{(2,2,1,\ldots,1)}]$ are independent over $\ZZ$ (if $r=3$ this is to be interpreted as stating that $[A^{[r]}_{(3,1,\ldots,1)}]$ is not a torsion class). We may assume that $A$ is a special abelian surface; we will assume that $A=E\times F$, where $E,F$ are elliptic curves. Given $q\in F$, we let $i_q\colon E\hra E\times F$ be the embedding as the slice $E\times\{q\}$. Let $D\subset E^{(3)}$ be a generic very ample divisor. Thus $D$ meets the curve $\{3p\mid p\in E\}$ in a finite non empty set. 
 Choose $q\in F$, and  let $x_1,\ldots,x_{r-3}$ be pairwise distinct points of $A\setminus i_{q}(E)$. Let    $\Sigma\subset  A^{[r]}$ be defined by
 $$\Sigma:=\{ i_q(Z)\sqcup \{x_1,\ldots,x_{r-3}\} \mid Z\in D\}.$$
 If $r\ge 4$, choose distinct $q_1,q_2\in F$, and let  $y_1,\ldots,y_{r-4}$ be pairwise distinct points of $A\setminus i_{q_1}(E)\setminus  i_{q_2}(E)$. Let  $\phi\colon E\to\PP^1$ be a degree $2$ map.  
 We let $D_E\subset E^{(2)}$  be the $g^1_2$ defined by $\phi$, i.e.~$D_E:=\{\phi^{*}(p)\mid p\in\PP^1\}$.  Let $\Omega\subset A^{[r]}$ be defined  by
 $$\Omega:=\{ i_{q_1}(W)\sqcup  i_{q_2}(Z)\sqcup \{y_1,\ldots,y_{r-4}\} \mid W,Z\in D_E\}.$$
 Both $\Sigma$ and $\Omega$ are projective, and have  (pure) dimension $2$. Thus we  
   may evaluate the classes $[A^{[r]}_{(2,2,1,\ldots,1)}]$ and $[A^{[r]}_{(3,1,\ldots,1)}]$ on $\Sigma$ and $\Omega$.  Now notice that $\Sigma$ meets $A^{[r]}_{(3,1,\ldots,1)}$ in  a finite non empy set, and that $\Omega$  meets $A^{[r]}_{(2,2,\ldots,1)}$ in a finite non empty set, and it does not  meet $A^{[r]}_{(3,1,\ldots,1)}$. It follows that the $2\times 2$ matrix 
 describing the evaluation of the classes  $[A^{[r]}_{(3,1,\ldots,1)}]$, $[A^{[r]}_{(2,2,1,\ldots,1)}]$  on 
 $\Sigma$ and $\Omega$ is non degenerate.  (If $r=3$, this is to  be interpreted as stating that  the evaluation of the class $[A^{[r]}_{(3,1,\ldots,1)}]$ on  $\Sigma$ is non zero). This proves that $\rho_r$ is injective.
\end{proof}
\begin{proof}[Proof of~\Ref{prp}{accatre}]
It suffices to prove that $\wt{\nu}_3(\beta)\in H^3(A^{[n+1]};\ZZ)$ is divisible by $2$. Let $U_{n+1}\subset A^{[n+1]}$ be the open dense subset defined in~\eqref{urvur}; by~\Ref{lmm}{restisom} it suffices to prove that $\wt{\nu}_3(\beta)|_{U_{n+1}}$ is divisible by $2$ in 
$H^3(U_{n+1};\ZZ)$. Let $\ov{\beta}\in H^1(A;\FF_2)$ be the reduction modulo $2$ of $\beta$; thus  
$\wt{\nu}_3(\ov{\beta})$ is  the reduction modulo $2$ of $\wt{\nu}_3(\beta)$. We must show that 
\begin{equation}\label{modue}
 \wt{\nu}_3(\ov{\beta})|_{U_{n+1}}=0.
\end{equation}
A piece of the long exact sequence of cohomology with $\FF_2$ coefficients for the couple $(U_{n+1},A^{[n+1]}_{(1,1,\ldots,1)})$ reads
\begin{equation}\label{coppiadue}
\scriptstyle
H^2(A^{[n+1]}_{(1,1,\ldots,1)};\FF_2)\overset{\partial}{\lra} H^3(U_{n+1},A^{[n+1]}_{(1,1,\ldots,1)};\FF_2)
\overset{\pi}{\lra} H^3(U_{n+1};\FF_2)\lra H^3(A^{[n+1]}_{(1,1,\ldots,1)};\FF_2)
\end{equation}
Thom's isomorphism gives an identification  
\begin{equation}\label{datreauno}
H^3(U_{n+1},A^{[n+1]}_{(1,1,\ldots,1)};\FF_2)\cong H^1(A^{[n+1]}_{(2,1,\ldots,1)};\FF_2).
\end{equation}
 Let $\tau\colon 
A^{[n+1]}_{(2,1,\ldots,1)}\to S$ be the composition of  $A^{[n+1]}_{(2,1,\ldots,1)}\to S\times S^{(n-1)}$ (the restriction of Hilbert-Chow) and the projection 
$S\times S^{(n-1)}\to S$.  Then
\begin{equation}\label{thompoin}
\pi(\tau^{*}(\ov{\beta}))=\wt{\nu}_3(\ov{\beta})|_{U_{n+1}}.
\end{equation}
 (The above equation makes sense by~\eqref{datreauno}).  By Lemma~3.1 in~\cite{tothilb} (Totaro's Lemma is stated for $n=1$, but the same proof gives the statement in general), $\tau^{*}(\ov{\beta})\in\im(\partial)$, and hence~\eqref{modue} holds.
\end{proof}
\begin{proof}[Proof of~\Ref{thm}{accatre}]
The map in~\eqref{bronte} is a morphism of Hodge structures, integral by~\Ref{prp}{accatre}, hence we are left with the task of proving  that it defines an isomorphism 
between $H^3(A;\ZZ)\oplus H^1(A;\ZZ)$ and $H^3(K_n(A);\ZZ)/\Tors$. We proceed as in the proof of Proposition 6.2 in~\cite{cogekum}.

Let $\{\eta_1,\eta_2,\eta_3,\eta_4\}$ be an \emph{oriented} basis of $H^1(A;\ZZ)$, i.e.~such that $\eta_1\smile \ldots\smile \eta_4$ is the orientation class, and let $\{\eta^{\vee}_1,\ldots,\eta^{\vee}_4\}$ be the dual  basis.
Since we have the perfect pairing
\begin{equation}\label{31duality}
\begin{matrix}
H^3(A)\times  H^1(A) & \overset{\la,\ra}{\lra} & \CC\\
(\alpha,\beta) & \mapsto & \left(\beta\mapsto \int_{A}\alpha\smile\beta\right)
\end{matrix}
\end{equation}
we may view each $\eta_i^{\vee}$ as an element of $H^3(A;\ZZ)$.  Let $\Sigma_1,\ldots,\Sigma_4\subset A$ be generic smooth oriented $1$-manifolds representing the Poincar\'e duals of  $\eta^{\vee}_1,\ldots,\eta^{\vee}_4$, and let  $\Omega_1,\ldots,\Omega_4\subset A$ be generic smooth oriented $2$-manifolds representing the Poincar\'e duals of  $\eta_4\smile \eta_1,\eta_1\smile \eta_2,\eta_2\smile \eta_3,\eta_3\smile \eta_4$.  Choose generic (distinct) points $x_1,\ldots,x_{n-2},y_1,\ldots,y_{n-2}\in A$. Let $\Gamma_1,\ldots,\Gamma_4,\Theta_1,\ldots,\Theta_4\subset K_n(A)$ be the smooth oriented $3$ manifolds 
\begin{eqnarray*}
%
\Gamma_i & := & \{(Z_0\sqcup \{x_1,\ldots,x_{n-2}\})\in K_n(A) \mid Z_0\cap \Sigma_i\not=\es,\ Z_0\cap \Omega_i\not=\es\},\\
\Theta_j & := & \{(Z_0\sqcup \{x_1,\ldots,x_{n-2}\})\in K_n(A) \mid |Z_0|=2p+q,\ \ p\in \Sigma_j\}.
\end{eqnarray*}
A straightforward computation shows that the $8\times 8$ matrix whose entries are the evaluations of the classes $\mu_3(\eta^{\vee}_1),\ldots,\mu_3(\eta^{\vee}_4),\nu_3(\eta_1 /2),\ldots,\nu_3(\eta_4 /2)$ on the $3$-homology classes represented by $\Sigma_1,\ldots,\Sigma_4,\Theta_1,\ldots,\Theta_4$ is 
 a  matrix $\begin{pmatrix}
C & * \\
0_{4,4} & D
\end{pmatrix}$, where $C,D$ are diagonal matrices with entries
 $\pm 1$ on the diagonals. This proves that the image  of $H^3(A;\ZZ)\oplus H^1(A;\ZZ)$ under the map in~\eqref{bronte} is a  rank $8$ saturated subgroup of  $H^3(K_n(A);\ZZ)/\Tors$. By G\"ottsche~\cite{gothesis} the rank of the latter is $8$ , and hence~\Ref{thm}{accatre} follows. 
\end{proof}
\subsection{Structure of $\phi$ for $X$ a generalized Kummer}\label{subsec:struttura}
\setcounter{equation}{0}
Let $X$ be a HK of Kummer type, of dimension $2n$. Let $\gU_X\in H^{2n-4,2n-4}_{\ZZ}(X)$ be an integral  Hodge class which remains of Hodge type for all deformations of $X$. Thus $\gU_X$ might be $\ov{q}^{n-2}$, where $\ov{q}$ is as in~\Ref{dfn}{qubarra}, or a weight $4n-8$ poynomial in the Chern classes of $X$. We let  
\begin{equation}\label{ficonu}
\phi(\gU_X)\colon \bigwedge^2 H^3(X) \lra H^2(X)^{\vee}
\end{equation}
be the composition of the map 
\begin{equation*}
\begin{matrix}
\bigwedge^2 H^3(X) & \lra & H^{4n-2}(X)\\
\gamma\wedge\gamma' & \mapsto & \gamma\smile \gamma'\smile \gU_X
\end{matrix}
\end{equation*}
and the map $H^{4n-2}(X)\to H^2(X)^{\vee}$ defined by cup product. Because of our hypothesis on $\gU_X$, the map $\phi(\gU_X)$ is a morphism of Hodge structures and is flat  for the Gauss-Manin connnection.

Now let $A$ be an abelian surface, and let $\gU=\gU_{K_n(A)}$. 
By~\Ref{prp}{accatre} have the isomorphism   $\mathsf\colon H^3(A)\oplus H^1(A)\overset{\sim}{\lra} H^3(K_n(A))$.   We let
\begin{equation*}
 \begin{matrix}
\bigwedge^2(H^3(A)\oplus H^1(A))  & \overset{\Phi(\gU)}{\lra} & H^2(K_n(A))^{\vee} \\
(\alpha,\beta)\wedge (\alpha',\beta') & \mapsto & \phi(\gU)({\mathsf F}( \alpha,\beta)\wedge {\mathsf F}(\alpha',\beta'))
\end{matrix}
\end{equation*}
We will describe the general structure of $\Phi(\gU)$.
\begin{ntn}\label{ntn:xiduale}
The projection  $H^2(K_n(A))\to H^2(A)$ defined by~\eqref{accaduekum}, defines an embedding $H^2(A)^{\vee}\hra H^2(K_n(A))^{\vee}$, and similarly we may define $\xi_n^{\vee}\in  H^2(K_n(A))^{\vee}$ to be  the function which is zero on $\mu_2(H^2(A))$, and takes the value $1$ on $\xi_n$.  
\end{ntn}
The codomain of $\Phi(\gU)$ is identified with $H^2(A)^{\vee}\oplus\CC\xi_n^{\vee}$. On the other hand  $H^2(A)$ is naturally identified with $\bigwedge^2 H^1(A)$, hence   $H^2(A)^{\vee}$ is naturally identified with $\bigwedge^2 H^1(A)^{\vee}$.  We also have the isomorphism $\bigwedge^2 \lambda\colon \bigwedge^2 H^1(A)^{\vee}\overset{\sim}{\to}\bigwedge^2 H^3(A)$
(see~\eqref{31duality}), hence we may write 
\begin{equation}\label{chiave}
\Phi(\gU)\colon \bigwedge^2(H^3(A)\oplus H^1(A)) \lra  \bigwedge^2 H^3(A)\oplus\CC\xi_n^{\vee}.
\end{equation}
\begin{dfn}
Let $\iota\colon \bigwedge^2 H^1(A)\overset{\sim}{\lra} \bigwedge^2 H^3(A)$ be the composition
\begin{equation}\label{stranacoppia}
 \bigwedge^2 H^1(A)\overset{\sim}{\lra} \bigwedge^2 H^1(A)^{\vee}\overset{\bigwedge^2\lambda}{\lra} \bigwedge^2 H^3(A),
\end{equation}
where  the first map 
is defined by wedge-product $\bigwedge^2 H^1(A)\times \bigwedge^2 H^1(A)  \lra  \CC$.
\end{dfn}
\begin{prp}\label{prp:eccotheta}
There exists $\vartheta(\gU)=(\vartheta_1(\gU),\vartheta_2(\gU),\vartheta_3(\gU))\in\ZZ^3$ such that 
\begin{equation}\label{espansione}
\scriptstyle
\Phi(\gU)((\alpha,\beta)\wedge(\alpha',\beta'))=\vartheta_1(\gU)\alpha\wedge \alpha'+\vartheta_2(\gU)\iota( \beta\wedge \beta')+\vartheta_3(\gU)(\la\alpha,\beta'\ra-\la \alpha',\beta\ra))\xi_n^{\vee}
\end{equation}
for all $(\alpha,\beta),(\alpha',\beta')\in H^3(A)\oplus H^1(A)$. (The expressions $\la\alpha',\beta\ra$, $\la\beta',\alpha\ra$ in~\eqref{espansione} are given by the perfect pairing in~\eqref{31duality}.)
\end{prp}
\begin{proof}
Let $\Phi_i(\gU)$ be the restriction of $\Phi(\gU)$ to the $i$-th summand of the direct sum decomposition
\begin{equation*}
%
\bigwedge^2(H^3(A)\oplus H^1(A))= \bigwedge^2 H^3(A)\oplus  \bigwedge^2 H^1(A)\oplus
  H^3(A)\otimes H^1(A).
\end{equation*}
Each of the maps $\Phi_i(\gU)$ is equivariant for the natural action of the monodromy group of $2$ dimensional compact complex tori  on domain and codomain. We stress that we deform $A$ to arbitrary $2$ dimensional compact complex tori, in general not projective; this makes sense because the generalized Kummer $K_n(T)$ is well-defined for an arbitrary $2$ dimensional compact complex torus. Since the images of the monodromy group  in $H^1(A;\ZZ)$ and $H^3(A;\ZZ)$ are the full (integral) special linear groups, each of the maps $\Phi_i(\gU)$ is equivariant for the natural actions of the (complex) groups 
 $\SL(H^1(A))$ and $\SL(H^3(A))$.
It follows that there exist $\vartheta_1(\gU),\vartheta_2(\gU),\vartheta_3(\gU)\in\CC$ such that $\Phi_1(\gU)=(\vartheta_1(\gU)\Id,0)$,  $\Phi_2(\gU)=(\vartheta_2(\gU)\iota,0)$,  and $\Phi_3(\gU)((\alpha,0)\wedge (0,\beta'))=\vartheta_3(\gU)\la \alpha,\beta'\ra\xi_n^{\vee}$. Since $\Phi(\gU)$ is integral, one gets that each $\vartheta_i(\gU)$ is an integer.
\end{proof}
\begin{dfn}\label{dfn:eccotheta}
Let $X$ be a HK manifold of Kummer type (of dimension  $2n\ge 4$), and let $\gU_X\in H^{2n-4,2n-4}_{\ZZ}(X)$ be an integral  Hodge class which remains of Hodge type for all deformations of $X$. Deforming $(X,\gU_X)$ to  $(K_n(A),\gU_{K_n(A)})$, we may set (unambiguously)  $\vartheta(\gU_X):=\vartheta(\gU_{K_n(A)})$, where $\vartheta(\gU_{K_n(A)})$ is   the triple of  integers defined in~\Ref{prp}{eccotheta}. 
\end{dfn}
\subsection{The cohomology ring of $A^{[m]}$.}\label{subsec:hilbring}
\setcounter{equation}{0}
Let $S$ be a smooth projective surface with torsion canonical class.  Lehn and Sorger~\cite{mancrist} have identified the cohomology \emph{ring}  of $S^{[m]}$ with a ring  
functorially associated to $H(S)$, the cohomology ring of $S$.  In the present subsection we will recall the construction of Lehn and Sorger for an abelian surface $A$ (there is one semplification, because the Euler characteristic  vanishes). If $Z$ is a  topological space, we let $H(Z)$ be its rational cohomology  ring. 
Throughout this subsection we will adhere  to the notation of~\cite{mancrist}.  In particular, we shift the grading of $H(A)$ by $2$, i.e.~we set
\begin{equation}\label{calodue}
\deg H^p(A):=p-2,
\end{equation}
\subsubsection{The ring $H(A)^{[m]}$}
Let $I$ be a finite set. One sets
\begin{equation}
H(A)^{\otimes I}:=H(A^I).
\end{equation}
Suppose that $I$ has  cardinality $r$. Let $[r]:=\{1,2,\ldots,r\}$. A choice of bijection $f\colon [r]\overset{\sim}{\to} I$ defines an isomorphism $H(A)^{\otimes r}\overset{\sim}{\to} H(A)^{\otimes I}$.   
We define a grading of $H(A)^{\otimes I}$ according to~\eqref{calodue}, i.e.
\begin{equation}\label{caloduerre}
\deg H^{p_{1}}(A)\otimes\ldots\otimes H^{p_{r}}(A)=p_1+\ldots+p_r-2r.
\end{equation}
The degree of a homogeneous element $\alpha\in H(A)^{\otimes I}$ is denoted $|\alpha|$. 
One defines 
\begin{equation*}
\begin{matrix}
\underbrace{H(A)\otimes \ldots\otimes H(A)}_{r} & \overset{T_r}{\lra} & \CC \\
\alpha_1\otimes\ldots\otimes \alpha_r & \mapsto & (-\int_{A}\alpha_1)\cdot\ldots\cdot(-\int_{A}\alpha_r)
\end{matrix}
\end{equation*}
(notice the minus signs), where $\int_{A}\alpha$ is the evaluation of the degree $4$ component of $\alpha$ over the $4$-cycle defined by $A$ with its complex orientation. 
Given a finite set $I$ of cardinality $r$, we may define $T_I\colon H(A)^{\otimes I}\to \CC$ by choosing a bijection $[r]\overset{\sim}{\lra} I$, and $T_I$ is clearly independent of the bijection.
Notice that $T_I$ is a non-degenerate bilinear form.

Let $I,J$ be finite sets, and let $f\colon I\to J$ be  a surjection; by taking the cup-product map  $H(A)^{f^{-1}(j)}\to H(A)$ for every $j\in J$ (see  p.~307 of~\cite{mancrist}), one defines a  map 
\begin{equation}\label{calo}
f^{*}\colon H(A)^{\otimes I}\to H(A)^{\otimes J}. 
\end{equation}
Let 
\begin{equation}\label{aumento}
f_{*}\colon H(A)^{\otimes J}\to H(A)^{\otimes I}
\end{equation}
 be the adjoint of $f^{*}$ with respect to the non degenerate bilinear forms  $T_J$ and $T_I$. In particular, let $\Delta_r\colon H(A)^{\otimes r}\to H(A)$ be multiplication. Then 
 $\Delta_{r,*}\colon H(A)\to H(A)^{\otimes r}$ is the adjoint of multiplication:
 \begin{equation}\label{diagonale}
T(\Delta_{r,*}(\alpha)\cdot \beta_1\otimes\ldots\otimes\beta_r)=-\int_{A}\alpha\smile \beta_1\smile\ldots\smile \beta_r.
\end{equation}
Next, let
\begin{equation}
H(A)\{\cS_m\}:=\bigoplus_{\pi\in\cS_m} H(A)^{\otimes\la\pi\ra\backslash [m]}\cdot\pi.
\end{equation}
Here $ \la\pi\ra\backslash [m]$ is the set of orbits in $[m]$ of the subgroup of $\cS_m$ spanned by $\pi$. If $\zeta\in H(A)^{\otimes\la\pi\ra\backslash [m]}$ is homogeneous, the \emph{degree} of $\zeta\pi$ is defined to be $|\zeta|$.

One defines a 
 multiplication on $H(A)\{\cS_m\}$ proceeding as follows (see Proposition~2.13 of~\cite{mancrist}). Let $\pi,\rho\in\cS_m$. The \emph{graph defect} $g(\pi,\rho)\colon 
 \la\pi,\rho\ra\backslash [m]\to\NN$ is the function (see Lemma 2.7 in~\cite{mancrist}) defined by
 \begin{equation*}
g(\pi,\rho)(B)=\frac{1}{2}\left(|B|+2-|\la\pi\ra\backslash B|-|\la\rho\ra\backslash B|-|\la\pi\rho\ra\backslash B|\right).
\end{equation*}
The surjections $\langle\pi\rangle \backslash [m]\lra \langle\pi,\rho\rangle \backslash [m]$ and $\langle\rho\rangle \backslash [m]\lra \langle\pi,\rho\rangle \backslash [m]$ define maps
\begin{equation*}
f^{\pi,\langle\pi,\rho\rangle}\colon H(A)^{\otimes \langle\pi\rangle \backslash [m]}\to H(A)^{\otimes \langle\pi,\rho\rangle \backslash [m]},\quad 
f^{\rho,\langle\pi,\rho\rangle}\colon H(A)^{\otimes \langle\rho\rangle \backslash [m]}\to H(A)^{\otimes \langle\pi,\rho\rangle \backslash [m]}
\end{equation*}
(see~\eqref{calo}), and the surjection $\langle\pi\rho\rangle \backslash [m]\lra \langle\pi,\rho\rangle \backslash [m]$ defines
\begin{equation*}
f_{\langle\pi,\rho\rangle,\langle\pi\rho\rangle}\colon H(A)^{\otimes \langle\pi,\rho\rangle \backslash [m]}\to H(A)^{\otimes \langle\pi\rho\rangle \backslash [m]},
\end{equation*}
as in~\eqref{aumento}. One defines $\mu_{\pi,\rho}\colon H(A)^{\otimes \langle\pi\rangle \backslash [m]}\otimes H(A)^{\otimes \langle\rho\rangle \backslash [m]}
\to H(A)^{\otimes \langle\pi\rho\rangle \backslash [m]}$ by setting
\begin{equation*}
\mu_{\pi,\rho}(a\otimes b):=
\begin{cases}
f_{_{\langle\pi,\rho\rangle,\langle\pi\rho\rangle}}(f^{\pi,\langle\pi,\rho\rangle}(a)\cdot f^{\rho,\langle\pi,\rho\rangle}(b)) & \text{if $g(\pi,\rho)=0$,}\\
0 & \text{otherwise.}
\end{cases}
\end{equation*}
The multiplication on $H(A)\{\cS_m\}$ is defined by setting
\begin{equation*}
\zeta\pi\cdot \xi\rho:=\mu_{\pi,\rho}(\zeta,\xi)\pi\rho.
\end{equation*}
  The group $\cS_m$ acts on $H(A)\{\cS_m\}$, see p.~310 of~\cite{mancrist}, and one sets
\begin{equation}\label{anenne}
H(A)^{[m]}=\left(H(A)\{\cS_m\}\right)^{\cS_m}.
\end{equation}
The restriction of multiplication to $H(A)^{[m]}$ is graded commutative, and homogeneous of degree $2m$, see Proposition~2.13 and Proposition~2.15 of~\cite{mancrist}.
Let 
\begin{equation}\label{eccogamma}
\bigoplus_{m=0}^{\infty}H(A)^{[m]} \overset{\Gamma}{\lra} \bigoplus_{m=0}^{\infty}H(A^{[m]})=:\HH
\end{equation}
be the isomorphism of  vector spaces defined on p.~318 of~\cite{mancrist}. 
\begin{thm}[Lehn-Sorger~\cite{mancrist}]\label{thm:anello}
 The map in~\eqref{eccogamma} is an isomorphism of graded commutative rings, provided  we define $\deg H^p(A^{[m]}):=p-2m$. 
\end{thm}
\subsubsection{Product of certain elements of  $H(A)\{\cS_m\}$}\label{subsubsec:prodotti}
First we introduce some notation.
Let $\tau\in\cS_m$.  We define a total ordering $\preceq$ on the orbit set $\langle\tau\rangle\backslash[m]$ by setting $I\preceq J$ if $\min(I)\leq\min(J)$. Thus, letting $p$ be the cardinality of  $\langle\tau\rangle\backslash[m]$, we have a preferred isomorphism
\begin{equation}\label{preferito}
H(A)^{\otimes p} \overset{\sim}{\to} H(A)^{\otimes\la  \tau\ra\backslash[m]}.
\end{equation}
\begin{dfn}\label{dfn:notls}
Keep notation as above, and let $\beta_1,\ldots,\beta_p\in H(A)$. We may view $\beta_1\otimes\ldots\otimes\beta_p$ as an element of $H(A)^{\otimes\la  \tau\ra\backslash[m]}$  because of~\eqref{preferito}. This understood, we let
$\beta_1\otimes\ldots\otimes\beta_p\tau$ be the corresponding element of $H(A)\{\cS_m\}$.
\end{dfn}
Given $\alpha\in H(A)$ and $1\le i\le m$, we let
\begin{equation}
p_i^{*}(\alpha):= 1\otimes\ldots\otimes 1\otimes \underset{i}{\alpha}\otimes 1\otimes\ldots\otimes 1\in H(A)^{\otimes m}.
\end{equation}
\begin{dfn}\label{dfn:indiag}
Let $\xi\in H(A)$. For and $1\le i<j\le(n+1)$, let $\Delta^{ij}_{*}(\xi)\in H(A)^{\otimes m}$ be the image of $\Delta_{2,*}(\xi)$ under the  homomorphism $H(A)^{\otimes 2} \to   H(A)^{\otimes m}$ mapping $a\otimes b$ to $p_i^{*}(a)\cdot p_j^{*}(b)$. Similarly, for 
$1\le h<k<l\le m$, let $\Delta^{hkl}_{*}(\xi)\in H(A)^{\otimes m}$ be the image of $\Delta_{3,*}(\xi)$ under the  homomorphism $H(A)^{\otimes m} \to   H(A)^{\otimes m}$ mapping $a\otimes b \otimes c$ to $p_h^{*}(a)\cdot p_k^{*}(b)\cdot p_l^{*}(c)$.
\end{dfn}
Let $1\le i<j\le m$ and $1\le h<k\le m$. Then the following formulae hold  (cf.~Example 2.17 in~\cite{mancrist}):
\begin{equation}\label{citau}
%
(p_i^{*}(\beta)(ij))\cdot
(p^{*}_{h}(\beta')(hk))=
\begin{cases}
 \Delta^{ij}_{*}(\beta\smile\beta') \Id &  \text{if $\{i,j\}=\{h,k\}$,}\\
 p^{*}_{\min\{i,j,h,k\}}(\beta\smile\beta') (ij)\cdot (hk) &  \text{if $|\{i,j\}\cap\{h,k\}|=1$,}\\
 z_{ijhk}(\beta,\beta') (ij)\cdot(hk) & 
 \text{if $\{i,j\}\cap\{h,k\}=\es$,}\\
\end{cases}
\end{equation}
where
\begin{equation}\label{prodtrans}
z_{ijhk}(\beta,\beta'):=
\begin{cases}
p^{*}_{i}{\beta}\cdot p^{*}_{h}{\beta'} &  \text{if $i<h<j$ or $h<i<k$,}\\
p^{*}_{i}{\beta}\cdot p^{*}_{h-1}{\beta'}  &  \text{if $i<j<h$,}\\
 p^{*}_{i-1}{\beta} \cdot p^{*}_{h}{\beta'} &  \text{if $h<k<i$.}\\
\end{cases}
\end{equation}
Let $1\le i<j<k\le m$. Then   (cf.~Example 2.17 in~\cite{mancrist}):
\begin{equation}\label{invcic}
%
(p_i^{*}(\beta)(ijk))\cdot
(p^{*}_{i}(\beta')(kji))=
 \Delta^{ijk}_{*}(\beta\smile\beta') \Id.
\end{equation}
Lastly, let $1\le i<j\le m$ and $1\le h<k\le m$. Let $\beta,\beta',\gamma,\gamma'\in H(A)$, with $\gamma$ and $\beta'$ \emph{homogeneous}. Then 
\begin{equation}\label{invklein}
\scriptstyle 
(p_i^{*}(\beta)(ij))\cdot
(p^{*}_{h}(\beta')(hk)) \cdot
(p^{*}_{i}(\gamma)(ij))\cdot (p_h^{*}(\gamma')(hk))=(-1)^{|\beta'|\cdot |\gamma|}\Delta^{ij}_{*}(\beta\smile\gamma)\cdot \Delta^{hk}_{*}(\beta'\smile\gamma')\Id.
\end{equation}
\subsubsection{Cohomology classes and Grojnowski-Nakajima operators}
Our next task is to describe  elements of $H(A)^{[m]} $ which correspond to  classes in $H(A^{[m]})$ that are relevant for our computations. In order to avoid misunderstandings,  we let $\wt{\mu}^{[m]}_r\colon H^r(A)\to H^r(A^{[m]})$ be the map that was previously denoted by $\wt{\mu}_r$ (we add the superscript 
$[m]$), and similarly we let $\wt{\nu}^{[m]}_r\colon H^{r-2}(A)\to H^r(A^{[m]})$ be the map that was previously denoted by $\wt{\nu}_r$. 

Notice that $\sum_{i=1}^m p_i^{*}(\alpha) \Id\in H(A)^{[m]}$. 
\begin{prp}\label{prp:targamu}
Keep notation as above, and let $\alpha\in H^r(A)$. Then
\begin{equation}
\Gamma\left(\sum_{i=1}^m p_i^{*}(\alpha) \Id\right)=\wt{\mu}^{[m]}_r(\alpha).
\end{equation}
\end{prp}
\begin{proof}
Given $\ell\in\NN$ and $\gamma\in H(A)$, let $\gp_{-\ell}(\gamma)\colon\HH\to\HH$ be the Grojnowski-Nakajima operator, see p.~315 in~\cite{mancrist}. Let ${\bf 1}\in H^0(A^{[0]})$ be the function $\{\es\}\to\CC$ with value $1$ (the \emph{vacuum}). By definition of $\Gamma$, the proposition follows from the easily verified equality
\begin{equation*}
\underbrace{\gp_{-1}(1)\cdot\ldots\cdot\gp_{-1}(1)}_{m-1}\cdot\gp_{-1}(\alpha)\cdot{\bf 1}=(m-1)! \wt{\mu}^{[m]}_r(\alpha)
\end{equation*}
\end{proof}
For $\beta\in H(A)$, let
\begin{equation}
c_m(\beta):=\sum_{1\le i<j\le m} p^{*}_{i}(\beta)(ij).
\end{equation}
(This is the only place where our notation differs from that of~\cite{mancrist}, our $c_m(1)$ is denoted $-\epsilon_{m,2}$, see p.~319 op.~cit.)
Notice that 
$ c_m(\beta)\in H(A)^{[m]}$.
\begin{prp}\label{prp:nufock}
Let $\beta\in H^{r-2}(A)$, and keep notation as above. Then
\begin{equation}
\Gamma(c_m(\beta))=\frac{1}{2}\wt{\nu}^{[m]}_r(\beta).
\end{equation}
\end{prp}
\begin{proof}
 By definition of $\Gamma$, the proposition follows from the  equality
\begin{equation*}
\underbrace{\gp_{-1}(1)\cdot\ldots\cdot\gp_{-1}(1)}_{m-2}\cdot\gp_{-2}(\beta)\cdot{\bf 1}=(m-2)! \wt{\nu}^{[m]}_r(\beta)
\end{equation*}
\end{proof}
Let $\beta,\beta'\in H(A)$. The following formula (which holds by~\eqref{citau}) will be handy: 
\begin{multline}\label{ruedeseine}
c_{m}(\beta)\cdot c_{m}(\beta')=\sum_{1\le i<j\le m}\Delta^{ij}_{*}(\beta\smile\beta')\Id + \\
+\sum_{\substack {\{h,k,l\}\subset\{1,\ldots,m\} \\ |\{h,k,l\}|=3}}
p^{*}_{\min\{h,k,l\}}(\beta\cup\beta')(hkl)+
\sum_{\substack {1\le i<j\le m\\ 1\le h<k\le m  \\ \{i,j\}\cap\{h,k\}=\es}}
z_{ijhk}(\beta,\beta')(ij)(hk).
\end{multline}
(Notice that in the second summation in the right hand side of~\eqref{ruedeseine} every order $3$ cyclic permutation appears $3$ times.)  

Lastly, let $\eta_{A^{[m]}}\in H^{4m}(S^{[m]})$ be the fundamental class; it follows directly from the definition of $\Gamma$ (see in particular the definition of $\Phi$ on p.~311 of~\cite{mancrist}) that
\begin{equation}\label{classefond}
\Gamma(\eta_A^{\otimes m} \Id)=\frac{1}{m!}\eta_{A^{[m]}}.
\end{equation}
\subsection{Computation of $\vartheta_1(\ov{q}^{n-2})$}\label{subsec:grancul}
\setcounter{equation}{0}
Let $X$ be a $2n$ dimensional hyperk\"ahler manifold of Kummer type, and let $\ov{q}_X\in H^{2,2}_{\ZZ}(X)$ be the class  in~\Ref{dfn}{qubarra}. Then 
$\ov{q}_X^{n-2}\in  H^{2n-4,2n-4}_{\ZZ}(X)$, and hence $\vartheta(\ov{q}_X^{n-2})$ is well defined.
 In the present subsection we will prove the following  result.
\begin{prp}\label{prp:grancul}
Let $X$ be a $2n$ dimensional hyperk\"ahler manifold of Kummer type, where $n\ge 2$. Then
\begin{equation}
%
%
\vartheta_1(\ov{q}^{n-2}_X)= -2^{n-2}(n+1)^{n-2}\frac{(2n+3)!!}{7!!}.
\end{equation}
\end{prp}
The proof of~\Ref{prp}{grancul} is given at the end of the  subsection.

We start  by going through some preliminary results. 
Let $H(K_n(A))_{(2)}\subset H(K_n(A))$ be the graded subring generated by $H^2(K_n(A))$. By a Theorem of Verbitsky~\cite{verb-cohom,bog-cohom},  the restriction of the Poincar\'e pairing to $H(K_n(A))_{(2)}$ is perfect, and the kernel of the natural map $\Sym H^2(K_n(A))\to H(K_n(A))_{(2)}$ is generated by all elements $\alpha^{n+1}$, where $q(\alpha)=0$. Now
suppose that  $p\le n$. Then the map $\Sym^p H^2(K_n(A))\to H(K_n(A))^{2p}_{(2)}$
is an isomorphism, and hence we have a direct sum decomposition
$$H^{2p}(K_n(A))=\Sym^p H^2(K_n(A))\oplus  \left(H(K_n(A))_{(2)}^{\bot}\right)^{2p},$$
 where orthogonality is with respect to  the Poincar\'e pairing. 
Let
\begin{equation*}
\Pi_p\colon H^{2p}(K_n(A)) \lra \Sym^p H^2(K_n(A))
\end{equation*}
be the projection. 
\begin{lmm}\label{lmm:cidi}
Let $n\ge 3$. There exist $C_i(n),D_i(n)\in\QQ$ for $i\in\{1,2,3\}$ such that for all $\alpha,\alpha'\in H^3(A)$ and 
$\beta,\beta'\in H^1(A)$, 
\begin{eqnarray*}
%
\Pi_3(\mu_3(\alpha)\smile \mu_3(\alpha')) & = & 
C_1(n) q^{\vee}\smile\mu_2(\iota^{-1}(\alpha\wedge\alpha'))+D_1(n) \mu_2(\iota^{-1}(\alpha\wedge\alpha'))\smile\xi_n^2, \label{cidiuno}\\
\Pi_3(\nu_3(\beta)\smile \nu_3(\beta')) & = & 
C_2(n) q^{\vee}\smile\mu_2(\beta\smile\beta')+D_2(n) \mu_2(\beta\smile\beta')\smile\xi_n^2, \label{cididue}\\
\Pi_3(\mu_3(\alpha)\smile \nu_3(\beta)) & = & 
C_3(n) \left(\int_A\alpha\smile\beta\right) q^{\vee}\smile\xi_n+D_3(n) \left(\int_A\alpha\smile\beta\right) \xi^3_n. \label{ciditre}\\
\end{eqnarray*}
\end{lmm}
\begin{proof}
Let $\Psi_1\colon \bigwedge^2 H^3(A) \to \Sym^3 H^2(K_n(A))$, $\Psi_2\colon \bigwedge^2 H^1(A) \to \Sym^3 H^2(K_n(A))$, and $\Psi_3\colon  H^3(A)\otimes H^1(A) \to \Sym^3 H^2(K_n(A))$ be the linear maps  which have values $\Pi_3(\mu_3(\alpha)\smile \mu_3(\alpha'))$, $ \Pi_3(\nu_3(\beta)\smile \nu_3(\beta'))$ and $\Pi_3(\mu_3(\alpha)\smile \nu_3(\beta))$ on decomposable vectors $\alpha\wedge\alpha'$, $\beta\wedge\beta'$ and $\alpha\otimes\beta$ respectively.
Because of~\eqref{accaduekum}, we write the codomain of $\Psi_i$ as 
\begin{equation}\label{codpsi}
%
 \Sym^3 H^2(A)\oplus  
\left(\Sym^2 H^2(A)\otimes \CC\xi_n\right)
\oplus  \left( H^2(A)\otimes \CC\xi^2_n\right)\oplus\CC\xi_n^3.
\end{equation}
The map $\Psi_i$ is equivariant for the action of the Monodromy group on domain and codomain. Since the monodromy group is 
$\SL H^3(A;\ZZ)$, $\Psi_i$ is equivariant for the action of $\SL H^3(A)$. The domains of $\Psi_1$ and  $\Psi_2$  are irreducible representations
  of $\SL H^3(A)$. 
Decomposing each summand of~\eqref{codpsi} into a direct sum of irreducible   $\SL H^3(A)$ representations, one gets the first two equations. The decomposition into irreducible summands of the domain of $\Psi_3$ is $\End_0(H^3(A))\oplus\CC\Id_{H^3(A)}$. Of these two representations, only the trivial one appears in the decomposition of~\eqref{codpsi}, and the third equation follows.
\end{proof}
Throughout  the present subsection we let  $\{\eta_1,\ldots,\eta_4\}$ be an oriented basis of $H^1(A)$, i.e.~such that $\eta:=\eta_1\smile\ldots\smile\eta_4$ is the fundamental class of $A$. 
\begin{prp}\label{prp:laprima}
Let $\alpha,\alpha'\in H^3(A)$, and $\gamma\in H^2(A)$. If $n\ge 2$, then
\begin{equation}\label{hallo}
\scriptstyle
\int\limits_{K_n(A)}\mu_3(\alpha)\smile \mu_3(\alpha')\smile \mu_2(\gamma)^{2n-3}=
-(2n-3)!!\left(\int_{A}\iota^{-1}(\alpha\wedge\alpha')\smile \gamma\right)\cdot \left(\int_{A}\gamma^2\right)^{n-2}.
\end{equation}
\end{prp}
\begin{proof}
The required computation can be done on $A^{n+1}$ (without appealing to the Lehn-Sorger formulae), because of the following argument. 
Let  
\begin{equation*}
\sigma_{n+1}\colon A^{(n+1)}\to  A,\qquad \wh{\sigma}_{n+1}\colon A^{n+1}\to  A
\end{equation*}
 be the summation maps, and let 
\begin{equation}\label{doppiavu}
W_{n+1}(A):=\sigma_{n+1}^{-1}(0),\qquad \wh{W}_{n+1}(A):=\wh{\sigma}_{n+1}^{-1}(0).
\end{equation}
The restriction of the Hilbert-Chow map to $K_{n}(A)$ is a  map $\gh_n\colon K_n(A)\to W_{n+1}(A)$ of degree $1$ and, for $\lambda\in H^k(A)$, the class  $\mu_{k}(\lambda)$ is equal to $\gh_n^{*}(\lambda^{(n+1)}|_{W_{n+1}(A)})$.  Hence the computation may be done on $W_{n+1}(A)$. On the other hand the natural map $\wh{W}_{n+1}(A)\to W_{n+1}(A)$ has degree $(n+1)!$, and therefore  
 the computation may be done on $\wh{W}_{n+1}(A)$. Lastly, we may compute on $A^{n+1}$, because the relevant classes on $\wh{W}_{n+1}(A)$ are the restrictions
  of classes on $A^{n+1}$.

Let $p_i\colon A^{n+1}\to A$ be the projection to the $i$-th factor. 
The kernel of $\wh{\sigma}_{n+1}$, i.e.~$\wh{W}_{n+1}(A)$, has Poincar\'e dual the class 
\begin{equation}\label{eccomega}
\omega:=\sum\limits_{a=1}^{n+1} p_a^{*}(\eta_1)\cup \sum\limits_{b=1}^{n+1} p_b^{*}(\eta_2)\cup \sum\limits_{c=1}^{n+1} p_c^{*}(\eta_3)\cup \sum\limits_{d=1}^{n+1} p_d^{*}(\eta_4). 
\end{equation}
Thus~\eqref{hallo} is equivalent to the following equality:
\begin{equation}\label{compemma}
\scriptstyle
\int\limits_{[A^{n+1}]}\left(\sum\limits_{r=1}^{n+1}p_r^{*}\alpha\right)\smile \left(\sum\limits_{s=1}^{n+1}p_s^{*}\alpha'\right)\smile \left(\sum\limits_{t=1}^{n+1}p_t^{*}\gamma\right)^{2n-3}
\smile \omega=-(n+1)!\cdot (2n-3)!!\left(\int_{A}\iota^{-1}(\alpha\wedge\alpha')\smile \gamma\right)\cdot \left(\int_{A}\gamma\smile \gamma\right)^{n-2}.
\end{equation}
It suffices to prove that~\eqref{compemma} holds for all choices 
\begin{equation}\label{alfadecomp}
\scriptstyle
\alpha=\eta_{i_1}\cup\eta_{i_2}\cup\eta_{i_3},\quad  \alpha'=\eta_{j_1}\cup\eta_{j_2}\cup\eta_{j_3},\quad 
1\le i_1<i_2<i_3\le 4,\  1\le j_1<j_2<j_3\le 4.
\end{equation}
Let $i_0,j_0$ be such that 
\begin{equation}\label{ijzero}
\{i_0, i_1, i_2, i_3\}=\{1,\ldots,4\},\quad \{j_0, j_1, j_2, j_3\}=\{1,\ldots,4\}. 
\end{equation}
By interchanging $\alpha$ and $\alpha'$, if necessary, we may assume that $i_0<j_0$.
 Let $h_0<k_0$ be such that 
 \begin{equation}\label{hkzero}
 \{i_0,j_0,h_0,k_0\}=\{1,\ldots,4\}. 
\end{equation}
Then
 \begin{equation}\label{bianca}
\iota^{-1}(\alpha\wedge\alpha')=-\eta_{h_0}\smile\eta_{k_0}.
\end{equation}
The integrand in the left hand side of~\eqref{compemma} is the sum of monomials, i.e.~products of the addends of the factors.  Each  non vanishing monomial is equal to
\begin{equation}\label{alex}
\scriptstyle
(-1)^{i_0+j_0}p_r^{*}(\alpha\smile\eta_{i_0})\smile p_s^{*}(\alpha'\smile\eta_{j_0})\smile p_{t_1}^{*}(\gamma^2)\smile \ldots\smile p_{t_{n-2}}^{*}(\gamma^{2})\smile 
p_{t_{n-1}}^{*}(\gamma\smile\eta_{h_0}\smile\eta_{k_0}),
\end{equation}
 where $\{r,s,t_1,\ldots,t_{n-1}\}=\{1,\ldots,n+1\}$.

By~\eqref{bianca}, the integral over $A^{n+1}$ of the  class in~\eqref{alex}  equals 
$$- \left(\int_{A}\iota^{-1}(\alpha\wedge\alpha')\smile \gamma\right)\cdot \left(\int_{A}\gamma^2\right)^{n-2}.$$ 
 Since there are $(n+1)!(2n-3)!!$ such integrals appearing, the proposition follows.
\end{proof}
\begin{prp}\label{prp:spartan}
Let $\alpha,\alpha'\in H^3(A)$, and $\gamma\in H^2(A)$. If $n\ge 3$, then 
\begin{equation}\label{ween}
\scriptstyle
\int\limits_{K_n(A)}\mu_3(\alpha)\smile \mu_3(\alpha')\smile \mu_2(\gamma)^{2n-5}\smile\xi_n^2=
2(n+1)\cdot(2n-5)!! \left(\int_{A}\iota^{-1}(\alpha\wedge\alpha')\smile \gamma\right)\cdot \left(\int_{A}\gamma^2\right)^{n-3}.
\end{equation}
\end{prp}
\begin{proof}
Let $Q$ be the number such that
\begin{equation}\label{eccoqu}
\scriptscriptstyle
c_{n+1}(1)^2\cdot
\left(\sum_{i=1}^{n+1}p_i^{*}(\alpha)\Id\right)\cdot\left(\sum_{i=1}^{n+1}p_i^{*}(\alpha')\Id\right)
\cdot\left(\sum_{i=1}^{n+1}p_i^{*}(\gamma)\Id\right)^{2n-5} 
\cdot\prod\limits_{s=1}^4\left(p_1^{*}(\eta_s)\Id+\ldots+p_{n+1}^{*}(\eta_s)\Id\right)=
Q \eta^{\otimes (n+1)}.
\end{equation}
By the results recalled in~\Ref{subsec}{hilbring}, the integral in the left hand side of~\eqref{ween} is equal to $Q/(n+1)!$. Now consider  Equation~\eqref{ruedeseine} for  $\beta=\beta'=1$, and plug it into the left hand side of~\eqref{eccoqu}: the terms in the right hand side of~\eqref{ruedeseine} which involve non trivial permutations will give zero when multiplied by the other factors, hence we get that the left hand side of~\eqref{eccoqu} is equal to the  sum, for $1\le i<j\le(n+1)$,  of the products obtained by substituting   $c_{n+1}(1)^2$ with $\Delta^{ij}_{*}(1)\Id$ in the left hand side of~\eqref{eccoqu}. 
Since there are $n(n+1)/2$ such terms, and each contributes (by symmetry) the same amount to $Q$, it follows that
\begin{multline}\label{pontoise}
\scriptstyle
\int\limits_{K_n(A)}\mu_3(\alpha)\smile \mu_3(\alpha')\smile \mu_2(\gamma)^{2n-5}\smile\xi_n^2=\\
\scriptstyle
\frac{1}{(n-1)!\cdot 2}\int\limits_{A^{n+1}}\Delta^{n,(n+1)}_{*}(1)\smile 
\left(\sum_{i=1}^{n+1}p_i^{*}(\alpha)\right)\smile
\left(\sum_{i=1}^{n+1}p_i^{*}(\alpha')\right)
\smile\left(\sum_{i=1}^{n+1}p_i^{*}(\gamma)\right)^{2n-5}
\smile\prod\limits_{s=1}^4\left(p_1^{*}(\eta_s)+\ldots+p_{n+1}^{*}(\eta_s)\right).
\end{multline}
 Next, notice that $-\Delta^{n,(n+1)}_{*}(1)$  is the Poincar\'e dual of $\{a\in A^{n+1}\mid a_n=a_{n+1}\}$. Thus, letting  $\nu$ be the cohomology class on $A^n$ given by
\begin{equation}
\scriptstyle
\nu:=\left(p_1^{*}(\gamma)+\ldots +p_{n-1}^{*}(\gamma)+2p_{n}^{*}(\gamma)\right)^{2n-5} 
\smile\prod\limits_{s=1}^4\left(p_1^{*}(\eta_s)+\ldots+p_{n-1}^{*}(\eta_s)+2p_{n}^{*}(\eta_s)\right),
\end{equation}
  the integral in the right hand side of~\eqref{pontoise} is equal to
\begin{equation}\label{portonovo}
\scriptstyle
-\int\limits_{A^{n}}
\left(p_1^{*}(\alpha)+\ldots+p_{n-1}^{*}(\alpha)+2p_n^{*}(\alpha)\right)\smile
\left(p_1^{*}(\alpha')+\ldots+p_{n-1}^{*}(\alpha')+2p_n^{*}(\alpha')\right)
\smile\nu.
\end{equation}
By~\Ref{prp}{eccotheta} it suffices to prove that~\eqref{ween} holds with one choice of $\alpha,\alpha'$ such that  $\alpha\wedge\alpha'\not=0$. We choose
\begin{equation}
\alpha=\eta_1\smile\eta_2\smile\eta_3,\qquad \alpha'=\eta_1\smile\eta_2\smile\eta_4.
\end{equation}
Notice that
\begin{equation}\label{naige}
\iota^{-1}(\alpha\wedge\alpha')=\eta_1\smile\eta_2.
\end{equation}
The integrand in~\eqref{portonovo} equals
\begin{equation}
\scriptstyle
\sum\limits_{\stackrel{i\not=j}{ 1\le i,j\le (n-1)}}
p_i^{*}(\alpha)\smile p_j^{*}(\alpha')\smile\nu +2\sum\limits_{ i=1}^{n-1}
p_i^{*}(\alpha)\smile p_n^{*}(\alpha')\smile\nu+2\sum\limits_{ j=1}^{n-1}
p_n^{*}(\alpha)\smile p_j^{*}(\alpha')\smile\nu.
\end{equation}
Since $\nu$ is $\cS_n$-invariant, it follows that the 
 integral in~\eqref{portonovo} equals
\begin{equation}\label{tresorelle}
\scriptscriptstyle
(n-1)(n-2)\int_{A^{n}} p_1^{*}(\alpha)\smile p_2^{*}(\alpha')\smile\nu+
2(n-1)\int_{A^{n}} p_1^{*}(\alpha)\smile p_n^{*}(\alpha')\smile\nu+
2(n-1)\int_{A^{n}} p_n^{*}(\alpha)\smile p_1^{*}(\alpha')\smile\nu.
\end{equation}
Expanding $\nu$ as a sum of monomials,  one gets that
\begin{multline}\label{nicoletta}
\scriptscriptstyle
 p_1^{*}(\alpha)\smile p_2^{*}(\alpha')\smile\nu= \\
\scriptscriptstyle
=\sum\limits_{i=3}^{n-1}4(n-3)!(2n-5)!!
 p_1^{*}(\alpha)\smile p_2^{*}(\alpha')\smile p_{3}^{*}(\gamma^2)\smile\ldots\smile p_{i-1}^{*}(\gamma^2)\smile 
p_{i}^{*}(\gamma)\smile  p_{i+1}^{*}(\gamma^2)\smile \ldots\smile p_{n}^{*}(\gamma^2)\smile p_i^{*}(\eta_1)\smile p_i^{*}(\eta_2)\smile 
p_2^{*}(\eta_3)\smile p_1^{*}(\eta_4)+ \\
\scriptscriptstyle
+8(n-3)!(2n-5)!! p_1^{*}(\alpha)\smile p_2^{*}(\alpha')\smile p_{3}^{*}(\gamma^2)\smile\ldots\smile p_{n-1}^{*}(\gamma^2)\smile 
p_{n}^{*}(\gamma)\smile p_n^{*}(\eta_1)\smile p_n^{*}(\eta_2)\smile 
p_2^{*}(\eta_3)\smile p_1^{*}(\eta_4).
\end{multline}
Thus,  recalling~\eqref{naige}, Equation~\eqref{nicoletta} gives
\begin{equation}
\scriptstyle
\int_{A^{n}} p_1^{*}(\alpha)\smile p_2^{*}(\alpha')\smile\nu=-4(n-1)(n-3)!(2n-5)!! \left(\int\limits_A\gamma\smile\iota^{-1}(\alpha\wedge\alpha')\right)\left(\int\limits_A\gamma^2\right)^{n-3}.
\end{equation}
Expanding again $\nu$ as a sum of monomials, one gets that
\begin{multline*}
\scriptscriptstyle
 p_1^{*}(\alpha)\smile p_n^{*}(\alpha')\smile\nu= \\
\scriptscriptstyle
=\sum\limits_{i=2}^{n-1} 2(n-3)!(2n-5)!!
 p_1^{*}(\alpha)\smile p_n^{*}(\alpha')\smile p_{2}^{*}(\gamma^2)\smile\ldots\smile p_{i-1}^{*}(\gamma^2)\smile 
p_{i}^{*}(\gamma)\smile  p_{i+1}^{*}(\gamma^2)\smile \ldots\smile p_{n-1}^{*}(\gamma^2)\smile p_i^{*}(\eta_1)\smile p_i^{*}(\eta_2)\smile 
 p_n^{*}(\eta_3)\smile p_1^{*}(\eta_4).
\end{multline*}
Recalling~\eqref{naige}, it follows that
\begin{equation}\label{amore}
\scriptstyle
\int\limits_{A^{n}} p_1^{*}(\alpha)\smile p_n^{*}(\alpha')\smile\nu= -2(n-2)!(2n-5)!!
\left(\int\limits_A\gamma\smile\iota^{-1}(\alpha\wedge\alpha')\right)\left(\int\limits_A\gamma^2\right)^{n-3}.
\end{equation}
Exchanging $\alpha$ and $\alpha'$ we see that the third integral appearing in~\eqref{tresorelle} is also equal to the right hand side of~\eqref{amore}. 
By~\eqref{tresorelle} it follows 
that the integral in the right hand side of~\eqref{pontoise} is equal (recall the minus sign in~\eqref{portonovo}) to 
\begin{equation*}
4(n+1)(n-1)!(2n-5)!!\left(\int\limits_A\gamma\smile\iota^{-1}(\alpha\wedge\alpha')\right)\left(\int\limits_A\gamma^2\right)^{n-3}.
\end{equation*}
The proposition now follows from~\eqref{pontoise}.
\end{proof}
\begin{crl}\label{crl:ciunodiuno}
Let $n\ge 3$. Then (notation as in~\Ref{lmm}{cidi})
\begin{equation*}
%
C_1(n)=
-\frac{1}{ (n+1)(2n+5)}, \qquad
D_1(n) =0.
\end{equation*}
\end{crl}
\begin{proof}
By~\Ref{lmm}{cidi}, we have 
\begin{multline}
\scriptstyle
\int\limits_{K_n(A)}\mu_3(\alpha)\smile \mu_3(\alpha')\smile \mu_2(\gamma)^{2n-3}=C_1(n) \int\limits_{K_n(A)}
q^{\vee}\smile\mu_2(\iota^{-1}(\alpha\wedge\alpha'))\smile \mu_2(\gamma)^{2n-3}+ \\
\scriptstyle
+D_1(n)  \int\limits_{K_n(A)}\mu_2(\iota^{-1}(\alpha\wedge\alpha'))\smile\xi_n^2\smile \mu_2(\gamma)^{2n-3},
\end{multline}
and
\begin{multline}
\scriptstyle
\int\limits_{K_n(A)}\mu_3(\alpha)\smile \mu_3(\alpha')\smile \mu_2(\gamma)^{2n-5}\smile\xi_n^2=C_1(n) \int\limits_{K_n(A)}
q^{\vee}\smile\mu_2(\iota^{-1}(\alpha\wedge\alpha'))\smile \mu_2(\gamma)^{2n-5}\smile\xi_n^2+ \\
\scriptstyle
+D_1(n) \int\limits_{K_n(A)}\mu_2(\iota^{-1}(\alpha\wedge\alpha'))\smile \mu_2(\gamma)^{2n-5}\smile\xi_n^4.
\end{multline}
Each of the integrals appearing in the right hand side of the above equations may be computed by invoking  the case $\ell=1$ of~\Ref{prp}{bellaform}, or Equation~\eqref{duenne} (see also~\Ref{rmk}{fujipol}). By~\Ref{prp}{laprima} and~\Ref{prp}{spartan}, it follows that  $C_1(n)$ and $D_1(n)$ are the solutions of the system of linear equations
\begin{equation}\label{sistemino}
\begin{array}{rcl}
\scriptstyle -(2n-3)!!  & \scriptstyle  = & \scriptstyle  (n+1)(2n+5)\cdot (2n-3)!!C_1(n)- 2(n+1)^2 \cdot (2n-3)!!  D_1(n),\\
\scriptstyle  2(n+1) \cdot (2n-5)!! & \scriptstyle  = & \scriptstyle  - 2 (n+1)^2(2n+5)\cdot (2n-5)!! C_1(n)+ 12(n+1)^3 \cdot (2n-5)!!D_1(n).
\end{array}
\end{equation}
Solving for $C_1(n)$ and $D_1(n)$ one gets the formulae of the proposition.
\end{proof}
\begin{proof}[Proof of~\Ref{prp}{grancul}]
We must prove that if  $\alpha,\alpha'\in H^3(A)$ and  $\gamma\in H^2(A)$, then
\begin{equation}\label{aldunque}
\scriptstyle
\int\limits_{K_n(A)}\mu_3(\alpha)\smile\mu_3(\alpha')\smile \ov{q}^{n-2}\smile\mu_2(\gamma)=
 -2^{n-2}(n+1)^{n-2}\frac{(2n+3)!!}{7!!}\int\limits_A \iota^{-1}(\alpha\wedge\alpha')\smile\gamma.
\end{equation}
If $n=2$, Equation~\eqref{aldunque} follows directly from~\eqref{hallo}. If $n\ge 3$, one  applies~\Ref{lmm}{cidi}. In fact, one assigns to $C_1(n)$ and $D_1(n)$ the values given by~\Ref{crl}{ciunodiuno}, and then one
 applies~\eqref{duenne} (see also~\Ref{rmk}{fujipol}) and~\Ref{prp}{bellaform} in order to carry out the required computations.
\end{proof}
\subsection{Computation of $\vartheta_2(\ov{q}^{n-2})$}\label{subsec:granculdue}
\setcounter{equation}{0}
 We will prove the following result. 
\begin{prp}\label{prp:duetre}
Let $X$ be a $2n$ dimensional hyperk\"ahler manifold of Kummer type, where $n\ge 2$. Then 
\begin{equation}\label{duetre}
\vartheta_2(\ov{q}^{n-2}_X)=-2^{n-2}(n+1)^{n-1}\frac{(2n+3)!!}{7!!}
\end{equation}
\end{prp}
The proof of~\Ref{prp}{duetre} will be given at the end of this subsection. Throughout the subsection, $\{\eta_1,\ldots,\eta_4\}$ is an oriented basis of $H^1(A)$, i.e.~$\eta:=\eta_1\smile\ldots\smile\eta_4$ is the fundamental class of $A$.
\begin{prp}\label{prp:sironi}
Let $\beta,\beta'\in H^1(A)$ and $\gamma\in H^2(A)$. If $n\ge 2$, then
\begin{equation}\label{betabeta}
\scriptstyle
 \int\limits_{K_n(A)}\nu_3(\beta)\smile \nu_3(\beta')\smile \mu_2(\gamma)^{2n-3} = 
 -4(n+1)(2n-3)!!\left(\int_{A}\beta\smile\beta'\smile \gamma\right)\cdot \left(\int_{A}\gamma^2\right)^{n-2}.
\end{equation}
\end{prp}
\begin{proof}
 Let $M$ be the integer  such that 
\begin{equation}\label{testaccio}
\scriptstyle
c_{n+1}(\beta)\cdot c_{n+1}(\beta')\cdot\left(\sum_{i=1}^{n+1}p_i^{*}(\gamma)\Id\right)^{2n-3}\cdot \prod_{s=1}^4\left(\sum_{j=1}^{n+1}p_j^{*}(\eta_s)\Id\right)=
M\eta^{\otimes(n+1)} \Id.
\end{equation}
By~\Ref{prp}{targamu}, \Ref{prp}{nufock} and~\eqref{classefond}, we have
\begin{equation}\label{anita}
%
 \int\limits_{K_n(A)}\nu_3(\beta)\smile \nu_3(\beta')\smile \mu_2(\gamma)^{2n-3} = 
\frac{4M}{(n+1)!}.
\end{equation}
Let us compute $M$.  By~\eqref{ruedeseine}
\begin{multline}\label{csaventino}
M=\sum_{1\le h<k\le(n+1)}\int\limits_{A^{n+1}}\Delta^{hk}_{*}(\beta\smile\beta')\smile\left(\sum\limits_{i=1}^{n+1}p_i^{*}(\gamma)\right)^{2n-3}\smile
\prod\limits_{s=1}^4\left(\sum\limits_{j=1}^{n+1}p_j^{*}(\eta_s)\right)= \\
\frac{n(n+1)}{2}\int\limits_{A^{n+1}}\Delta^{12}_{*}(\beta\smile\beta')\smile\left(\sum\limits_{i=1}^{n+1}p_i^{*}(\gamma)\right)^{2n-3}\smile
\prod\limits_{s=1}^4\left(\sum\limits_{j=1}^{n+1}p_j^{*}(\eta_s)\right).
\end{multline}
Since 
$$\Delta^{12}_{*}(\beta\smile\beta')=-p_1^{*}(\beta\smile\beta')\smile {\rm P.D.}\{a\in A^{n+1} \mid a_1=a_2\},$$
(here ${\rm P.D.}$ stands for \lq\lq Poincar\'e dual\rq\rq), it follows that
\begin{multline}\label{tennis}
\scriptstyle
 \int\limits_{K_n(A)}\nu_3(\beta)\smile \nu_3(\beta')\smile \mu_2(\gamma)^{2n-3} = \\
\scriptstyle
=-\frac{2}{(n-1)!}\int\limits_{A^n} p_1^{*}(\beta\smile\beta')\smile(2p_1^{*}(\gamma)+p_2^{*}(\gamma)+\ldots+p_n^{*}(\gamma))^{2n-3}
\smile \prod\limits_{s=1}^4 (2p_1^{*}(\eta_s)+p_2^{*}(\eta_s)+\ldots+p_n^{*}(\eta_s)).
\end{multline}
By~\Ref{prp}{eccotheta} it suffices to prove that~\eqref{betabeta} holds with one choice of $\beta,\beta'$ such that  $\beta\smile \beta'\not=0$. We choose
\begin{equation}
\beta=\eta_1,\qquad \beta'=\eta_2.
\end{equation}
The integrand in the right hand side of~\eqref{tennis} is equal to
\begin{multline}\label{palestra}
\scriptscriptstyle
 \sum\limits_{i=2}^n 2(n-2)!(2n-3)!! p_1^{*}(\beta\smile\beta')\smile p_1^{*}(\gamma)\smile p_2^{*}(\gamma^2)\smile\ldots 
 \smile p_{i-1}^{*}(\gamma^2) \smile p_{i+1}^{*}(\gamma^2)\smile\ldots\smile p_{n}^{*}(\gamma^2)\smile p_i^{*}(\eta)+\\
 \scriptscriptstyle
+\sum\limits_{i=2}^n 4(n-2)!(2n-3)!! p_1^{*}(\beta\smile\beta')\smile p_2^{*}(\gamma^2)\smile\ldots 
 \smile p_{i-1}^{*}(\gamma^2)\smile p_i^{*}(\gamma) \smile p_{i+1}^{*}(\gamma^2)\smile\ldots\smile p_{n}^{*}(\gamma^2)
 \smile p_i^{*}(\eta_1\smile \eta_2)\smile p_1^{*}(\eta_3\smile \eta_4)+\\
 \scriptscriptstyle
+\sum\limits_{2\le i<j\le n} 4(n-2)!(2n-3)!! p_1^{*}(\beta\smile\beta')\smile  p_1^{*}(\gamma)\smile  
p_2^{*}(\gamma^2)\smile\ldots 
 \smile p_{i-1}^{*}(\gamma^2)\smile p_i^{*}(\gamma) \smile p_{i+1}^{*}(\gamma^2)\smile\ldots 
  \smile p_{j-1}^{*}(\gamma^2)\smile p_j^{*}(\gamma) \smile p_{j+1}^{*}(\gamma^2)\smile \ldots \smile p_{n}^{*}(\gamma^2)
 \smile \nu_{ij},
\end{multline}
where
\begin{equation}\label{nutella}
\nu_{ij}:=\sum\limits_{\stackrel{1\le a<b\le 4}{\stackrel{1\le c<d\le 4}{\{a,b,c,d\}=\{1,\ldots,4\}}}}
 (-1)^{a+b-1} p_i^{*}(\eta_a\smile\eta_b)\smile p_j^{*}(\eta_c\smile\eta_d).
\end{equation}
Thus
\begin{multline*}
\scriptstyle
\int\limits_{A^n} p_1^{*}(\beta\smile\beta')\smile(2p_1^{*}(\gamma)+p_2^{*}(\gamma)+\ldots+p_n^{*}(\gamma))^{2n-3}
\smile \prod\limits_{s=1}^4 (2p_1^{*}(\eta_s)+p_2^{*}(\eta_s)+\ldots+p_n^{*}(\eta_s))=\\
\scriptstyle
=2(n+1)(n-1)!(2n-3)!!
\left(\int_{A}\beta\smile\beta'\smile \gamma\right)\cdot \left(\int_{A}\gamma\smile \gamma\right)^{n-2},
\end{multline*}
and the proposition follows from~\eqref{tennis}.
\end{proof}
\begin{prp}\label{prp:mario}
Let $\beta,\beta'\in H^1(A)$, and $\gamma\in H^2(A)$. If $n\ge 3$, then
\begin{equation}\label{ossola}
\scriptstyle
\int\limits_{K_n(A)}\nu_3(\beta)\smile \nu_3(\beta')\smile \mu_2(\gamma)^{2n-5}\smile\xi_n^2 =
8(n+1)^2(2n-5)!! \left(\int_{A}\beta\smile\beta'\smile \gamma\right)\cdot \left(\int_{A}\gamma^2\right)^{n-3}.
 \end{equation}
\end{prp}
\begin{proof}
 Let $P$ be the integer  such that 
\begin{equation}\label{cinquanta}
\scriptstyle
c_{n+1}(\beta)\cdot c_{n+1}(\beta')\cdot  c_{n+1}(1)\cdot c_{n+1}(1)\cdot 
\left(\sum_{i=1}^{n+1}p_i^{*}(\gamma)\Id\right)^{2n-5}\cdot\prod_{s=1}^4\left(\sum_{j=1}^{n+1}p_j^{*}(\eta_s)\Id\right)=
P\eta^{\otimes(n+1)} \Id.
\end{equation}
By~\Ref{prp}{targamu}, \Ref{prp}{nufock} and~\eqref{classefond}, we have
\begin{equation}\label{anita}
%
 \int\limits_{K_n(A)}\nu_3(\beta)\smile \nu_3(\beta')\smile \mu_2(\gamma)^{2n-5}\smile\xi_n^2 = 
\frac{4P}{(n+1)!}.
\end{equation}
Let us compute $P$. 
By~\eqref{ruedeseine} and the formulae in~\Ref{subsubsec}{prodotti}, we have
\begin{multline*}
%
c_{n+1}(\beta)\cdot c_{n+1}(\beta')\cdot  c_{n+1}(1)\cdot c_{n+1}(1)=
\sum_{\substack {1\le a<b\le(n+1)\\ 1\le c<d\le(n+1)}}\Delta^{ab}_{*}(\beta\smile\beta')\cdot \Delta^{cd}_{*}(1)\Id+ \\
+18\left(\sum_{1\le h<k<l \le(n+1)}\Delta^{hkl}_{*}(\beta\smile\beta')\Id\right) 
+2\left(\sum_{\substack {1\le r<s \le(n+1)\\ 1\le t<u\le(n+1)\\ \{r,s\}\cap\{t,u\}=\es}}\Delta^{rs}_{*}(\beta)\cdot \Delta^{tu}_{*}(\beta')\Id\right) + \cR,
\end{multline*}
where the remainder $\cR$ is a sum of terms involving non trivial permutations. 

Let $\tau:=\left(\sum_{i=1}^{n+1}p_i^{*}(\gamma)\Id\right)^{2n-5}\cdot\prod_{s=1}^4\left(\sum_{j=1}^{n+1}p_j^{*}(\eta_s)\Id\right)$, where $p_i$ is  projection to the $i$-th factor.
 Since $\tau$ is $\cS_{n+1}$-invariant, 
\begin{multline}\label{sommaint}
\scriptstyle
P=(n+1)n(n-1)\left(\int\limits_{A^{n+1}}\Delta^{12}_{*}(\beta\smile\beta')\smile \Delta^{13}_{*}(1)\smile  \tau\right)
+\frac{1}{4}(n+1)n(n-1)(n-2) \left(\int\limits_{A^{n+1}}\Delta^{12}_{*}(\beta\smile\beta')\smile \Delta^{34}_{*}(1)\smile \tau\right)+ \\
\scriptstyle
+3(n+1)n(n-1)\left(\int\limits_{A^{n+1}} \Delta^{123}_{*}(\beta\smile\beta')\smile \tau\right)
+\frac{1}{2}(n+1)n(n-1)(n-2)\left(\int\limits_{A^{n+1}}\Delta^{12}_{*}(\beta)\smile \Delta^{34}_{*}(\beta')\smile  \tau\right).
\end{multline}
(Notice that  $\Delta^{ab}_{*}(\beta\smile\beta')\cdot \Delta^{ab}_{*}(1)=0$ for dimension reasons.)

Next, notice that $\Delta_{r,*}(1)$ is the Poincar\'e dual of the small diagonal in $A^r$ \emph{multiplied by $(-1)^{r+1}$}. Moreover, if $\gamma\in H(A)$ then $\Delta_{r,*}(\gamma)=p_1^{*}(\gamma)\smile\Delta_{r,*}(1)$. It follows that the integrals in~\eqref{sommaint}  are equal to integrals over the subset $\{(x,x,x,y_1,\ldots,y_{n-2}\}\subset A^{n+1}$, or the subset $\{(x,x,y,y,z_1,\ldots,z_{n-3}\}\subset A^{n+1}$. More precisely, let $\tau_3$ and $\tau_{2,2}$ be the top cohomology classes on $A^{n-1}$ given by
\begin{eqnarray*}
\tau_3 & := & \left(3 p_1^{*}(\gamma)+\sum_{i=2}^{n-1}p_i^{*}(\gamma)\right)^{2n-5}\smile 
 \prod_{s=1}^4\left(3 p_1^{*}(\eta_s)+\sum_{j=2}^{n-1}p_j^{*}(\eta_s)\right), \\
\tau_{2,2} & := &  \left(2 p_1^{*}(\gamma)+2 p_2^{*}(\gamma)+\sum_{i=3}^{n-1}p_i^{*}(\gamma)\right)^{2n-5}\smile
 \prod_{s=1}^4 \left( 2 p_1^{*}(\eta_s)+2 p_2^{*}(\eta_s)+\sum_{j=3}^{n-1}p_j^{*}(\eta_s)\right).
\end{eqnarray*}
Then~\eqref{sommaint} reads
\begin{multline}\label{intdiag}
\scriptstyle
P=(n+1)n(n-1)\left(\int\limits_{A^{n-1}} p_1^{*}(\beta\smile\beta') \smile  \tau_3 \right)
+\frac{1}{4}(n+1)n(n-1)(n-2) \left(\int\limits_{A^{n-1}}  p_1^{*}(\beta\smile\beta')\smile \tau_{2,2}\right)+ \\
\scriptstyle
+3(n+1)n(n-1)\left(\int\limits_{A^{n-1}} p_1^{*}(\beta\smile\beta')\smile  \tau_3 \right)
+\frac{1}{2}(n+1)n(n-1)(n-2)\left(\int\limits_{A^{n-1}} p_1^{*}(\beta)\smile p_2^{*}(\beta')\smile  \tau_{2,2}\right).
\end{multline}
By~\Ref{prp}{eccotheta}, it suffices to prove that~\eqref{ossola} holds for one choice of $\beta,\beta'$ such that $\beta\smile\beta'\not=0$. We let
 \begin{equation}\label{guiccardini}
%
\beta=\eta_{1},\qquad \beta'=\eta_{2}.
\end{equation}
A computation gives that
\begin{multline}\label{palestra}
\scriptscriptstyle
p_1^{*}(\beta\smile\beta') \smile  \tau_3= \sum\limits_{i=2}^{n-1} 3(n-3)!(2n-5)!! p_1^{*}(\beta\smile\beta')\smile p_1^{*}(\gamma)\smile p_2^{*}(\gamma^2)\smile\ldots 
 \smile p_{i-1}^{*}(\gamma^2) \smile p_{i+1}^{*}(\gamma^2)\smile\ldots\smile p_{n-1}^{*}(\gamma^2)\smile p_i^{*}(\eta)+\\
 \scriptscriptstyle
+\sum\limits_{i=2}^{n-1} 9(n-3)!(2n-5)!!  p_1^{*}(\beta\smile\beta')\smile p_2^{*}(\gamma^2)\smile\ldots 
 \smile p_{i-1}^{*}(\gamma^2)\smile p_i^{*}(\gamma) \smile p_{i+1}^{*}(\gamma^2)\smile\ldots\smile p_{n-1}^{*}(\gamma^2)
 \smile p_i^{*}(\eta_1\smile \eta_2)\smile p_1^{*}(\eta_3\smile \eta_4)+\\
 \scriptscriptstyle
+\sum\limits_{2\le i<j\le (n-1)} 6(n-3)!(2n-5)!!  p_1^{*}(\beta\smile\beta'\smile  \gamma)\smile  
p_2^{*}(\gamma^2)\smile\ldots 
 \smile p_{i-1}^{*}(\gamma^2)\smile p_i^{*}(\gamma) \smile p_{i+1}^{*}(\gamma^2)\smile\ldots 
  \smile p_{j-1}^{*}(\gamma^2)\smile p_j^{*}(\gamma) \smile p_{j+1}^{*}(\gamma^2)\smile \ldots \smile p_{n-1}^{*}(\gamma^2)
 \smile \nu_{ij},
\end{multline}
where $\nu_{ij}$ is given by~\eqref{nutella}.
Thus
\begin{equation}\label{tautre}
\scriptstyle
\int\limits_{A^{n-1}} p_1^{*}(\beta\smile\beta') \smile  \tau_3 =3(n+1)(n-2)!(2n-5)!! 
\left(\int_{A}\beta\smile\beta'\smile \gamma\right)\cdot \left(\int_{A}\gamma\smile \gamma\right)^{n-3}.
\end{equation}
Similarly, 
\begin{multline}\label{palestra}
\scriptscriptstyle
p_1^{*}(\beta\smile\beta') \smile  \tau_{2,2}= 32(n-3)! (2n-5)!!   p_1^{*}(\beta\smile\beta')\smile p_1^{*}(\gamma)\smile p_3^{*}(\gamma^2)\smile\ldots \smile p_{n-1}^{*}(\gamma^2)\smile p_2^{*}(\eta)+\\
\scriptscriptstyle
+ \sum\limits_{i=3}^{n-1}  8(n-3)! (2n-5)!!  p_1^{*}(\beta\smile\beta')\smile p_1^{*}(\gamma)\smile 
p_2^{*}(\gamma^2)\smile p_3^{*}(\gamma^2)\smile\ldots
  p_{i-1}^{*}(\gamma^2) \smile p_{i+1}^{*}(\gamma^2)\smile\ldots\smile p_{n-1}^{*}(\gamma^2)\smile p_i^{*}(\eta)+\\ 
  \scriptscriptstyle
+ 32(n-3)! (2n-5)!!   p_1^{*}(\beta\smile\beta')\smile p_2^{*}(\gamma)\smile p_3^{*}(\gamma^2)\smile\ldots \smile p_{n-1}^{*}(\gamma^2)
 \smile p_2^{*}(\eta_1\smile \eta_2)\smile p_1^{*}(\eta_3\smile \eta_4)+\\
  \scriptscriptstyle
+\sum\limits_{i=3}^{n-1}   16(n-3)! (2n-5)!!    p_1^{*}(\beta\smile\beta')\smile p_2^{*}(\gamma^2)\smile  p_3^{*}(\gamma^2)\smile\ldots 
 \smile p_{i-1}^{*}(\gamma^2)\smile p_i^{*}(\gamma) \smile p_{i+1}^{*}(\gamma^2)\smile\ldots\smile p_{n-1}^{*}(\gamma^2)
 \smile p_i^{*}(\eta_1\smile \eta_2)\smile p_1^{*}(\eta_3\smile \eta_4)+\\ 
 \scriptscriptstyle
+\sum\limits_{j=3}^{n-1} 32(n-3)! (2n-5)!!   p_1^{*}(\beta\smile\beta'\smile  \gamma)\smile  
p_2^{*}(\gamma)\smile p_3^{*}(\gamma^2)\smile \ldots 
 \smile p_{j-1}^{*}(\gamma^2)\smile p_j^{*}(\gamma) \smile p_{j+1}^{*}(\gamma^2)\smile\ldots  \smile p_{n-1}^{*}(\gamma^2)
 \smile \nu_{2j}+\\
 \scriptscriptstyle
+\sum\limits_{3\le i<j\le (n-1)}  16(n-3)! (2n-5)!!   p_1^{*}(\beta\smile\beta'\smile  \gamma)\smile  
p_2^{*}(\gamma^2) \smile p_3^{*}(\gamma^2)\smile\ldots 
 \smile p_{i-1}^{*}(\gamma^2)\smile p_i^{*}(\gamma) \smile\\ 
 \scriptscriptstyle
\smile p_{i+1}^{*}(\gamma^2)\smile\ldots 
  \smile p_{j-1}^{*}(\gamma^2)\smile p_j^{*}(\gamma) \smile p_{j+1}^{*}(\gamma^2)\smile \ldots \smile p_{n-1}^{*}(\gamma^2)
 \smile \nu_{ij},
\end{multline}
where $\nu_{ij}$ is given by~\eqref{nutella}.
Thus
\begin{equation}\label{tauduedue}
\scriptstyle
\int\limits_{A^{n-1}} p_1^{*}(\beta\smile\beta') \smile  \tau_{2,2} = 8(n+1)(n-1)(n-3)!(2n-5)!! 
\left(\int_{A}\beta\smile\beta'\smile \gamma\right)\cdot \left(\int_{A}\gamma\smile \gamma\right)^{n-3}.
\end{equation}
Lastly, for $a\not= b\in\{1,\ldots,4\}$, let
\begin{equation*}
\epsilon_{a,b}:=
\begin{cases}
(-1)^{a+b-1} & \text{if $a<b$,} \\
(-1)^{a+b} & \text{if $a>b$.}
\end{cases}
\end{equation*}
and 
\begin{equation*}
\omega_i:=\sum\limits_{\stackrel{\{a,b,c,d\}=\{1,\ldots,4\}}{c<d}} \epsilon_{a,b} \cdot p_1^{*}(\eta_a)\smile p_2^{*}(\eta_b)\smile 
p_i^{*}(\eta_c\smile \eta_{d}).
\end{equation*}
Then
\begin{multline}\label{palestra}
\scriptscriptstyle
p_1^{*}(\beta)\smile p_2^{*}(\beta') \smile  \tau_{2,2}= 32(n-3)! (2n-5)!!  p_1^{*}(\beta)\smile   p_2^{*}(\beta')\smile p_1^{*}(\gamma)\smile p_3^{*}(\gamma^2)\smile\ldots \smile p_{n-1}^{*}(\gamma^2)\smile p_2^{*}(\eta_1)\smile p_1^{*}(\eta_2)\smile p_2^{*}(\eta_3\smile \eta_4)+\\
\scriptscriptstyle
+ 32(n-3)! (2n-5)!!   p_1^{*}(\beta)\smile   p_2^{*}(\beta')\smile p_2^{*}(\gamma)\smile 
p_3^{*}(\gamma^2)\smile\ldots \smile p_{n-1}^{*}(\gamma^2)\smile p_2^{*}(\eta_1)\smile p_1^{*}(\eta_2\smile\eta_3\smile \eta_4)+\\ 
  \scriptscriptstyle
+\sum\limits_{i=3}^{n-1}    32(n-3)! (2n-5)!!     p_1^{*}(\beta)\smile\ p_2^{*}(\beta')\smile  p_1^{*}(\gamma)\smile  p_2^{*}(\gamma)\smile  
p_3^{*}(\gamma^2)\smile\ldots  \smile p_{i-1}^{*}(\gamma^2)\smile p_i^{*}(\gamma) \smile p_{i+1}^{*}(\gamma^2)\smile\ldots\smile 
p_{n-1}^{*}(\gamma^2) \smile \omega_i.\\ 
\end{multline}
Thus
\begin{equation}\label{taudueduebis}
\scriptstyle
\int\limits_{A^{n-1}} p_1^{*}(\beta)\smile p_2^{*}(\beta') \smile  \tau_{2,2} = -16(n+1)(n-3)! (2n-5)!!  
\left(\int_{A}\beta\smile\beta'\smile \gamma\right)\cdot \left(\int_{A}\gamma\smile \gamma\right)^{n-3}.
\end{equation}
Thus  $P=2(n+1)!(n+1)^2(2n-5)!!$ by~\eqref{intdiag}, \eqref{tautre}, \eqref{tauduedue}, \eqref{taudueduebis}, and  the proposition follows from~\eqref{anita}.
\end{proof}
\begin{crl}\label{crl:cidue}
Let $n\ge 3$. Then (notation as in~\Ref{lmm}{cidi})
\begin{equation*}
%
C_2(n)=
-\frac{4}{(2n+5)}, \qquad
D_2(n) =0.
\end{equation*}
\end{crl}
\begin{proof}
By~\Ref{prp}{laprima}, \Ref{prp}{spartan}, \Ref{prp}{sironi}, and~\Ref{prp}{mario} we have
\begin{eqnarray*}
\scriptstyle \int\limits_{K_n(A)}\nu_3(\beta)\smile \nu_3(\beta')\smile \mu_2(\gamma)^{2n-3} & \scriptstyle  = & 
\scriptstyle 4(n+1)\int\limits_{K_n(A)}\mu_3(\alpha)\smile \mu_3(\alpha')\smile \mu_2(\gamma)^{2n-3}, \\
\scriptstyle  \int\limits_{K_n(A)}\nu_3(\beta)\smile \nu_3(\beta')\smile \mu_2(\gamma)^{2n-5}\smile\xi_n^2 & 
\scriptstyle  = &
\scriptstyle  4(n+1) \int\limits_{K_n(A)}\mu_3(\alpha)\smile \mu_3(\alpha')\smile \mu_2(\gamma)^{2n-5}\smile\xi_n^2.
\end{eqnarray*}
Thus,  going through the proof of~\Ref{crl}{ciunodiuno} one gets that $C_2(n)$ and $D_2(n)$ satisfy the system of linear equations obtained from~\eqref{sistemino} by multiplying the left hand terms by $4(n+1)$ and replacing  $C_1(n)$ $D_1(n)$  by  $C_2(n)$, $D_2(n)$ respectively. Hence $C_2(n)=4(n+1)C_1(n)$ and  $D_2(n)=4(n+1)D_1(n)$. Thus the result follows from~\Ref{crl}{ciunodiuno}.
\end{proof}
\begin{proof}[Proof of~\Ref{prp}{duetre}]
We must prove that $\vartheta_2(\ov{q}^{n-2})=(n+1) \vartheta_1(\ov{q}^{n-2})$. This holds because $C_2(n)=4(n+1)C_1(n)$ and $D_2(n)=4(n+1)D_1(n)$. (Recall that ${\mathsf F}(0,\beta)=\nu(\beta)/2$, where $\mathsf F$ is the isomorphism in~\eqref{bronte}.)
\end{proof}
\subsection{Proof of the first main result}\label{subsec:quadtre}
We will prove~\Ref{thm}{primoteor}. 

Let us prove Item~(1), i.e.~surjectivity of $\phi$. Since  $\vartheta_1(\ov{q}^{n-2})$ and $\vartheta_2(\ov{q}^{n-2})$ are non zero, the map $\phi$ is non zero. Let $X_0$ be very general. Then there is no non trvial sub Hodge structure of $H^2(X_0)^{\vee}$, and hence $\phi$ is surjective.  Since $\phi$ is flat for the Gauss-Manin connection, it follows that  $\phi$ is surjective for every $X$.

Before going to Item~(2), we note the following.
\begin{lmm}\label{lmm:niunozero}
Let $X$ be a HK of Kummer type, of dimension $2n$. Then $\vartheta_3(\ov{q}^{n-2})$ is non zero.
\end{lmm}
\begin{proof}
 We may assume that $X=K_n(A)$. Since $\phi$ is surjective, it follows that $\vartheta_3(\ov{q}^{n-2})$ is non zero.
\end{proof}
\begin{rmk}
We will compute  $\vartheta_3(\ov{q}^{n-2})$ up to sign, see~\Ref{crl}{thirdman}.
\end{rmk}
In proving Item~(2), we may assume that $X$ is a generalized Kummer $K_n(A)$, 
 Identify $H^3(A;\ZZ)\oplus H^1(A;\ZZ)$ with $H^3(K_n(A);\ZZ)$   via~\Ref{thm}{accatre}, see~\eqref{bronte}. Identify $H^1(A;\ZZ)$ with $H^3(A;\ZZ)^{\vee}$ via~\eqref{31duality}. Given these identifications, we define the following integral unimodular quadratic form:
\begin{equation}\label{eccobi}
\begin{matrix}
H^3(K_n(A))/\Tors & \overset{{\bf q}_{K_n(A)}}{\lra} & \CC \\
(\alpha,\beta) & \mapsto & 2\beta(\alpha)
\end{matrix}
\end{equation}
 Let $\gamma\in H^3(K_n(A))$. Then, since the components of  $\vartheta(\ov{q}^{n-2})$ are non zero,  
\begin{equation}
\dim\phi(\gamma\wedge H^3(K_n(A))=
\begin{cases}
0 & \text{if $\gamma=0$,} \\
4 & \text{if $\gamma\not=0$ and $q(\gamma)=0$,} \\
7 & \text{if  $q(\gamma)\not=0$.} \\
\end{cases}
\end{equation}
Item~(2) of~\Ref{thm}{primoteor}  follows.

Lastly, we prove Item~(3). Let $0\not=\gamma\in H^{2,1}(X)$. Then $\phi(\gamma\wedge H^{2,1}(X))\subset\Ann(F^1 H^2(X))$, and hence  $\dim\phi(\gamma\wedge H^3(X))\le 5$. Thus $[\gamma]\in{\bf Q}(X)$.

\section{Reconstructing $H^3(X)$ from  $H^2(X)$}\label{sec:alglin}
\subsection{Summary}
\setcounter{equation}{0}
In the present section we will prove~\Ref{thm}{secondoteor}. We will assume that we are given an abstract version of the  map  in~\eqref{espansione}, i.e.~a linear map 
\begin{equation}\label{moltastra}
\bigwedge^2(V_{\CC}\oplus V_{\CC}^{\vee})\overset{\Phi_{\vartheta}}{\lra} \bigwedge^2 V_{\CC}\oplus \CC,
\end{equation}
depending on a choice of  $\vartheta:=(\vartheta_1,\vartheta_2,\vartheta_3)\in \ZZ^3$, where  $V$ is a free $\ZZ$-module of rank $4$. 
Assuming that  all the components of $\vartheta$ are non zero, we determine for which $4$-dimensional subspaces $\Gamma$ of the domain the image $\Phi_{\vartheta}(\bigwedge^2\Gamma)$ is one dimensional. The motivation is the~\Ref{kob}{chiave}. Next, we equip the codomain of 
$\Phi_{\vartheta}$ with a non degenerate quadratic form modelled on the BBF  of generalized Kummers, and we get a corresponding open subset of a quadric, call it $\cD$,  parametrizing weight $2$ Hodge structure of $K3$ type. 
We show that for  $\Gamma$ as above,  $\Phi_{\vartheta}(\bigwedge^2\Gamma)$ is an element of $\cD$, and that conversely every element of $\cD$ comes from a unique $\Gamma$. Thus associated to each point of $\cD$ there is an integral effective weight $1$ Hodge structure, and hence a compact complex torus. In the last subsection we prove~\Ref{thm}{secondoteor}.
\subsection{Set up}\label{subsec:genset}
\setcounter{equation}{0}
Keeping notation as above, let  $\vol\colon \bigwedge^4 V\overset{\sim}{\lra} \ZZ$ be a volume form. Let $(,)$ be the  bilinear symmetric non degenerate form on  $\bigwedge^2 V$ defined by
$(\alpha,\beta):=\vol(\alpha\wedge\beta)$.
We extend bilinearly $(,)$ to $V_{\CC}:=V\otimes_{\ZZ}\CC$, and we denote it by the same symbol. Let
\begin{equation}\label{mappaiota}
\iota\colon \bigwedge^2 V_{\CC}^{\vee}\overset{\sim}{\lra} \bigwedge^2 V_{\CC}
\end{equation}
be the isomorphism  defined by $(,)$.

We define the map $\Phi_{\vartheta}$ in~\eqref{moltastra} to be the one induced   by the bilinear antisymmetric  map
\begin{equation}\label{eccophi}
\begin{matrix}
(V_{\CC}\oplus V_{\CC}^{\vee})\times(V_{\CC}\oplus V_{\CC}^{\vee}) & \lra & \bigwedge^2 V_{\CC}\oplus \CC \\
((v,g),(w,h)) & \mapsto & (\vartheta_1 v\wedge w+\vartheta_2\iota(g\wedge h),\vartheta_3(g(w)-h(v)))
\end{matrix}
\end{equation}
\begin{rmk}\label{rmk:oasibirra}
Let $X$ be a HK manifold of Kummer type of dimension  $2n$. Let $\gU_X\in H^{2n-4,2n-4}_{\ZZ}(X)$ be an integral  Hodge class which remains of Hodge type for all deformations of $X$, e.g.~$\gU_X=\ov{q}_X^{n-2}$. Let  $\vartheta(\gU_X)$ be as in~\Ref{dfn}{eccotheta}. By~\Ref{prp}{eccotheta}, there exist  isomorphisms $H^3(X;\ZZ)\cong V\oplus (V^{\vee}/2)$  and $H^2(X;\ZZ)^{\vee}\cong (\bigwedge^2 V\oplus\ZZ)$, such that $\Phi(\gU_X)$ gets identified 
with $\Phi_{\vartheta(\gU_X)}$.  
\end{rmk}
\begin{ntn}\label{ntn:zetaclass}
Keeping notation as above, we let $\zeta:=(0,1)\in(\bigwedge^2 V_{\CC}\oplus\CC)$.
\end{ntn}
\subsection{A result in linear algebra}\label{subsec:alglin}
\setcounter{equation}{0}
Let $f\colon V_{\CC}\to V_{\CC}^{\vee}$ be a linear  map.  We  let  $\omega_f\in\bigwedge^2 V_{\CC}^{\vee}$ be the antisymmetric form defined by $f$, i.e.
\begin{equation}\label{antiform}
\omega_f(v,w)=\frac{1}{2}\left(\la f(w),v\ra-\la f(v),w\ra\right),
\end{equation}
where $\la,\ra$ denotes the natural pairing between $ V_{\CC}^{\vee}$ and $ V_{\CC}$.

We let $f=f_{+}+f_{-}$ be the decomposition into the sum of a symmetric  and an antisymmetric linear map. Notice that $\omega_{f_{-}}=\omega_f$.

Now assume that $f$ is  antisymmetric. 
 The \emph{Pfaffian} of $f$ is the Pfaffian of (any) matrix associated to $f$ by the choice of a basis $\{v_1,\ldots,v_4\}$ of $V_{\CC}$ of volume $1$, and the dual basis $\{v^{\vee}_1,\ldots,v^{\vee}_4\}$  of $V_{\CC}^{\vee}$. We denote the 
  Pfaffian of $f$ by $\Pf(f)$. If $\omega_f=\sum_{1\le i<j\le 4}a_{ij}v_i^{\vee}\wedge v_j^{\vee}$, then 
\begin{equation}\label{eccopfaff}
\Pf(f)=a_{12}a_{34}-a_{13}a_{24}+a_{14}a_{23}.
\end{equation}
The motivation for proving the proposition below is provided by the~\Ref{kob}{chiave} and~\Ref{rmk}{oasibirra}.
\begin{prp}\label{prp:grafanti}
Keep notation as above, and assume that $\vartheta_1$, $\vartheta_2$ and $\vartheta_3$ are  non zero.
Let $\Gamma\subset (V_{\CC}\oplus V_{\CC}^{\vee})$ be a vector subspace of dimension $4$. 
Then  $\Phi_{\vartheta}(\bigwedge^2\Gamma)$ is a one-dimensional subspace of  $\bigwedge^2 V\oplus \CC$ if and only if one of the following holds:
\begin{enumerate}
\item
There exists an antisymmetric non degenerate  linear  map $f\colon V_{\CC}\to V_{\CC}^{\vee}$ such that 
\begin{equation}\label{prima}
\vartheta_1=\vartheta_2 \cdot \Pf(f),
\end{equation}
and 
\begin{equation}\label{grafico}
\Gamma:=\{(v,f(v))\mid v\in V_{\CC}\}.  
\end{equation}
If this is the case, then
\begin{equation}\label{unodim}
\Phi_{\vartheta}(\bigwedge^2\Gamma)=\Span\{ \vartheta_2\iota(\omega_f)-2\vartheta_3\zeta\}.
\end{equation}
\item
$\Gamma=U\oplus U^{\bot}$, where $U\subset V_{\CC}$ is a $2$ dimensional subspace, and $U^{\bot}\subset V_{\CC}^{\vee}$ is the annihilator of $U$.  If this is the case, then
\begin{equation}\label{facile}
\Phi_{\vartheta}(\bigwedge^2\Gamma)=\bigwedge^2 U=\iota(\bigwedge^2 U^{\bot}).
\end{equation}
\end{enumerate}
\end{prp}
\begin{proof}
\emph{Suppose that~\eqref{grafico} holds.}
Then  decomposable elements of $\bigwedge^2\Gamma$ are given by $(v,f(v))\wedge(w,f(w)$ for $v,w\in V_{\CC}$, and
\begin{equation}
\Phi_{\vartheta}\left((v,f(v))\wedge(w,f(w)\right)=\vartheta_1 v\wedge w+\vartheta_2\iota(f(v)\wedge f(w))-2\vartheta_3 \omega_f(v,w)\zeta.
\end{equation}
It follows that 
\begin{equation}\label{semputil}
\text{if $f_{-}\not=0$, then $\Phi_{\vartheta}(\bigwedge^2\Gamma)$ contains a vector $\alpha+s\zeta$ where $\alpha\in\bigwedge^2 V_{\CC}$ and $s\not=0$.}
\end{equation}
Now, we do a case-by-case analysis - always assuming that~\eqref{grafico} holds.
\begin{enumerate}
\item{\boxed{\text{$f_{-}$ is nondegenerate.}}} 
Let $v,w,a,b\in V_{\CC}$; then
\begin{equation*}
\Pf(f_{-})\vol(v\wedge w\wedge a\wedge b)=\omega_f(v,w)\cdot\omega_f(a,b)-\omega_f(v,a)\cdot\omega_f(w,b)+\omega_f(v,b)\cdot\omega_f(w,a),
\end{equation*}
i.e.
\begin{equation*}
\Pf(f_{-})v\wedge w=\iota\left(\omega_f(v,w)\cdot\omega_f-f_{-}(v)\wedge f_{-}(w)\right).
\end{equation*}
By hypothesis $\Pf(f_{-})\not=0$, hence
\begin{equation}
v\wedge w=\Pf(f_{-})^{-1}\iota\left(\omega_f(v,w)\cdot\omega_f-f_{-}(v)\wedge f_{-}(w)\right).
\end{equation}
It follows  that
\begin{multline}\label{pappa}
\scriptstyle
\Phi_{\vartheta}\left((v,f(v))\wedge(w,f(w))\right)=\\
\scriptstyle=\vartheta_1 \cdot\Pf(f_{-})^{-1}\omega_f(v,w)\iota(\omega_f )+\vartheta_2 \iota(f_{+}(v)\wedge f_{+}(w) +f_{+}(v)\wedge f_{-}(w) +f_{-}(v)\wedge f_{+}(w))+\\
\scriptstyle
+(\vartheta_2-\vartheta_1\cdot\Pf(f_{-})^{-1})\iota(f_{-}(v)\wedge f_{-}(w))-2\vartheta_3\omega_f(v,w)\zeta.
\end{multline}
Next, we distinguish between the two subcases: $f$ antisymmetric or $f$ not antisymmetric.
\begin{enumerate}
\item{\boxed{\text{$f_{+}=0$.}}}
Equation~\eqref{pappa} shows that if~\eqref{grafico} and~\eqref{prima} hold, and $f$ is antisymmetric non degenerate, then 
  $\Phi_{\vartheta}(\bigwedge^2\Gamma)$ is  one-dimensional, given by~\eqref{unodim}. 

Let us  prove that  if  $\Phi_{\vartheta}(\bigwedge^2\Gamma)$ is   one dimensional,  \eqref{grafico} holds, and $f$ is antisymmetric non degenerate,   then~\eqref{prima} holds. Let $v,w\in V_{\CC}$ be linearly independent, and such that $\omega_f(v,w)=0$. Then the right hand side of~\eqref{pappa} is equal to
\begin{equation*}
(\vartheta_2-\vartheta_1\cdot\Pf(f_{-})^{-1})\iota(f_{-}(v)\wedge f_{-}(w)). 
\end{equation*}
Since $f_{-}(v)\wedge f_{-}(w)$ is not zero (recall that $f_{-}$ is non degenerate by hypothesis),  the above vector and the vector in~\eqref{semputil} are linearly dependent only if 
$\vartheta_2-\vartheta_1\cdot\Pf(f_{-})^{-1}=0$, i.e.~\eqref{prima} holds. 
\item{\boxed{\text{$f_{+}\not=0$.}}}
Assume that  $\Phi_{\vartheta}(\bigwedge^2\Gamma)$ is   one dimensional; we will reach a contradiction.
Let $v,w\in V_{\CC}$ be such that $\omega_f(v,w)=0$. By~\eqref{pappa}, $\Phi_{\vartheta}(\bigwedge^2\Gamma)$ contains the vector
\begin{equation}\label{vasari}
\scriptstyle
\vartheta_2\iota((f_{+}(v)\wedge f_{+}(w) +f_{+}(v)\wedge f_{-}(w) +f_{-}(v)\wedge f_{+}(w))+(\vartheta_2-\vartheta_1\cdot\Pf(f_{-})^{-1}) f_{-}(v)\wedge f_{-}(w).
\end{equation}
By~\eqref{semputil}, we get that the vector in~\eqref{vasari} is zero.
Multiplying the vector in~\eqref{vasari}  by $f_{-}(v)$, we get (recall that by hypothesis $\vartheta_2\not=0$)
\begin{equation}
%
f_{+}(v)\wedge f_{-}(v)\wedge f(w) =0.
\end{equation}
We claim that this is a contradiction, i.e.~that there exist $v,w\in V_{\CC}$  such that $\omega_f(v,w)=0$, and  $f_{+}(v), f_{-}(v), f(w)$ are linearly independent. In fact, 
since  $f_{-}$ is non degenerate and $f_{+}$ is non zero,
$f_{+}(v), f_{-}(v)$ are linearly independent for generic $v$. 
Now suppose first that $f$ is non degenerate. Let $v$ be such that $f_{+}(v), f_{-}(v)$ are linearly independent. Let $v^{\bot}\subset V_{\CC}$ be the orthogonal to $v$ with respect to $\omega_f$. Then $v^{\bot}$ is $3$-dimensional, and hence so is $f(v^{\bot})$. Thus there exists $u\in f(v^{\bot})\setminus\Span\{f_{+}(v), f_{-}(v)\}$. Letting $w:=f^{-1}(u)$, we get that  $\omega_f(v,w)=0$, and $f_{+}(v), f_{-}(v), f(w)$ are linearly independent. 

One argues similarly if $\rk f\in\{2,3\}$ (since $f_{-}=(f-f^{t})/2$ is non degenerate, $\rk f\ge 2$). More precisely, if $\rk f=3$, and $\ker f$ is generated by $v_0$, we repeat the argument above, with $v$ a generic vector not in $v_0^{\bot}$. If $\rk f=2$ we let  $v$ be a generic vector in $V_{\CC}$. Then $f_{-}(v)\notin \im f$, the span of $f_{+}(v)$ and $f_{-}(v)$ intersects $\im f$ in a one dimensional space, and $v^{\bot}$ does not contain $\ker f$. In particular $f(v^{\bot})=\im f$, and hence 
if  $w\in v^{\bot}$ is generic, then $f(w)$ is not contained in the span of $f_{+}(v), f_{-}(v)$. 
\end{enumerate}
\item{\boxed{\text{$\rk(f_{-})=2$.}}} 
Suppose that   $\Phi_{\vartheta}(\bigwedge^2\Gamma)$ is one dimensional; we will reach a contradiction. Let $v,w\in V_{\CC}$ be such that  $\omega_f(v,w)=0$; then $\Phi_{\vartheta}((v,f(v))\wedge(w,f(w)))=0$ by~\eqref{semputil}. Let $a,b\in V_{\CC}$; multiplying the first component of $\Phi_{\vartheta}((v,f(v))\wedge(w,f(w)))$ (as given in~\eqref{eccophi})   by $a\wedge b$, we get that
\begin{equation}\label{ernia}
\vartheta_1 \vol(v\wedge w\wedge a\wedge b)+\vartheta_2(\la f(v),a\ra\cdot\la f(w),b\ra-\la f(v),b\ra\cdot\la f(w),a\ra)=0.
\end{equation}
Now let $v\in V_{\CC}$ be a non zero element of the (two dimensional) kernel of $f_{-}$, i.e.~such that $f(v)=f^t(v)$. Then $f^{-1}\Ann(v)$ has dimension at least $3$, hence there exist 
$a,b\in f^{-1}\Ann(v)$ such that $v,a,b$ are linearly independent. Complete  $\{v,a,b\}$ to a basis  $\{v,a,b,w\}$ of $V_{\CC}$.  Then the left hand side of~\eqref{ernia} is non zero. In fact
$\la f(v),a\ra=\la f(a),v\ra=0$ because $f(v)=f^t(v)$ and $a\in f^{-1}\Ann(v)$. Similarly $\la f(v),b\ra=0$. Hence  the left hand side of~\eqref{ernia} is equal to 
$\vartheta_1 \vol(v\wedge w\wedge a\wedge b)$, which is non zero because   $\{v,a,b,w\}$ is a basis of $V_{\CC}$ (and $\vartheta_1\not=0$). On the other hand  $\omega_f(v,w)=0$ because 
$v$ is in the  kernel of $f_{-}$. That is a contradiction, and hence  $\Phi_{\vartheta}(\bigwedge^2\Gamma)$ is not one dimensional. 

\item{\boxed{\text{$f_{-}=0$.}}} 
Then $f$ is symmetric, and hence we may diagonalize $f$. An explicit computation shows that  $\Phi_{\vartheta}(\bigwedge^2\Gamma)$ is not one dimensional. 
\end{enumerate}
\medskip
\n
{\it Suppose that there does not exist a linear  map $f\colon V_{\CC}\to V_{\CC}^{\vee}$ such that~\eqref{grafico} holds.\/} Thus $\Gamma\cap (V_{\CC},0)$ is  non trivial. An easy case-by-case analysis shows that Item~(2) holds. Viceversa, it is clear that if Item~(2) holds, then  $\Phi_{\vartheta}(\bigwedge^2\Gamma)$ is one dimensional, given by~\eqref{facile}. 
\end{proof}
\subsection{From weight $2$ to weight $1$}\label{subsec:twoone}
\setcounter{equation}{0}
Let $m\in\QQ_{+}$. We let $(,)$ be the  bilinear symmetric non degenerate form on  $(\bigwedge^2 V_{\CC}\oplus \CC)\times(\bigwedge^2 V_{\CC}\oplus \CC)$ defined by
\begin{equation}\label{formaest}
(\alpha+x\zeta,\beta+y\zeta):=\vol(\alpha\wedge\beta)-mxy.
\end{equation}
\begin{expl}\label{expl:bbfkum}
Let $A$ be an abelian surface. Let $V=H^1(A;\ZZ)$, and let 
\begin{equation*}
\vol\colon \bigwedge^4 H^1(A;\ZZ)\overset{\sim}{\lra}\ZZ
\end{equation*}
 be defined by the orientation of $A$. Then $\bigwedge^2 V$ is identified with $H^2(A;\ZZ)$, and the bilinear symmetric form on $\bigwedge^2 V_{\CC}\times\bigwedge^2 V_{\CC}$ defined by $(\alpha,\beta)\mapsto\vol(\alpha\wedge\beta)$ is identified with cup product.  By~\eqref{bbfkum}, the  BBF bilinear form on $H^2(K_n(A))\times H^2(K_n(A))$   is identified  with~\eqref{formaest}, if we let 
$\zeta=\xi_n$, and $m=2(n+1)$. 
\end{expl}
\begin{expl}\label{expl:bbfdual}
Let $A$ be an abelian surface. Let $V= H^3(A;\ZZ)\cong H^1(A;\ZZ)^{\vee}$, where the isomorphism is defined by~\eqref{31duality}. 
Let $\zeta=\xi^{\vee}_n$, see~\Ref{ntn}{xiduale}.  Then we have an isomorphism
\begin{equation}\label{giugno}
\bigwedge^2 V_{\CC}\oplus\CC \zeta \overset{\sim}{\lra} H^2(K_n(A))^{\vee}. 
\end{equation}
 Since the  BBF bilinear symmetric form is non degenerate, it defines an isomorphism $H^2(K_n(A))\overset{\sim}{\lra} H^2(K_n(A))^{\vee}$, and hence a (dual BBF) rational bilinear symmetric form on $ H^2(K_n(A))^{\vee}\times H^2(K_n(A))^{\vee}$. 
By~\eqref{bbfkum}, the  dual BBF bilinear symmetric form    is identified  with~\eqref{formaest} if we let 
 $m=\frac{1}{2(n+1)}$.
\end{expl}
Complex conjugation   defines a conjugation operator  on $\bigwedge^2 V_{\CC}\oplus\CC$. Let 
\begin{equation}\label{perdom}
\cD:=\{[\sigma]\in\PP( \bigwedge^2 V_{\CC}\oplus \CC) \mid (\sigma,\sigma)=0,\quad (\sigma,\ov{\sigma})>0\}.
\end{equation}
Then $\cD$ is a connected complex manifold of dimension $5$. We recall that $\cD$ parametrizes integral weight $2$ Hodge structures  
$(\bigwedge^2 V\oplus \ZZ,H^{p,q})$ of $K3$ type as follows. Given $[\sigma]\in\cD$, we let
\begin{equation}\label{pesodue}
H^{2,0}_{[\sigma]}:=[\sigma],\quad H^{1,1}_{[\sigma]}:=\{\sigma,\ov{\sigma}\}^{\bot},\quad H^{0,2}_{[\sigma]}:=[\ov{\sigma}].
\end{equation}
\begin{prp}\label{prp:relazione}
Suppose that $\vartheta_1$, $\vartheta_2$, and $\vartheta_3$ are non zero.  
Let $[\sigma]\in\cD\setminus \zeta^{\bot}$,  and assume that there exists an    
integral (effective) Hodge structure $(V\oplus V^{\vee}, H^{p,q})$  of weight $1$ such that $\Phi_{\vartheta}$ is a morphism of Hodge structures.  Then
\begin{equation}\label{seconda}
\vartheta_1\cdot\vartheta_2 = 2m\vartheta_3^2.
\end{equation}
\end{prp}
\begin{proof}
Since  $\Phi_{\vartheta}$ is a morphism of Hodge structures,  
$\Phi_{\vartheta}(\bigwedge H^{1,0}_{[\sigma]}(\vartheta))\subset H^{2,0}_{[\sigma]}$, and we have equality because $\Phi_{\vartheta}$ is surjective. Since $H^{1,0}_{[\sigma]}(\vartheta)$ is a $4$ dimensional subspace of $V_{\CC}\oplus V^{\vee}_{\CC}$, and  $H^{2,0}_{[\sigma]}$ is $1$ dimensional, we may apply~\Ref{prp}{grafanti}; we get that 
there exists an  antisymmetric non degenerate map  $f\colon V_{\CC}\to V_{\CC}^{\vee}$ 
 such that~\eqref{prima} holds,  
\begin{equation*}
H^{1,0}=\{(v,f(v)) \mid v\in V_{\CC}\}, 
\end{equation*}
and
\begin{equation*}
H^{2,0}_{[\sigma]}=\Phi_{\vartheta}(\bigwedge^2 H^{1,0})=\Span\{ \vartheta_2\iota(\omega_f)-2\vartheta_3\zeta\}.
\end{equation*}
Thus 
\begin{equation}\label{newrong}
0=(\vartheta_2\iota(\omega_f)-2\vartheta_3\zeta,\vartheta_2\iota(\omega_f)-2\vartheta_3\zeta)=\vartheta_2^2\vol(\iota(\omega_f)\wedge \iota(\omega_f))-4m\vartheta_3^2.
\end{equation}
Let $\vol^{\vee}\colon\bigwedge^4 V^{\vee}_{\CC}\overset{\sim}{\lra}\CC$ be the volume form dual to $\vol$.
Since $\vol(\iota(\omega_f)\wedge\iota(\omega_f))=\vol^{\vee}(\omega_f\wedge\omega_f)=2\Pf(f)$, Equation~\eqref{seconda} follows from~\eqref{newrong} and~\eqref{prima}.
\end{proof}
\begin{crl}\label{crl:thirdman}
Let $X$ be a $2n$ dimensional hyperk\"ahler manifold of Kummer type, where $n\ge 2$. Then
\begin{equation}\label{lula}
%
%
\vartheta(\ov{q}^{n-2}_X)=- 2^{n-2}(n+1)^{n-2}\frac{(2n+3)!!}{7!!}(1,(n+1),(-1)^{\epsilon_n} (n+1)),
\end{equation}
for some $\epsilon_n\in\{0,1\}$.
\end{crl}
\begin{proof}
The values of $\vartheta_1(\ov{q}^{n-2}_X)$ and $\vartheta_2(\ov{q}^{n-2}_X)$ are given   by~\Ref{prp}{grancul} and~\Ref{prp}{duetre} respectively. It remains to compute  $\vartheta_3(\ov{q}^{n-2}_X)$ up to sign.
By~\Ref{lmm}{niunozero}, $\vartheta_3(\ov{q}^{n-2}_X)$ is non zero.
By~\Ref{expl}{bbfdual} and~\Ref{prp}{relazione}, we get that
\begin{equation*}
(n+1)\vartheta_1(\ov{q}^{n-2}_X)\cdot \vartheta_2(\ov{q}^{n-2}_X)=\vartheta_3(\ov{q}^{n-2}_X)^2, 
\end{equation*}
and hence $\vartheta_3(\ov{q}^{n-2}_X)=\pm 2^{n-2}(n+1)^{n-1}\frac{(2n+3)!!}{7!!}$.
\end{proof}
\begin{rmk}\label{rmk:menosei}
Let $\alpha\in H^3(A)$ and $\beta\in H^1(A)$. A straightforward computation (similar to the computations carried out to prove~\Ref{prp}{spartan}, \Ref{prp}{duetre} and~\Ref{prp}{mario}, but much simpler) gives that
\begin{equation}
\int_{K_2(A)}\mu_3(\alpha)\smile \nu_3(\beta)\smile\xi_2=-6\int_A\alpha\smile\beta.
\end{equation}
Thus $\vartheta_3(1)=-3$. Equivalently  $\epsilon_2=1$.
\end{rmk}
\begin{thm}\label{thm:grafanti}
Suppose that $\vartheta_1$, $\vartheta_2$, and $\vartheta_3$ are non zero, and that~\eqref{seconda} holds. Let $[\sigma]\in\cD$ (see~\eqref{perdom}). There exists a unique   
integral (effective) Hodge structure 
\begin{equation}\label{pesuno}
(V\oplus V^{\vee}, H_{[\sigma]}^{p,q}(\vartheta))
\end{equation}
 of weight $1$ with the property that $\Phi_{\vartheta}$ is a morphism of integral Hodge structures, and it is described as follows: 
\begin{enumerate}
\item
If $\sigma\notin \zeta^{\bot}$ (see~\Ref{ntn}{zetaclass}), rescale $\sigma$ so that $\sigma=\vartheta_2 \alpha-2\vartheta_3\zeta$, where $\alpha\in\bigwedge^2 V_{\CC}$. Thus 
 $\alpha\wedge\alpha\not=0$ because $\sigma$ is isotropic. Let $f\colon V_{\CC}\to V_{\CC}^{\vee}$ be the  antisymmetric non degenerate map 
 such that $\iota(\omega_f)=\alpha$. Then 
\begin{equation}\label{accauno}
H_{[\sigma]}^{1,0}(\vartheta)=\{(v,f(v)) \mid v\in V_{\CC}\},\qquad H_{[\sigma]}^{0,1}(\vartheta):=\ov{H_{[\sigma]}^{1,0}(\vartheta)}.
\end{equation}
\item
If $\sigma\in \zeta^{\bot}$, i.e.~$\sigma$ is a decomposable element of $\bigwedge^2 V_{\CC}$, then  
\begin{equation}
H_{[\sigma]}^{1,0}(\vartheta)=U\oplus U^{\bot},\qquad H_{[\sigma]}^{0,1}(\vartheta):=\ov{H_{[\sigma]}^{1,0}(\vartheta)},
\end{equation}
where $U\in\Gr(2,V_{\CC})$ is the unique element such that $\bigwedge^2 U=[\sigma]$.
\end{enumerate}
\end{thm}
\begin{proof}
Suppose that there exists an effective integral weight $1$ Hodge structure~\eqref{pesuno} such that $\Phi_{\vartheta}$ is morphism of Hodge structures. Then 
$\Phi_{\vartheta}(\bigwedge H^{1,0}_{[\sigma]}(\vartheta))\subset H^{2,0}_{[\sigma]}$, and we have equality because $\Phi_{\vartheta}$ is surjective. Since $H^{1,0}_{[\sigma]}(\vartheta)$ is a $4$ dimensional subspace of $V_{\CC}\oplus V^{\vee}_{\CC}$, and  $H^{2,0}_{[\sigma]}$ is $1$ dimensional, we may apply~\Ref{prp}{grafanti}; we get that either~(1) or~(2) holds. 

Next, we show that~\eqref{pesuno} does  indeed define an effective integral weight $1$ Hodge structure, and   that $\Phi_{\vartheta}$ is morphism of Hodge structures. If 
$\sigma\in \zeta^{\bot}$, the verification is straightforward. 
Thus we assume that  $\sigma\notin \zeta^{\bot}$.   By definition of $f$, $\sigma=\vartheta_2\iota(\omega_f)-2\vartheta_3\zeta$. The proof of~\Ref{prp}{relazione} gives 
that~\eqref{prima} holds.
 Thus, by~\Ref{prp}{grafanti}, 
 \begin{equation}\label{gorgonzola}
\Phi_{\vartheta}(\bigwedge H^{1,0}_{[\sigma]}(\vartheta))=H^{2,0}_{[\sigma]}
\end{equation}
by~\Ref{prp}{grafanti}. Since $\Phi_{\vartheta}$ is real, it follows that $\Phi_{\vartheta}(\bigwedge H^{0,1}_{[\sigma]}(\vartheta))=H^{0,2}_{[\sigma]}$. 
Let us prove that
\begin{equation}\label{propfit}
%
 \Phi_{\vartheta}(H^{1,0}_{[\sigma]}(\vartheta)\wedge H^{0,1}_{[\sigma]}(\vartheta))\subset H^{1,1}_{[\sigma]}. 
\end{equation}
It suffices to prove that 
\begin{equation}
(\vartheta_2\iota(\omega_f)-2\vartheta_3\zeta)\bot \Phi_{\vartheta}(H^{1,0}_{[\sigma]}(\vartheta)\wedge H^{0,1}_{[\sigma]}(\vartheta)).
\end{equation}
Let $(v,f(v)),(w,f(w))\in H^{1,0}_{[\sigma]}(\vartheta)$. Then
\begin{multline}
 \left(\vartheta_2\iota(\omega_f)-2\vartheta_3\zeta,\Phi_{\vartheta}((v,f(v))\wedge(\ov{w},\ov{f}(w)))\right)= \\
 =\vartheta_2(\vartheta_2\vol^{\vee}(\omega_f\wedge f(v)\wedge \ov{f}(w))-\vartheta_1\omega_{\ov{f}}(v,w)).
\end{multline}
The right hand side vanishes because of the formula (proved by a straightforward computation)
\begin{equation}
\vol^{\vee}(\omega_f\wedge f(a)\wedge \ov{f}(b))  =  \Pf(f)\omega_{\ov{f}}(a,b).
\end{equation}
 It remains to prove   that 
 \begin{equation}\label{coniugato}
 H^{1,0}_{[\sigma]}(\vartheta)\cap (V_{\RR}\oplus V_{\RR}^{\vee})=\{0\}. 
\end{equation}
Suppose the contrary. It follows  that there exists a basis $\{v_1,v_2,v_3,v_4\}$ of $V_{\RR}$ such that $\omega_f=v_1^{\vee}\wedge v_2^{\vee}+c v_3^{\vee}\wedge v_4^{\vee}$. By~\eqref{prima}, $\Pf(f)$ is real, and hence $c\in\RR$. It follows that $(\sigma,\ov{\sigma})=(\sigma,\sigma)=0$, and that contradicts the hypothesis that $[\sigma]\in\cD$. 
\end{proof}
 \begin{expl}\label{expl:francese}
Let  $K_n(A)$ be a generalized Kummer of dimension $2n$. We adopt the identifications of~\Ref{expl}{bbfdual}, and we set 
$\vartheta=\vartheta(\ov{q}^{n-2})$. Let $\Gamma\subset(V_{\CC}\oplus V_{\CC}^{\vee})$ be the graph of a non degenerate linear map such that there exists a HK of Kummer type of dimension $2n$ and a Gauss-Manin parallel transport operator $H^3(X)\overset{\sim}{\lra} 
(V_{\CC}\oplus V_{\CC}^{\vee})$ sending $H^{2,1}(X)$ to $\Gamma$. By~\Ref{kob}{chiave} and~\Ref{prp}{grafanti}, $f$ is skew-symmetric and  
\begin{equation}\label{uguaglio}
\Pf(f)=\frac{\vartheta_1(\ov{q}^{n-2})}{\vartheta_2(\ov{q}^{n-2})}=\frac{1}{(n+1)}.
\end{equation}
Conversely, let $f$ be a  generic skew-symmetric map  as above, such that~\eqref{uguaglio} holds, and let $\Gamma$ be  the graph  of $f$. Then by~\Ref{thm}{grafanti}    there exists a HK of Kummer type of dimension $2n$ and a Gauss-Manin parallel transport operator $H^3(X)\overset{\sim}{\lra} 
(V_{\CC}\oplus V_{\CC}^{\vee})$ sending $H^{2,1}(X)$ to $\Gamma$. 
\end{expl}
\subsection{The compact complex torus associated to a point of $\cD$}
\setcounter{equation}{0}
\begin{dfn}\label{dfn:torosi}
Keep notation and hypotheses as in~\Ref{thm}{grafanti}. We let
\begin{equation}
J_{[\sigma]}(\vartheta):=(V_{\CC}\oplus V_{\CC}^{\vee})/(H_{[\sigma]}^{1,0}(\vartheta)+(V\oplus V^{\vee})).
\end{equation}
Thus $J_{[\sigma]}(\vartheta)$ is a compact complex torus of dimension $4$. 
\end{dfn}
 \begin{expl}\label{expl:jacint}
Set  $m=1/2(n+1)$, and $\vartheta=\vartheta(\ov{q}_X^{n-2})$, where $X$ is a  HK of Kummer type, of dimension $2n$.  
Going through the identifications in~\Ref{expl}{bbfdual}, we may identify $\Ann F^1(X)$ with a point $[\sigma]\in  \cD$.  We recall that the integral cohomology  $H^3(X;\ZZ)$ is identified with  $V\oplus V^{\vee}$, see~\Ref{thm}{accatre}. Thus we have  an isomorphism
\begin{equation}\label{isogenia}
J_{[\sigma]}(\vartheta) \overset{\sim}{\lra} J^3(X).
\end{equation}
\end{expl}
For a very general $[\sigma]\in\cD$ the torus $J_{[\sigma]}(\vartheta)$ is not projective. We will prove that if  there exists a rational class of positive square in the orthogonal $\sigma^{\bot}$, then $J_{[\sigma]}(\vartheta)$ is an abelian variety. 
\begin{dfn}\label{dfn:domacca}
Let    $h\in (\bigwedge^2 V\oplus \ZZ)^{\vee}$ be non zero. We let
\begin{equation*}
\cD_h:=\{[\sigma]\in\cD \mid \la h,\sigma\ra=0\},
\end{equation*}
where $\la,\ra$ is the duality pairing.
\end{dfn}
\begin{dfn}
Since the bilinear  symmetric form defined in~\eqref{formaest} is non degenerate, it defines a rational isomorphism
\begin{equation}\label{identici}
 (\bigwedge^2 V_{\CC}\oplus\CC)^{\vee}\overset{\sim}{\lra} (\bigwedge^2 V_{\CC}\oplus\CC),
\end{equation}
and hence also  a rational  bilinear  symmetric form on $(\bigwedge^2 V_{\CC}\oplus\CC)^{\vee}$, that will be denoted $(,)^{\vee}$. 
\end{dfn}
\begin{rmk}
Let    $h\in (\bigwedge^2 V\oplus \ZZ)^{\vee}$ be a class of (strictly) positive square. Then $\cD_h$ is not connected, in fact
it has two connected   components, interchanged by conjugation. Each connected component of $\cD_h$ is a Type IV bounded symmetric domain.
\end{rmk}
 \begin{prp}\label{prp:poveraitalia}
Let    $h\in (\bigwedge^2 V\oplus \ZZ)^{\vee}$ be a class of strictly positive square. Let  $\sigma\in\cD_h$.   Then
\begin{equation}\label{nonnullo}
\la i\Phi_{\vartheta}(\alpha\wedge\ov{\alpha}),h\ra \not= 0 \quad \forall \alpha\in H_{[\sigma]}^{1,0}(\vartheta)\setminus\{0\}.
\end{equation}
\end{prp}
 \begin{proof}
First of all, notice that $i\Phi_{\vartheta}(\alpha\wedge\ov{\alpha})\in(\bigwedge^2 V_{\RR}\oplus\RR)$.
Suppose that~\eqref{nonnullo} does not hold, and that $\alpha\in H_{[\sigma]}^{1,0}(\vartheta)$ provides a counterexample - we will arrive at a contradiction. Let $\ell\in (\bigwedge^2 V\oplus \QQ)$ be the class corresponding to $h$ via the isomorphism in~\eqref{identici}. The  restriction to  $(\bigwedge^2 V_{\RR}\oplus\RR)$ of the bilinear symmetric form $(,)$ has signature $(3,4)$. The real subspace  $W\subset (\bigwedge^2 V_{\RR}\oplus\RR)$ spanned by $\ell$ and $\{c\sigma+\ov{c\sigma} \mid c\in\CC\}$ is $3$ dimensional and the restriction of $(,)$ to $W$ is positive definite because $(\ell,\ell)>0$. Since $\Phi_{\vartheta}$ is a morphism of Hodge structures, and since ~\eqref{nonnullo} does not hold, $i\Phi_{\vartheta}(\alpha\wedge\ov{\alpha})$ is orthogonal to $W$. It follows that  $(i\Phi_{\vartheta}(\alpha\wedge\ov{\alpha}),i\Phi_{\vartheta}(\alpha\wedge\ov{\alpha}))\le 0$, with equality only if 
$i\Phi_{\vartheta}(\alpha\wedge\ov{\alpha})=0$. Let $\alpha=(v,g)$, where $v\in V_{\CC}$ and $g\in V^{\vee}_{\CC}$. By~\Ref{thm}{grafanti}
\begin{equation}\label{svanisce}
g(v)=0.
\end{equation}
We have
\begin{equation}
\Phi_{\vartheta}(\alpha\wedge\ov{\alpha})= 
\vartheta_1 v\wedge\ov{v}+\vartheta_2\iota(g\wedge\ov{g})+2\vartheta_3i\im g(\ov{v})\zeta.
\end{equation}
(In the above equation $\im g(\ov{v})$ is the imaginary part of  $g(\ov{v})$.)
Thus
\begin{multline}
(i\Phi_{\vartheta}(\alpha\wedge\ov{\alpha}),i\Phi_{\vartheta}(\alpha\wedge\ov{\alpha}))= \\
=-\vol((\vartheta_1 v\wedge\ov{v}+\vartheta_2\iota(g\wedge\ov{g})\wedge(\vartheta_1 v\wedge\ov{v}+\vartheta_2\iota(g\wedge\ov{g}))
-4m\vartheta_3^2(\im g(\ov{v}))^2= \\
=2\vartheta_1\vartheta_2 |g(\ov{v})|^2-4m\vartheta_3^2(\im g(\ov{v}))^2=2\vartheta_1\vartheta_2( |g(\ov{v})|^2-(\im g(\ov{v}))^2).
\end{multline}
(The second-to-last equality follows from~\eqref{svanisce}, the last equality follows from~\eqref{seconda}.)  Since 
$(i\Phi_{\vartheta}(\alpha\wedge\ov{\alpha}),i\Phi_{\vartheta}(\alpha\wedge\ov{\alpha}))\le 0$, with equality only if $\Phi_{\vartheta}(\alpha\wedge\ov{\alpha})=0$, it follows that 
\begin{eqnarray}
g(\ov{v}) & = & 0, \label{umbria}\\
  \vartheta_1 v\wedge\ov{v}+\vartheta_2\iota(g\wedge\ov{g}) & = & 0. \label{marche}
\end{eqnarray}
By~\eqref{coniugato}  one (at least) among $v\wedge \ov{v}$ and $g\wedge \ov{g}$ is non zero.  Since $\vartheta_1$ and $\vartheta_2$ are both non zero, it follows that  
\begin{equation}
 v\wedge \ov{v}\not=0,\qquad g\wedge \ov{g}\not=0.
\end{equation}
Either Item~(1) or Item~(2) of~\Ref{thm}{grafanti} holds. 
We deal separately with the two cases.  

Suppose that Item~(1) holds. 
 There exists a basis  $\{v_1,\ldots,v_4\}$  of $V_{\RR}$ such that $v=v_1+i v_2$ and $\vol(v_1\wedge\ldots\wedge v_4)=1$. Let $A=(a_{ij})$ be the matrix of $f\colon V_{\CC}\to V_{\CC}^{\vee}$ with respect to the bases  $\{v_1,\ldots,v_4\}$ and  $\{v^{\vee}_1,\ldots,v^{\vee}_4\}$. Thus $A^t=-A$. 
The linear function $g$ is equal to $f(v)$. Since $\la f(v),\ov{v}\ra=0$, we have $a_{12}=0$. Thus
\begin{equation}
\Pf(f)=a_{14}a_{23}-a_{13}a_{24}=\frac{\vartheta_1}{\vartheta_2}.
\end{equation}
(The second equality follows from~\eqref{prima}.) 
Next we notice that by~\eqref{seconda}, the inequality $(\vartheta_2 \iota(\omega_f)-2\vartheta_3\zeta,\vartheta_2 \iota(\omega_f)-2\vartheta_3\zeta)>0$ is equivalent to 
\begin{equation}\label{positivo}
(\iota(\omega_f),\ov{\iota(\omega_f)})>\frac{2\vartheta_1}{\vartheta_2}.
\end{equation}
We will write the above inequality in an equivalent form. Straightforward computations give
\begin{multline}\label{conti}
\left| 
\begin{array}{cc}
a_{13}-\ov{a}_{13}  &  a_{14}-\ov{a}_{14}     \\
 a_{23}-\ov{a}_{23}   & a_{24}-\ov{a}_{24}   
\end{array}
\right|= -\Pf(f)-\ov{\Pf(f)}+2\re(a_{14}\ov{a}_{23}-a_{13}\ov{a}_{24}) = \\
=-\frac{2\vartheta_1}{\vartheta_2}+2\re(a_{14}\ov{a}_{23}-a_{13}\ov{a}_{24}) =-\frac{2\vartheta_1}{\vartheta_2}+ (\iota(\omega_f),\ov{\iota(\omega_f)}).
\end{multline}
Let $D$ be the real number such that
\begin{equation*}
4D=
\left| 
\begin{array}{cc}
a_{13}-\ov{a}_{13}  &  a_{14}-\ov{a}_{14}     \\
 a_{23}-\ov{a}_{23}   & a_{24}-\ov{a}_{24}   
\end{array}
\right|
\end{equation*}
By~\eqref{positivo} and~\eqref{conti}, we have  
\begin{equation}\label{negativo}
D>0.
\end{equation}
Straighforward computations give
\begin{eqnarray*}
v\wedge \ov{v} & = & -2i v_1\wedge v_2, \\
  \iota(f(v)\wedge \ov{f(v)}) & = &  2i \left(\re(a_{14}\ov{a}_{23}-a_{13}\ov{a}_{24})+
\im( a_{13}\ov{a}_{14}+a_{23}\ov{a}_{24}) \right) v_1\wedge v_2.
\end{eqnarray*}
Using~\eqref{conti}, we get that~\eqref{marche} holds if and only if
\begin{equation}\label{prepotenti}
2D+\im( a_{13}\ov{a}_{14}+a_{23}\ov{a}_{24})=0.
\end{equation}
Now let
\begin{equation*}
a_{13}=x_1+iy_1,\ \ a_{14}=x_2+iy_2,\ \ a_{23}=x_3+iy_3,\ \ a_{24}=x_4+iy_4,\quad x_k,y_k\in\RR.
\end{equation*}
Writing~\eqref{prepotenti} and the equation $\im \Pf(f)=0$ in terms of $x_1,\ldots,x_4,y_1,\ldots,y_4$, we get that
\begin{eqnarray}
x_1y_2-x_2y_1  & = & -x_3y_4+x_4 y_3+2D,\\
x_1y_4-x_2 y_3  & = & x_3 y_2-x_4 y_1.
\end{eqnarray}
By Cramer's formula 
\begin{eqnarray}\label{albero}
Dx_1 & = & -x_3y_3 y_4+x_4 y^2_3-x_3 y_1 y_2+x_4 y_1^2+2D y_3, \label{albero} \\
Dx_2  & = & -x_3 y_4^2+x_4 y_3 y_4-x_3 y_2^2+x_4 y_1 y_2+2D y_4. \label{preferito}
\end{eqnarray}
Writing out the equation $\re\Pf(f)=\vartheta_1/\vartheta_2$ in terms of  $x_1,\ldots,x_4,y_1,\ldots,y_4$, multiplying it by $D$, replacing $Dx_1$ 
by the  expression in the right hand side of~\eqref{albero}, and   
$Dx_2$ by the  expression in the right hand side of~\eqref{preferito}, we get that
\begin{equation}
-(D-x_3y_4+x_4y_3)^2-(x_3y_2-x_4y_1)^2=D\frac{\vartheta_1}{\vartheta_2}.
\end{equation}
Since $D$ and $\vartheta_1/\vartheta_2$ are strictly positive by~\eqref{negativo} and~\eqref{seconda} respectively, the above equation is absurd.
We have reached a contradiction if  Item~(1)  of~\Ref{thm}{grafanti} holds. 

Now suppose that Item~(2) holds. By~\eqref{coniugato} both $v\wedge\ov{v}$ and $g\wedge\ov{g}$ are non zero. Complete $v$ to a basis $\{v,w\}$ of $U$. The inequality $(\sigma,\ov{\sigma})>0$ translates into
\begin{equation}\label{caldo}
\vol(v\wedge w\wedge\ov{v}\wedge\ov{w})>0.
\end{equation}
Since $g(v)=g(\ov{v})=0$, we have
\begin{equation}
\iota(g\wedge\ov{g})=\lambda v\wedge\ov{v},\quad \lambda\in \RR^{*}.
\end{equation}
 Moreover
\begin{equation*}
\vol(\lambda v\wedge\ov{v} \wedge w\wedge\ov{w})=\la\iota(g\wedge\ov{g},w\wedge\ov{w}\ra=g(w)\cdot \ov{g}(\ov{w})-\ov{g}(w)\cdot g(\ov{w})
=-|g(\ov{w})|^2<0.
\end{equation*}
(Recall that $g\in U^{\bot}$.) By~\eqref{caldo} it follows that $\lambda>0$. This contradicts~\eqref{marche} because $\vartheta_1$ and $\vartheta_2$ have the same sign by~\eqref{seconda} . 
\end{proof}

Let    $h\in (\bigwedge^2 V\oplus \ZZ)^{\vee}$. We  define the following skew-symmetric bilinear form on $V_{\CC}\oplus V_{\CC}^{\vee}$:
\begin{equation}\label{formbil}
\la\alpha,\beta\ra_{\vartheta,h}:=\la h,\Phi_{\vartheta}(\alpha\wedge\beta)\ra.
\end{equation}
Note that $\la,\ra$ in the right hand side of the above equation denotes the duality pairing.

\begin{dfn}\label{dfn:polab}
Let    $h\in (\bigwedge^2 V\oplus \ZZ)^{\vee}$. Let  $\sigma\in\cD_h$. Restricting the  skew-symmetric  bilinear  form in~\eqref{formbil} to $H_{[\sigma]}^{0,1}(\vartheta)\times H_{[\sigma]}^{1,0}(\vartheta)$, i.e.~the product of the  tangent space at the origin of 
$J_{[\sigma]}(\vartheta)$ and its complex conjugate, we get a translation invariant rational $(1,1)$-form on $J_{[\sigma]}(\vartheta)$. We let
\begin{equation}
\Theta_{[\sigma]}(\vartheta)\in H^{1,1}_{\QQ}(J_{[\sigma]}(\vartheta)
\end{equation}
be the corresponding cohomology class. 
\end{dfn}
 \begin{prp}\label{prp:ampio}
Let    $h\in (\bigwedge^2 V\oplus \ZZ)^{\vee}$ be a class of positive square. For one of the two connected components of $\cD_h$, call it  $\cD^{+}_h$, the following holds. 
Let  $[\sigma]\in\cD^{+}_h$; then the cohomology class $\Theta_{[\sigma]}(\vartheta)$ is ample on $J_{[\sigma]}(\vartheta)$.
\end{prp}
 \begin{proof}
Let $\cD^{1}_h,\cD^{2}_h$ be the two connected components of $\cD_h$. The set $\cV^{j}$ of couples $([\sigma],\alpha)$ where  $[\sigma]\in\cD^{j}_h$ and $0\not=\alpha\in H_{[\sigma]}^{0,1}(\vartheta)$ is the complement of the zero section in a  complex vector bundle over the connected space $\cD^{j}_h$. Thus  $\cV^{j}$  is connected. 
By~\Ref{prp}{poveraitalia}, the real number
\begin{equation}\label{segno}
\la i\Phi_{\vartheta}(\alpha\wedge\ov{\alpha}),h\ra 
\end{equation}
is either strictly positive for all $([\sigma],\alpha)\in\cV^{j}_h$, or always strictly negative.  Conjugation $([\sigma],\alpha)\mapsto ([\ov{\sigma}],\ov{\alpha})$ maps bijectively $\cV^1$ to $\cV^2$. Since conjugation changes sign to the number in~\eqref{segno}, the proposition follows.
\end{proof}
 \begin{expl}\label{expl:posteta}
Let us go back to~\Ref{expl}{jacint}, and assume that  $X$ is projective. Let $L$ be an ample line bundle on $X$.    Referring to~\Ref{expl}{bbfdual},  $c_1(L)$ gets identified with an element of $h\in  (\bigwedge^2 V_{\ZZ}\oplus\ZZ \zeta)^{\vee}$ (see~\eqref{giugno}) of positive square.
  By~\Ref{rmk}{oasibirra}, the bilinear form~\eqref{formbil} is  identified, via the isomorphism in~\eqref{isogenia}, with the  bilinear  form 
\begin{equation}\label{polnat}
\begin{matrix}
H^3(X) \times H^3(X) & \lra & \CC \\
(\alpha,\beta) & \mapsto  &  \int_{[X]}\alpha\smile\beta\smile \ov{q}_X^{n-2}\smile c_1(L).
\end{matrix}
\end{equation}
It follows that if $n=2$,  the  isomorphism
in~\eqref{isogenia} 
matches $\Theta_{[\sigma]}(\vartheta)$ and the  polarization $\Theta_L$ of $J^3(X)$.  Later on we will show that an analogous statement holds also for $n>2$.
\end{expl}
\subsection{A rank $7$ sub local system of the local system with fiber $S^{+}(X)$}
\setcounter{equation}{0}
Let $\bf q$ be the integral unimodular bilinear symmetric form on $V_{\CC}\oplus V_{\CC}^{\vee}$ defined by
\begin{equation}\label{quadstand}
\begin{matrix}
V_{\CC}\oplus V_{\CC}^{\vee} & \overset{\bf q}{\lra} & \CC \\
(v,\ell) & \mapsto & 2\ell(v)
\end{matrix}
\end{equation}
(See~\eqref{eccobi}.) Let ${\bf Q}:=V(q)\subset \PP(V_{\CC}\oplus V_{\CC}^{\vee})$. One of the two spinor representations of $O(\bf q)$ may be identified with $S^{+}:=\bigwedge^{ev}V_{\CC}=\CC\oplus\bigwedge^{2}V_{\CC}\oplus \bigwedge^{4}V_{\CC}$. We recall  the identification of a specific quadric hypersurface in $\PP(S^{+})$ with one of the two irreducible components of the variety parametrizing $3$-dimensional linear subspaces of ${\bf Q}$, see \S 20.3 in~\cite{fulhar}. Denote elements of 
$\bigwedge^{ev}V_{\CC}$ as follows:
\begin{equation}\label{axib}
\alpha+\eta+\beta, \qquad \alpha\in \CC,\ \eta \in \bigwedge^2 V_{\CC},\ \beta \in \bigwedge^4 V_{\CC}.
\end{equation}
Let ${\bf q}^{+}$ be the integral unimodular bilinear symmetric form on $\bigwedge^{ev}V_{\CC}$ defined by
\begin{equation}\label{quadstand}
\begin{matrix}
\bigwedge^{ev}V_{\CC} & \overset{{\bf q}^{+}}{\lra} & \CC \\
\alpha+\eta+\beta & \mapsto & \vol(\eta\wedge\eta-2\alpha\beta)
\end{matrix}
\end{equation}
Let ${\bf Q}^{+}\subset\PP(\bigwedge^{ev}V_{\CC})$ be the set of zeroes of ${\bf q}^{+}$. 

Given $\ell\in V_{\CC}^{\vee}$ and $\eta\in\bigwedge^{\bullet}V_{\CC}$, we let $\ell(\eta)$ be the contraction of $\ell$ and $\eta$. Given 
$[\alpha+\eta+\beta]\in {\bf Q}^{+}$, we let
\begin{equation}
Z_{[\alpha+\eta+\beta]}:=\{[(v,\ell)]\in\PP(V_{\CC}\oplus V_{\CC}^{\vee}) \mid \alpha v+\ell(\eta)+\eta\wedge v+\ell(\beta)=0\}.
\end{equation}
\begin{lmm}\label{lmm:livella}
Let $[\alpha+\eta+\beta]\in {\bf Q}^{+}$, and suppose that $\eta\wedge\eta\not=0$. Then $Z_{[\alpha+\eta+\beta]}$ is the graph of a non degenerate skew-symmetric map $f\colon V_{\CC}\to V^{\vee}_{\CC}$ such that 
\begin{equation}\label{macron}
\vol(\beta) \iota(\omega_f)=-\eta.
\end{equation}
\end{lmm}
\begin{proof}
First we note that 
\begin{equation}\label{eccozeta}
Z_{[\alpha+\eta+\beta]}:=\{[(v,\ell)]\in\PP(V_{\CC}\oplus V_{\CC}^{\vee}) \mid \alpha v+\ell(\eta)=0\}.
\end{equation}
In fact,  assume that $\alpha v+\ell(\eta)=0$. Multiplying by $\eta$ we get that $\alpha v\wedge\eta+\ell(\eta)\wedge\eta=0$. On the other hand $\eta\wedge\eta=2\alpha\beta$, hence 
\begin{equation*}
\ell(\eta)\wedge\eta=\frac{1}{2}\ell(\eta\wedge\eta)=\alpha\ell(\beta).
\end{equation*}
 Thus $\alpha v\wedge\eta+\alpha\ell(\beta)=0$, and since $\alpha\not=0$, it follows that $v\wedge\eta+\ell(\beta)=0$. This proves~\eqref{eccozeta}.
From~\eqref{eccozeta} we get that $Z_{[\alpha+\eta+\beta]}$ is the graph of a nondegenerate map  $f\colon V_{\CC}\to V^{\vee}_{\CC}$. Since 
$0=\ell(\alpha v+\ell(\eta))=\alpha\ell(v)$, the map $f$ is skew-symmetric. 

Lastly, we prove~\eqref{macron}. We may choose a basis 
$\{v_1,\ldots,v_4\}$ of $V_{\CC}$ of volume $1$ such that $\eta=v_1\wedge v_2+tv_3\wedge v_4$, for some $t\in\CC^{*}$. A computation gives that
\begin{equation*}
\omega_f=-\alpha v_1^{\vee}\wedge v_2^{\vee}-\frac{\alpha}{t}v_3^{\vee}\wedge v_4^{\vee}.
\end{equation*}
Equation~\eqref{macron} follows from the above equality.
\end{proof}
The (well-known) result below follows easily from~\Ref{lmm}{livella}.
\begin{prp}\label{prp:qupiu}
Let $[\alpha+\eta+\beta]\in {\bf Q}^{+}$. Then $Z_{[\alpha+\eta+\beta]}$ 
 is a $3$ dimensional linear subspace of ${\bf Q}$, and it is the graph of a non degenerate skew-symmetric map $f\colon V_{\CC}\to V^{\vee}_{\CC}$ if and only if $\eta\wedge\eta\not=0$.
 The map assigning  $Z_{[\alpha+\eta+\beta]}$ to $[\alpha+\eta+\beta]\in {\bf Q}^{+}$ is an isomorphism between ${\bf Q}^{+}$ and one of the two irreducible components of the variety parametrizing maximal dimensional linear subspaces of ${\bf Q}$. 
\end{prp}
\begin{rmk}\label{rmk:livella}
Let $U\subset V_{\CC}$ be a $2$ dimensional subspace. Let $\eta$ be a generator of $\bigwedge^2 U\subset \bigwedge^2 V_{\CC}$. Then $[0,\eta,0]\in{\bf Q}^{+}$, and $Z_{[\alpha+\eta+\beta]}=\PP(U\oplus U^{\bot})$.
\end{rmk}
The following definition makes sense by the~\Ref{kob}{chiave}, \Ref{prp}{grafanti}, \Ref{lmm}{livella}, and~\Ref{rmk}{livella}.
\begin{dfn}\label{dfn:pimpinelli}
Let $X$ be a HK of Kummer type, of dimension $2n$. Let $T^{+}(X)\subset S^{+}(X)$ be the minimal vector subspace such that  $\PP(T^{+}(X))$ contains all $x\in{\bf Q}^{+}(X)$ parametrizing a $3$ dimensional linear space $\PP(\Gamma)\subset{\bf Q}(X)$ for which there exist a HK  $Y$ of Kummer type of dimension $2n$, and a parallel transport operator $g\colon H^3(Y)\overset{\sim}{\lra} H^3(X)$ such that $g(H^{2,1}(X))=\Gamma$. 
\end{dfn}
\begin{rmk}\label{rmk:pimpinelli}
Let $\pi\colon\cX\to B$ be a family of HK's of Kummer type. Let $b_0,b_1\in B$ and let $X_0:=\pi^{-1}(b_0)$, $X_1:=\pi^{-1}(b_1)$. Let $\lambda$ be an arc starting in $b_0$ and ending in $b_1$. By~\Ref{dfn}{pimpinelli} the induced isomorphism $S^{+}(\lambda)\colon S^{+}(X_0)\to S^{+}(X_1)$  maps $T^{+}(X_0)$ to $T^{+}(X_1)$.
\end{rmk}
\begin{prp}\label{prp:dicarcaci}
Keeping notation as above, the following hold:
\begin{enumerate}
\item 
Suppose that $X=K_n(A)$. As usual,  identify $H^3(A;\ZZ)\oplus H^1(A;\ZZ)$ with $H^3(K_n(A);\ZZ)$   via~\Ref{thm}{accatre}, and identify $H^1(A;\ZZ)$ with $H^3(A;\ZZ)^{\vee}$ via~\eqref{31duality}. Then
\begin{equation}\label{tartaglia}
T^{+}(K_n(A)):=\{(\alpha+\eta+\beta)\in \bigwedge^{ev}H^3(A) \mid (n+1)\alpha-\vol(\beta)=0\}.
\end{equation}
\item 
If $\pi\colon\cX\to B$ is a family of HK's of Kummer type, then there is a sub local system of rank $7$ of $S^{+}(\pi)$ with fiber $T^{+}(\pi^{-1}(b))$ over $b\in B$. 
\end{enumerate}
\end{prp}
\begin{proof}
Let us prove Item~(1). Suppose  that $Z_{[\alpha+\eta+\beta]}=\PP(\Gamma)$ where $\Gamma\subset(V_{\CC}\oplus V_{\CC}^{\vee})$ is as in~\Ref{dfn}{pimpinelli}. 
By~\Ref{prp}{grafanti}, either  $\Gamma$  
 is the graph of a non degenerate antisymmetric map $f\colon V_{\CC}\to V_{\CC}^{\vee}$ such that $\Pf(f)=1/(n+1)$ (see~\Ref{expl}{francese} for the last equality), or $\Gamma=U\oplus U^{\bot}$, where  $U\subset V_{\CC}$ is a two dimensional subspace.   
In the first case, by~\eqref{macron} we have
\begin{equation*}
\frac{1}{(n+1)}=\Pf(f)=\frac{1}{2}\vol(\omega_f\wedge\omega_f)=\frac{1}{2\vol(\beta)^2}\vol(\eta\wedge\eta)=\frac{\alpha}{\vol(\beta)}.
\end{equation*}
Thus $[\alpha+\eta+\beta]$ belongs to the right hand side of~\eqref{tartaglia}. In the second case, the same holds because $\alpha=\beta=0$.
This proves that $T^{+}(K_n(A))$ is contained in the right hand side of~\eqref{tartaglia}. 
On the other hand (see~\Ref{expl}{francese}) there is a dense (in the Zariski topology)  open (in the classical topology) subset of 
$\PP(T^{+}(K_n(A)))\cap {\bf Q}^{+}(X)$ parametrizing $\Gamma$'s as in~\Ref{dfn}{pimpinelli}. Item~(1) follows.

It is clear that there is a sub local system  of $S^{+}(\pi)$ with fiber $T^{+}(\pi^{-1}(b))$ over $b\in B$. It has  rank $7$ by Item~(1).
\end{proof}
\subsection{Proof of the second main result}
\setcounter{equation}{0}
Item~(1) of~\Ref{thm}{secondoteor} follows from~\Ref{dfn}{pimpinelli} and Item~(2) of~\Ref{prp}{dicarcaci}. 

Next we proceed to prove Item~(2) of~\Ref{thm}{secondoteor}. Let $X=K_n(A)$. As usual identify $H^2(K_n(A))^{\vee}=\bigwedge^2 H^3(A)\oplus \CC\xi_n^{\vee}$, see~\Ref{expl}{bbfdual}. With this identification, $S^{+}(K_n(A))=\CC\oplus\bigwedge^2 H^3(A)\oplus \bigwedge^4 H^3(A)$. Let $\tau \in\bigwedge^4 H^3(A)$ be the element of volume $1$. 
We let
\begin{equation}\label{immergo}
\begin{matrix}
\bigwedge^2 H^3(A)\oplus \CC\xi_n^{\vee} & \overset{i}{\hra} & S^{+}(K_n(A)) \\
\eta+x\xi_n^{\vee} & \mapsto & \frac{(-1)^{\epsilon_n}x}{2(n+1)}+\eta+\frac{(-1)^{\epsilon_n}x}{2}\tau
\end{matrix}
\end{equation}
By Item~(1) of~\Ref{prp}{dicarcaci}, the above map defines an isomorphism between $H^2(K_n(A))^{\vee}$ and $T^{+}(K_n(A))$. 
Let ${\bf q}^{+}_{K_n(A))}$ be the quadratic form on $ S^{+}(K_n(A))$ defined by~\eqref{quadstand} with $V=H^3(A;\ZZ)$. One checks easily that the restriction of ${\bf q}^{+}_{K_n(A))}$ to $H^2(K_n(A))^{\vee}$ is the dual of the BBF. This proves Item~(1) of~\Ref{thm}{secondoteor} for $X=K_n(A)$. 
Item~(2) of~\Ref{thm}{secondoteor} for $X=K_n(A)$ follows from~\Ref{lmm}{livella}. 
\begin{clm}\label{clm:bendef}
The map $i$ in~\eqref{immergo} is equivariant up to sign for the action of monodromy. 
\end{clm}
\begin{proof}
Let $\pi\colon\cX\to B$ be a family of HK's of Kummer type, with the fiber $\pi^{-1}(b_0)$ isomorphic to $K_n(A)$. Let $\lambda$ be a loop in $B$ based at $b_0$. Then $\lambda$ defines a diffeomorphism $\lambda_{*}\colon K_n(A)\to K_n(A)$. 
Let $H^2(\lambda_{*})^t$ be the action of $\lambda_{*}$ on  $H^2(K_n(A))^{\vee}$, and let $S^{+}(\lambda_{*})$ be the action of $\lambda_{*}$ on $S^{+}(K_n(A))$. We must prove that 
\begin{equation}\label{daniele}
S^{+}(\lambda_{*})\circ i=\pm i\circ H^2(\lambda_{*})^t. 
\end{equation}
First $S^{+}(\lambda_{*})$ maps $T^{+}(K_n(A))$ to itself - see~\Ref{rmk}{pimpinelli}.  Next, let $\Gamma\subset H^3(K_n(A))$ be a $4$-dimensional vector subspace such that $\phi(\bigwedge^2\Gamma)$ has dimension $1$. Then $S^{+}(\lambda_{*})\circ i(\phi(\bigwedge^2\Gamma_{*}))=i\circ H^2(\lambda_{*})^t\phi(\bigwedge^2\Gamma)$ by Item~(2) (which has been proved for $X=K_n(A)$). The set of elements of $i(H^2(K_n(A))^{\vee}\cap {\bf Q}^{+}(K_n(A))$ of the form $\phi(\bigwedge^2\Gamma)$ for $\Gamma$ as above is an open dense subset of $i(H^2(K_n(A))^{\vee}\cap {\bf Q}^{+}(K_n(A))$. Equation~\eqref{daniele} follows. 
\end{proof}
Now we may define $i\colon H^2(X)^{\vee} \overset{i}{\hra}  S^{+}(X)$ for arbitrary $X$ acting by parallel transport  on the map $i$ in~\eqref{immergo}. The map is well-defined up to sign, i.e.~independent (up to sign) of the chosen path connecting $K_n(A)$ to $X$ (in a connected family of HK's in which both $K_n(A)$ and $X$ are fibers), because of~\Ref{clm}{bendef}. Items~(2) of~\Ref{thm}{secondoteor} for $X$ follows from the corresponding statements for $K_n(A)$.

\begin{proof}[Proof of~\Ref{crl}{monlim}]
We may assume that $X=K_n(A)$. Thus we have the identifications of~\Ref{prp}{dicarcaci}.  Let $\pi\colon\cX\to B$ be a family of HK's of Kummer type, with  fiber $\pi^{-1}(b_0)$ isomorphic to $K_n(A)$. Let $\lambda$ be a loop in $B$ based at $b_0$, and let $\lambda_{*}\colon K_n(A)\to K_n(A)$ 
be the associated diffeomorphism. We have associated maps
\begin{equation*}
\scriptstyle
H^2(\lambda_{*})\in O(H^2(X;\ZZ),q_{X}),\quad  S^{+}(\lambda_{*})\in O(S^{+}(X),{\bf q}^{+}_{X}).
\end{equation*}
(We let  $X=K_n(A)$ for typographical reasons.)  Since the BBF quadratic form $q_X$ is non-degenerate, it defines an embedding $H^2(K_n(A);\ZZ)^{\vee}\subset H^2(K_n(A);\QQ)$. By~\eqref{daniele} we have
\begin{equation}\label{intreccio}
 H^2(\lambda_{*})^t=\pm i^{-1}\circ S^{+}(\lambda_{*})\circ i.
\end{equation}
 Thus we examine $S^{+}(\lambda_{*})$.  In order to simplify notation, we let $\rho:=S^{+}(\lambda_{*})$.
First, we notice that
\begin{equation}\label{detuno}
\Det \rho=1
\end{equation}
because by Item~(3) of~\Ref{thm}{primoteor} monodromy does not exchange the two irreducible components of the variety parametrizing maximal linear subspaces of ${\bf Q}(K_n(A))$.

Let $\tau \in\bigwedge^4 H^3(A)$ be the element of volume $1$. Let
\begin{equation}
u:=1-(n+1)\tau,\quad w:=1+(n+1)\tau.
\end{equation}
Then $u,w\in S^{+}(K_n(A))$, $u$ is orthogonal to $T^{+}(K_n(A))$, and $w\in T^{+}(K_n(A))$.

Since $\rho$ maps $T^{+}(K_n(A))$ to iself, we have $\rho(u)=\pm u$. By~\eqref{detuno}, and the equation $i(2(n+1)\xi_n^{\vee})=w$ we must prove that
\begin{equation}\label{duecasi}
\rho(w)=
\begin{cases}
(1+2a(n+1))w+2(n+1)y,\ \ a\in\ZZ,\  y\in T^{+}(K_n(A);\ZZ) & \text{if $\rho(u)= u$,} \\
(-1+2a(n+1))w+2(n+1)y,\ \ a\in\ZZ,\ y\in T^{+}(K_n(A);\ZZ) & \text{if $\rho(u)= -u$.}
\end{cases}
\end{equation}
Suppose that  $\rho(u)= u$. The $\rho(w)-u=\rho(w-u)=\rho(2(n+1)\tau)=2(n+1)\rho(\tau)$. Since $\rho(\tau)\bot u$, we get that there exist $a\in\ZZ$ and 
$ y\in T^{+}(K_n(A);\ZZ)$ such that $\rho(\tau)=aw+\tau+y$. Thus 
\begin{equation}
\rho(w)=u+2(n+1)\tau+2a(n+1)w+2(n+1)y=(1+2a(n+1))w+2(n+1)y.
\end{equation}
This proves~\eqref{duecasi} if $\rho(u)= u$. The proof in the case $\rho(u)= -u$ is similar.
\end{proof}
\section{Polarization type of $J^3(X)$ for $X$ of dimension $4$}\label{sec:polandmon}
If $X$ is a polarized HK of Kummer type, then $J^3(X)$ is a four dimensional abelian variety, with a polarization associated to the polarization of $X$. In the present subsection we  compute the discrete invariants (elementary divisors)  of  the polarization of $J^3(X)$, for $X$ of dimension $4$.
\subsection{Set up}
\setcounter{equation}{0}
Let $X$ be a HK fourfold of Kummer type, and let $L$ be an ample line bundle on $X$. The  skew-symmetric form
\begin{equation}\label{polarizzo}
\begin{matrix}
H^3(X)\times H^3(X) & \overset{\la,\ra}{\lra} & \CC \\
(\alpha,\beta) & \mapsto & \int_X\alpha\smile\beta\smile c_1(L)
\end{matrix}
\end{equation}
defines a polarization $\Theta_L$ of $J^3(X)$ by the Hodge-Riemann bilinear relations. We may assume that $X$ is a generalized Kummer $K_2(A)$. We recall that~\eqref{accaduekum} and the map $\mathsf F$ in~\eqref{bronte} give identifications
\begin{eqnarray}
 H^2(K_2(A);\ZZ) & = &\bigwedge^2 H^1(A;\ZZ)\oplus \ZZ\xi_2,\\
  H^3(K_2(A);\ZZ) & = & H^3(A;\ZZ)\oplus H^1(A;\ZZ).
\end{eqnarray}
We identify $H^1(A;\ZZ)$ with $H^3(A;\ZZ)^{\vee}$ via~\eqref{31duality}, and we set $V:=H^3(A;\ZZ)$. 
Let  $\{v_1,\ldots,v_4\}$ be a basis of $V$ such that $\vol(v_1\wedge v_2\wedge v_3\wedge v_4)=1$. We may write 
\begin{equation}\label{eccopol}
c_1(L)=c(e v^{\vee}_1\wedge v^{\vee}_2+v^{\vee}_3\wedge v^{\vee}_4)+s\zeta^{\vee},
\end{equation}
where
\begin{equation}\label{ipotesi}
 c\in\NN_{+},\quad s\in\ZZ,\quad \gcd\{c,s\}=1.
\end{equation}

Let $(w,f),(w',f')\in (V\oplus V^{\vee})$. By~\Ref{crl}{thirdman} and~\Ref{rmk}{menosei} we have
\begin{equation*}
\scriptstyle
\la(w,f),(w',f')\ra=\la- w\wedge w'-3 \iota(f\wedge f')-3(\la w,f'\ra-\la w',f\ra)\zeta^{\vee},c(e v^{\vee}_1\wedge v^{\vee}_2+v^{\vee}_3\wedge v^{\vee}_4)-s\zeta^{\vee}\ra.
\end{equation*}
In the above equation, the angle brackets in the left hand side denote the polarization form in~\eqref{polarizzo}, those in the right hand side denote the duality pairing. 
Let 
\begin{multline}\label{alfabeto}
\alpha_1:=(0,v_1^{\vee}),\ \alpha_2:=(v_2,0),\ \alpha_3:=(v_4,0),\ \alpha_{4}:=(0,v_3^{\vee}),\\ 
\beta_1:=(0,-v_2^{\vee}),\ \beta_2:=(v_1,0),\ \beta_3:=(v_3,0),\ \beta_{4}:=(0,-v_4^{\vee}).
\end{multline}
Then $\{\alpha_1,\ldots,\beta_4\}$ is a basis of $V\oplus V^{\vee}/2$, and  both the span of the $\alpha_i$'s and the span of the $\beta_j$'s   are $\la,\ra_{\vartheta,h}$-isotropic  subgroups of $V\oplus V^{\vee}/2$.  The intersection matrix between the $\alpha_i$'s and the $\beta_j$'s   is equal to
\begin{equation}\label{intmat}
(\la \alpha_i,\beta_j\ra_{\vartheta,h})=
\left(
\begin{array}{cccc}
3c  & 3s  & 0 & 0   \\
 3s & c\cdot e & 0  & 0  \\
0  & 0  & c & 3s \\
0 & 0 & 3s & 3c\cdot e  
\end{array}
\right)
\end{equation}
 Let $X$ be a $4$ dimensional generalized Kummer. By Corollary 4.8 in~\cite{marmer} (the proof is in~\cite{mar-on-weil}), non zero elements  $\alpha\in H^2(X;\ZZ)$  up to monodromy  are classified by the value $q_X(\alpha)$ and by the divisibility $\divisore(\alpha)$ (see~\Ref{subsec}{ginevra} for the definition of  $\divisore(\alpha)$). The divisibility is an element of $\{1,2,3,6\}$. Thus we distinguish four cases.
\subsection{Divisibility $1$}
\setcounter{equation}{0}
Suppose that 
\begin{equation}\label{essezero}
c_1(L)=ev^{\vee}_1\wedge v^{\vee}_2+ v^{\vee}_3\wedge v^{\vee}_4,
\end{equation}
i.e.~$c=1$ and $s=0$ in the notation of~\eqref{eccopol}. Then  
 $(c_1(L),c_1(L))=2e$, and $\divisore(c_1(L))=1$.
Let $g:=\gcd\{3,e\}$, and let $x,y$ be integers such that $3x+ey=g$. A basis of $V\oplus\frac{1}{2} V^{\vee}$ is given by 
\begin{equation*}
%
\{\alpha_{3},\ x\alpha_1+y\alpha_2,\ (e\alpha_1-3\alpha_2)/g,\ \alpha_4,\ \beta_{3},\ \beta_1+\beta_2,\ 
(ey\beta_1-3x\beta_2)/g,\ \beta_4\}.
\end{equation*}
The matrix of  $\la,\ra$ in the above basis is equal to 
$\begin{pmatrix} 0 & \Delta \\
-\Delta & 0\end{pmatrix}$, where $\Delta$ is the $4\times 4$ diagonal matrix with entries $1,g,3e/g,3e$. 
\subsection{Divisibility $2$}
\setcounter{equation}{0}
Suppose that 
\begin{equation}\label{essezero}
c_1(L)=2(e v^{\vee}_1\wedge v^{\vee}_2+ v^{\vee}_3\wedge v^{\vee}_4)+\zeta^{\vee},
\end{equation}
i.e.~$c=2$ and $s=1$ in the notation of~\eqref{eccopol}. Then  
  $(c_1(L),c_1(L))=2(4e-3)$, and  $\divisore(c_1(L))=2$. Let $g:=\gcd\{3,e\}=\gcd\{3,2e\}$, and let $x,y\in\ZZ$ be such that $3x+2ey=g$. A basis of $V\oplus\frac{1}{2} V^{\vee}$ is given by
\begin{equation*}
%
\{\alpha_3,\, x\alpha_1+y\alpha_2,\, (2e\alpha_1-3\alpha_2)/g,\,(6e-3)\alpha_3-\alpha_4,\,
\beta_4-\beta_3,\,\beta_2,\,(g\beta_1-(6x+3y)\beta_2)/g,\,3\beta_3-2\beta_4\}.
\end{equation*}
The matrix of  $\la,\ra$ in the above basis is equal to 
$\begin{pmatrix} 0 & \Delta \\
-\Delta & 0\end{pmatrix}$, where $\Delta$ is the $4\times 4$ diagonal matrix with entries $1,g,3(4e-3)/g,3(4e-3)$. 

\subsection{Divisibility $3$}
\setcounter{equation}{0}
Suppose that 
\begin{equation}\label{essezero}
c_1(L)=3(e v^{\vee}_1\wedge v^{\vee}_2+ v^{\vee}_3\wedge v^{\vee}_4)+\zeta^{\vee},
\end{equation}
i.e.~$c=3$ and $s=1$ in the notation of~\eqref{eccopol}. Then  
  $(c_1(L),c_1(L))=6(3e-1)$, and  $\divisore(c_1(L))=3$. A basis of $V\oplus\frac{1}{2} V^{\vee}$ is given by
\begin{equation*}
%
\{\alpha_3,\ \alpha_1,\ \alpha_{4}-\alpha_3,\ e\alpha_1-\alpha_2,\ \beta_3,\ \beta_2,\ \beta_{4}-\beta_3,\ \beta_1-3\beta_2\}.
\end{equation*}
The matrix of  $\la,\ra$ in the above basis is equal to 
$\begin{pmatrix} 0 & \Delta \\
-\Delta & 0\end{pmatrix}$, where $\Delta$ is the $4\times 4$ diagonal matrix with entries $3,3,3(3e-1),3(3e-1)$. 
\subsection{Divisibility $6$}
\setcounter{equation}{0}
Suppose that 
\begin{equation}\label{essezero}
c_1(L)=6(e v^{\vee}_1\wedge v^{\vee}_2+ v^{\vee}_3\wedge v^{\vee}_4)+\zeta^{\vee},
\end{equation}
i.e.~$c=6$ and $s=1$ in the notation of~\eqref{eccopol}. Then  
  $(c_1(L),c_1(L))=6(12e-1)$, and  $\divisore(c_1(L))=6$. A basis of $V\oplus\frac{1}{2} V^{\vee}$ is given by
\begin{equation*}
%
\{\alpha_1,\,\alpha_3,\,2\alpha_3-\alpha_4,\,2e\alpha_1-\alpha_2,\,\beta_2,\,\beta_4,\,\beta_3-6e\beta_4,\,\beta_1-6\beta_2\}.
\end{equation*}
The matrix of  $\la,\ra$ in the above basis is equal to 
$\begin{pmatrix} 0 & \Delta \\
-\Delta & 0\end{pmatrix}$, where $\Delta$ is the $4\times 4$ diagonal matrix with entries $3,3,3(12e-1),3(12e-1)$.

\section{ Weil type}\label{sec:tipoweil}
\setcounter{equation}{0}
\subsection{Abelian varieties of Weil type}\label{subsec:weilrecap}
\setcounter{equation}{0}
We recall that a compact complex torus $T$ of dimension $2g$ is of \emph{Weil type} (see~\cite{weil-on-hc}) if there exists an endomorphism $\varphi\colon T\to T$ such that the following hold:
\begin{enumerate}
\item
$\varphi\circ\varphi=-D \Id_T$, where $D$ is a strictly positive integer.
\item
The restriction of $\varphi^{*}$ to $H^{1,0}(T)$ decomposes as the direct sum of $\pm\sqrt{-D}$ eigenspaces of the same dimension $g$.
\end{enumerate}
Such a torus $T$ has a $2$ dimensional space of classes in $H^{g,g}_{\ZZ}(T)$ which are not in the ring generated by  
$H^{1,1}_{\ZZ}(T)$ unless $g=1$. Voisin~\cite{voisin-hc-nonproj} proved that they provide counterexamples to the extension of the Hodge conjecture to compact K\"ahler manifolds. On the other hand, for certain families of abelian varieties of Weil type it is known that the Weil classes are algebraic~\cite{schoen-hc}. As references for what follows, we recommend~\cite{berthodge}  and~\cite{schoen-hc-add}. 

If $A$ is an abelian variety of Weil type, with endomorphism $\varphi$, there exists a polarization $\Theta$ such that $\varphi^{*}\Theta\equiv d\Theta$. If this is the case, one says that $(A,\varphi,\Theta)$ is a polarized abelain variety of Weil type. Let us view the polarization $\Theta$ as a bilinear alternating function $E\colon H_1(A;\QQ)\times H_1(A;\QQ)\to\QQ$. The endomorphism $\varphi$ gives $H_1(A;\QQ)$ the structure of a vector space over the quadratic field $\KK:=\QQ[\sqrt{-D}]$. One defines
\begin{equation}\label{formahermit}
\begin{matrix}
H_1(A;\QQ)\times H_1(A;\QQ) & \overset{H}{\lra} & \KK \\
(\alpha,\beta) & \mapsto & E(\alpha,\varphi_{*}\beta)+\sqrt{-D} E(\alpha,\beta)
\end{matrix}
\end{equation}
As is easily checked $H$ is $\KK$ linear in the second entry, and $H(\beta,\alpha)=\ov{(\alpha,\beta)}$. 
Thus $H$ is a nondegenerate Hermitian form on the $\KK$ vector space $H_1(A;\QQ)$. The determinant of the Hermitian matrix associated to $H$ by a choice of $\KK$-basis of  $H_1(A;\QQ)$ is well-determined modulo moltiplication by elements of $\Nm(\KK^{*})$.  
Thus we may associate to $H$ its determinant $\Det H\in \QQ^{*}\backslash \Nm(\KK^{*})$.  We denote $\Det H$ by $\Det \Theta$. 

Given an imaginary quadratic field $\KK$, and an element of $ \QQ^{*}\backslash \Nm(\KK^{*})$, one may construct a complete (up to isogeny) \emph{irreducible} family of $2g$ dimensional polarized abelian 
varieties of Weil type $(A,\varphi,\Theta)$ with associated field $\KK$, and  assigned $\Det$ of the polarization, of dimension $g^2$.  complete up to isogeny means that every polarized abelian 
variety of Weil type $(A,\varphi,\Theta)$ with the given field  and determinant is isogenous to one of the varieties in the family (of course the isogeny matches the endomorphisms and the polarizations). 
\subsection{The abelian variety associated to a point of $\cD_h$ is of Weil type}\label{subsec:tipoweil}
\setcounter{equation}{0}
We suppose that $\vartheta=(\vartheta_1,\vartheta_2,\vartheta_3)\in\ZZ^3$, with all entries nonzero. We suppose also that $m$ is a (strictly) positive rational number, and that Equation~\eqref{seconda} holds. We will adopt the notation of~\Ref{sec}{alglin} without further notice. We will prove the following two results.
\begin{thm}\label{thm:sorpresa}
Let    $h\in (\bigwedge^2 V\oplus \ZZ)^{\vee}$ be a vector of positive square, and assume that $\sigma\in\cD^{+}_h$ (see~\Ref{prp}{ampio} for the definition of $\cD^{+}_h$). Let $J_{[\sigma]}(\vartheta)$ be the compact complex torus in~\Ref{dfn}{torosi}, and let $\Theta_{[\sigma]}(\vartheta)\in H^{1,1}_{\QQ}(J_{[\sigma]}(\vartheta)$ be the ample class 
in~\Ref{dfn}{polab}.  Then $(J_{[\sigma]}(\vartheta),\Theta_{[\sigma]}(\vartheta))$ is of Weil type, with an 
embedding 
\begin{equation}\label{eccocampo}
\QQ[\sqrt{-m(h,h)}]\subset\End(J_{[\sigma]}(\vartheta),\Theta_{[\sigma]}(\vartheta))_{\QQ}.
\end{equation}
The determinant of the   polarization $\Theta_{[\sigma]}(\vartheta)$ is $1$.  By varying  $\sigma\in\cD^{+}_h$, one gets a complete (up to isogeny)  family of $4$ dimensional  polarized abelian varieties of Weil type  with associated field $\QQ[\sqrt{-m(h,h)}]$, and  $\Det\equiv 1$.
\end{thm}
Before proving~\Ref{thm}{sorpresa},  we will go through a few elementary results. 
The proof of the next lemma is a straightforward exercise.
\begin{lmm}\label{lmm:invskew}
Let $X=(x_{ij})$ be a $4\times 4$ invertible antisymmetric matrix with coefficients in a field $\KK$. Then
\begin{equation*}
X^{-1}=
\Pf(X)^{-1}\left(
\begin{array}{rrrr}
 0 & -x_{34}  &  x_{24} & -x_{23} \\
x_{34}  & 0  & -x_{14} & x_{13}   \\
-x_{24}  & x_{14}  & 0 & -x_{12} \\
  x_{23} & -x_{13} & x_{12} & 0
\end{array}
\right)
\end{equation*}
\end{lmm}
\begin{lmm}\label{lmm:dragomir}
Let $X,Y$ be $4\times 4$ antisymmetric matrices over a field $\KK$ (if $\charact\KK=2$, a matrix $X=(x_{ij})$ is antisymmetric if $X^t=-X$, and $x_{ii}=0$ for all $i$). Then
\begin{equation}
(X\cdot Y)^2-\frac{1}{2}\Tr(X\cdot Y)X\cdot Y+\Pf(X)\cdot\Pf(Y)1_4=0.
\end{equation}
(Note:  $\frac{1}{2}\Tr(X\cdot Y)=-\sum_{1\le i<j\le 4}x_{ij}y_{ij}$, hence it makes sense even if $\charact\KK=2$.) 
\end{lmm}
\begin{proof}
This is the content of the main result of~\cite{djokovic},  in the case of $4\times 4$  matrices. In fact, let  $p\in\ZZ[x_{ij},y_{kh}][\lambda]$ be 
equal to $\Pf(X)\cdot\Pf(\lambda X^{-1}-Y)$. In~\cite{djokovic}  it is proved that $p(X\cdot Y)=0$ (we replace $\lambda$ by $X\cdot Y$). Expanding   $\Pf(X)\cdot\Pf(\lambda X^{-1}-Y)$ (\Ref{lmm}{invskew} will be handy), one gets the lemma.
\end{proof}
\begin{proof}[Proof of~\Ref{thm}{sorpresa}.]
By the Theorem on elementary divisors, there exists a basis $\cB=\{v_1,\ldots,v_4\}$ of $V$ (of volume $1$) such that 
\begin{equation}\label{accanorm}
h=h_0+s\zeta^{\vee},\qquad h_0=c(e v^{\vee}_1\wedge v^{\vee}_2+ v^{\vee}_3\wedge v^{\vee}_4),\quad c,e\in\NN_{+},\ s\in\ZZ.
\end{equation}
 Let  $g\colon V_{\CC}\to V^{\vee}_{\CC}$ be the antisymmetric map such that $\omega_g=h_0$. Notice that  $g(V_{\QQ})=V_{\QQ}^{\vee}$.  For $N,b\in\QQ$, let
\begin{equation}\label{mappapsi}
\begin{matrix}
 V_{\CC}\oplus V_{\CC}^{\vee} & \overset{\Psi}{\lra} & V_{\CC} \oplus V_{\CC}^{\vee} \\
(v,\ell) & \mapsto & ( g^{-1}(\ell)-b v,b \ell -N g(v))
\end{matrix}
\end{equation}
Then $\Psi(V_{\QQ}\oplus V_{\QQ}^{\vee})=V_{\QQ}\oplus V_{\QQ}^{\vee}$, and
\begin{equation}
\Psi\circ\Psi=-(N-b^2)\Id_{V_{\CC}\oplus V_{\CC}^{\vee}}. 
\end{equation}
The map $\Psi$ induces an endomorphism of $J_{[\sigma]}(\vartheta)$  if
\begin{equation}\label{danilo}
\Psi(H^{1,0}_{[\sigma]})  \subset H^{1,0}_{[\sigma]}.
\end{equation}

If $[\sigma]\in\zeta^{\bot}$, then~\eqref{danilo} holds for any choice of $N,b$. 
In fact, by~\Ref{prp}{grafanti}, there exists a two dimensional subspace $U\subset V_{\CC}$ such that $H^{1,0}_{[\sigma]}(\vartheta)=U\oplus U^{\bot}$, and   $[\sigma]=\bigwedge^2 U$. Since $\la, h_0,\sigma\wedge \ra=0$ (because $[\sigma]\in\zeta^{\bot}$), and $\omega_g=h_0$, 
 the subspace $U$ is isotropic for the symplectic form $\omega_{g}$; it follows that $g(U)=U^{\bot}$. This shows that 
 $\Psi(u,0)\subset H^{1,0}_{[\sigma]}(\vartheta)$ for all $u\in U$. Similarly, one checks tht $\Psi(0,\ell)\subset H^{1,0}_{[\sigma]}(\vartheta)$ for all $\ell\in U^{\bot}$. This proves that~\eqref{danilo} holds if $[\sigma]\in\zeta^{\bot}$ and  $N,b$ are arbitrary.

Now let us assume  that $\sigma\not\in\zeta^{\bot}$. By~\Ref{thm}{grafanti}, we may 
assume that
\begin{equation}\label{periodo}
 \sigma=\vartheta_2\iota(\omega_f)-2\vartheta_3\zeta,\quad H^{1,0}_{[\sigma]}(\vartheta)  =  \{(v,f(v)) \mid v\in V_{\CC}\},
\end{equation}
where $f\colon V_{\CC}\to V_{\CC}^{\vee}$ is an invertible antisymmetric map, and $\omega_f\in\bigwedge^2 V^{\vee}$ is the symplectic form associated to $f$. In particular~\eqref{danilo} holds if and only if
\begin{equation}\label{quadrenne}
g^{-1}\circ f\circ g^{-1}\circ f-2b(g^{-1}\circ f)+N\Id_{V_{\CC}}=0.
\end{equation}
There exist $N,b\in \QQ$ such that~\eqref{quadrenne} holds because of~\Ref{lmm}{dragomir}. 
In fact, let $X,Y$ be the matrices of $g$ and $f$  respectively (with respect to the bases $\cB,\cB^{\vee}$). Because of the equality $\la h,\vartheta_2\iota(\omega_f)-2\vartheta_3\zeta\ra=0$, we have
\begin{equation}\label{traraz}
\Tr(X^{-1}\cdot Y)= 2c^{-1}e^{-1}(y_{12}+e y_{34})=2c^{-2}e^{-1}\la \iota(h_0,\iota(\omega_f)\ra=4c^{-2}e^{-1}s\frac{\vartheta_3}{\vartheta_2}.
\end{equation}
Equation~\eqref{traraz} is the key point: the trace on the left hand side is rational because $[\sigma]\in\cD_h$. Next, we have
\begin{equation*}
\Pf(X^{-1})= c^{-2} e^{-1},\quad \Pf(Y)=\frac{\vartheta_2}{\vartheta_1}.
\end{equation*}
 By~\Ref{lmm}{dragomir} it follows that~\eqref{quadrenne}  holds if we set
\begin{equation}\label{ennebi}
N:=c^{-2} e^{-1}\frac{\vartheta_1}{\vartheta_2},\quad b:=c^{-2} e^{-1}s\frac{\vartheta_3}{\vartheta_2}.
\end{equation}
We have proved that if $N,b$ are as above, then $\Psi$ induces an endomorphism 
\begin{equation}
\varphi\colon J_{[\sigma]}(\vartheta)\to J_{[\sigma]}(\vartheta)
\end{equation}
 such that $\varphi\circ\varphi=-(N-b^2)\Id$. Using~\eqref{seconda}, we get that
\begin{equation}
N-b^2=c^{-4} e^{-2}\frac{\vartheta^2_3}{\vartheta^2_2}(2mc^2e-s^2)=c^{-4} e^{-2}\frac{\vartheta^2_3}{\vartheta^2_2}m(h,h)^{\vee}.
\end{equation}
Thus $\varphi$ defines an embedding $\QQ[\sqrt{-m(h,h)}]\subset\End(J_{[\sigma]}(\vartheta))_{\QQ}$. 
 In order to prove~\eqref{eccocampo}, it remains to show that 
\begin{equation}\label{polaretto}
\la\Psi(\alpha),\Psi(\beta)\ra_{\vartheta,h}  = (N-b^2) \la \alpha,\beta\ra_{\vartheta,h}.
\end{equation}
Let $N,b$ be as in~\eqref{ennebi}, and let
\begin{equation*}
\lambda_{1}:=b+i\sqrt{N-b^2},\quad \lambda_{2}:=b-i\sqrt{N-b^2}.
\end{equation*}
Let   $E_{\pm i\sqrt{N-b^2}}\subset V_{\CC}\oplus V^{\vee}_{\CC}$ be the  $\Psi$-eigenspace  with eigenvalue $\pm i\sqrt{N-b^2}$. An easy computation gives that
\begin{equation}\label{camposcuola}
E_{i\sqrt{N-b^2}}  =  \left\{(v,\lambda_1 g(v)) \mid v\in V_{\CC}\right\}, \quad
E_{-i\sqrt{N-b^2}}  =  \left\{(v,\lambda_2  g(v))  \mid v\in V_{\CC}\right\}.
\end{equation}
We claim that
\begin{equation}\label{isotropo}
\la\,\,, \,\ra_{\vartheta,h}|_{E_{\pm i\sqrt{N-b^2}}}=0, 
\end{equation}
and that for $v,w\in V_{\CC}$, we have
\begin{equation}\label{fattore}
\la (v,\lambda_{1} g(v)),(w,\lambda_{2} g(w)\ra_{\vartheta,h} = \vartheta_1 c^{-2} e^{-1} (h,h)^{\vee} \omega_g(v,w).
\end{equation}
In order to prove~\eqref{isotropo}, let $j,k\in\{1,2\}$, and let $v,w\in V_{\CC}$. Then (keep in mind that $\omega_g=h_0$)
\begin{multline}\label{assaggini}
\la (v,\lambda_{j} g(v)),(w,\lambda_{k} g(w))\ra_{\vartheta,h} = \vartheta_1\la  \omega_g, v\wedge w\ra+\\
+\vartheta_2 \lambda_j \lambda_k \la \omega_g, \iota(g(v)\wedge g(w)\ra-s\vartheta_3(\lambda_j+\lambda_k) \omega_g(v,w).
\end{multline}
We have $\la  \omega_g, v\wedge w\ra=\omega_g(v,w)$, and a simple argument shows that
\begin{equation*}
 \la \omega_g, \iota(g(v)\wedge g(w))\ra=\Pf(g)\omega_g(v,w)=c^2 e \omega_g(v,w).
\end{equation*}
Thus~\eqref{assaggini} reads
\begin{equation}\label{assaggi}
\la (v,\lambda_{j} g(v)),(w,\lambda_{k} g(w))\ra_{\vartheta,h} = \omega_g(v,w)(\vartheta_2 c^2 e \lambda_j \lambda_k  
-s\vartheta_3(\lambda_j+\lambda_k)+\vartheta_1).
\end{equation}
A minimal polynomial of $\lambda_j$ is equal to $x^2-2bx+N$. By~\eqref{ennebi} we get that Equations~\eqref{isotropo} and~\eqref{fattore} follow from~\eqref{assaggi}. Equation~\eqref{polaretto} follows at once from~\eqref{isotropo} and~\eqref{fattore}.

We must also prove that the $\pm i\sqrt{N-b^2}$-eigenspaces of the action of $\Psi$ on $H^{1,0}_{[\sigma]}(\vartheta)$ have dimension $2$. Since $\cD^{+}_h$ is irreducible,   the dimensions of  $\pm i \sqrt{N-b^2}$-eigenspaces  are independent of $[\sigma]\in\cD^{+}_h$. Hence we may assume that $[\sigma]\in\zeta^{\bot}$. Thus there exists a two dimensional subspace $U\subset V_{\CC}$ such that $H^{1,0}_{[\sigma]}(\vartheta)=U\oplus U^{\bot}$. The statement about eigenspaces follows at once from~\eqref{camposcuola}. This finishes the proof that~\eqref{eccocampo} holds.

Next, we prove that $\Det\Theta_{[\sigma]}(\vartheta)\equiv 1$. Let $H$ be the Hermitian form  defined by~\eqref{formahermit}. We must compute the determinant of the Gram matrix of $H$ relative to a basis  f $V_{\QQ}\oplus V^{\vee}_{\QQ}$ as vector space over $\QQ[i\sqrt{N-b^2}]$. 
Let $\cB=\{v_1,\ldots,v_4\}$ be the basis of $V$ such that~\eqref{accanorm} holds. Then $\{(v_1,0),\ldots,(v_4,0)\}$ is a basis of $V_{\QQ}\oplus V^{\vee}_{\QQ}$ as vector space over $\QQ[i\sqrt{N-b^2}]$. 
A computation gives that the Gram matrix of $H$ relative to the chosen $\QQ[i\sqrt{N-b^2}]$-basis is equal to
\begin{equation*}
\begin{pmatrix}
0 & i\vartheta_1 ce\sqrt{N-b^2} & 0 & 0 \\
 -i\vartheta_1 ce\sqrt{N-b^2} & 0 & 0 & 0 \\
0 & 0 &  0 & i\vartheta_1 c\sqrt{N-b^2}  \\
0 & 0 & -i\vartheta_1 c\sqrt{N-b^2} & 0  \\
\end{pmatrix}.
\end{equation*}
Thus $\Det H\in\Nm(\QQ[i\sqrt{N-b^2}])$.

It remains to show that, by varying  $\sigma\in\cD_h$, one gets a complete (up to isogeny) family of \lq\lq polarized\rq\rq\ tori of Weil type with fixed discrete invariants.  Let $(V\oplus V^{\vee},H^{p,q})$ be a weight $1$ Hodge structure such that $\Psi$  induces a homomorphism of $ (V_{\CC}\oplus V^{\vee}_{\CC})/ (H^{1,0}+ V\oplus V^{\vee})$, i.e.~such that $\Psi(H^{1,0})\subset H^{1,0}$, and such that the restriction of $\Psi$ to $H^{1,0}$ has eigenspaces of equal dimensions (i.e.~of dimension $2$). Then $\dim(H^{1,0}\cap
 E_{\pm i\sqrt{N-b^2}})=2$. Moreover, since $H^{1,0}$ is isotropic for $\la , \ra_{\vartheta,h}$,
 we have
 \begin{equation*}
H^{1,0}\cap E_{-i\sqrt{N-b^2}}=(H^{1,0}\cap E_{i\sqrt{N-b^2}})^{\bot}.
\end{equation*}
(Recall that  $\la , \ra_{\vartheta,h}$ gives a perfect pairing between $E_{i\sqrt{N-b^2}}$ and $E_{-i\sqrt{N-b^2}}$.) Let $\pi\colon(V_{\CC}\oplus 
V^{\vee}_{\CC})\to V_{\CC}$ be the projection. Let 
 \begin{equation*}
W_{\pm}:=\pi\left(H^{1,0}\cap E_{\pm i\sqrt{N-b^2}}\right).
\end{equation*}
Then $W_{\pm}$ are two dimensional subspaces of $V_{\CC}$, and by~\eqref{fattore} they are orthogonal for the symplectic form $\omega_g$.  Thus either $W_{+}\cap W_{-}=\{0\}$, or $W_{+}=W_{-}$. In the former case $H^{1,0}$ is the graph of a non degenerete skew-symmetric map $f\colon V_{\CC}\to V^{\vee}_{\CC}$ such that $\Pf(f)=\vartheta_1/\vartheta_2$, in the latter case $H^{1,0}=U\oplus U^{\bot}$ for a $2$ dimensional subspace (equal to $W_{+}=W_{-}$) of $V_{\CC}$. Hence in both cases $H^{1,0}=H^{1,0}_{[\sigma]}(\vartheta)$ for some $[\sigma]\in\cD$. Since $H^{1,0}$ is isotropic for $\la,\ra_{\vartheta,h}$, we have 
$[\sigma]\in\cD_{h}$, and actually $[\sigma]\in\cD^{+}_{h}$ by the ampleness of $\Theta_{[\sigma]}(\vartheta)$. 
\end{proof}
\subsection{Proof of the third main result}
\setcounter{equation}{0}
We prove~\Ref{thm}{terzoteor}. The isogeny between $\KS(X,L)$ and $J^3(X)$ follows from Item~(1) of~\Ref{thm}{primoteor} and  results of F.~Charles~\cite{charles-univ} and van Geemen, Voisin~\cite{voisin-bert}. 

Next, we prove the other statements of the theorem. 
There exists an isomorphism
\begin{equation*}
\varphi\colon J_{[\sigma(X)]}(\vartheta)\overset{\sim}{\lra} J^3(X),
\end{equation*}
where $[\sigma(X)]=H^{2,0}(X)$ is the period point of $X$, and $\vartheta=\vartheta(\ov{q}_X^{n-2})$. If $n=2$, then $\varphi^{*}\Theta_L\equiv \Theta_{[\sigma(X)]}(\vartheta)$, and hence~\Ref{thm}{terzoteor} follows from~\Ref{thm}{sorpresa}. Now let $n>2$. Since  the definitions of $\varphi^{*}\Theta_L$ and $\Theta_{[\sigma(X)]}(\vartheta)$ are different (see~\Ref{expl}{posteta}), we argue as follows. For a very generic $X$ the N\'eron-Severi groups of  $ J_{[\sigma(X)]}(\vartheta)$ and of $J^3(X)$ have rank $1$, and hence 
there exists $c\in\QQ_{+}$ such that $\varphi^{*}\Theta_L\equiv c\Theta_{[\sigma(X)]}(\vartheta)$. Thus~\Ref{thm}{terzoteor} follows from~\Ref{thm}{sorpresa}.  
\begin{rmk}
The proof of~\Ref{thm}{terzoteor} provides the following non trivial statement. Let $X$ be a HK of Kummer type, of dimension $2n$. Let $\gamma\in H^2(X;\QQ)$ be a class of positive square. Then there exists $c_{\gamma}\in\QQ^{*}$ such that for all $\alpha,\beta\in H^3(X;\QQ)$
\begin{equation}
\int_X\alpha\smile\beta\smile \gamma^{2n-3}=c_{\gamma}\int_X \alpha\smile\beta\smile \gamma\smile (q_X^{\vee})^{n-2}.
\end{equation}
\end{rmk}
\subsection{An example}
\setcounter{equation}{0}
We work out one example in order to emphasize that our procedure is very explicit. We assume that $(X,L)$ is a polarized HK fourfold of Kummer type, 
and that $q_K(L)=2$ and $c_1(L)$ has divisibility $1$. By~\Ref{thm}{terzoteor} there exists an injection 
$\QQ[\sqrt{-3}]\subset\End(J^3(X),\Theta_L)_{\QQ}$. Since $\Det\Theta_L\equiv 1$, it follows that $J^3(X)$ is isogenous to the Prym variety of an \'etale cyclic triple cover $\wt{C}\to C$, where $C$ is a curve of genus $3$ (possibly a stable curve), i.e.~examples considered by Schoen~\cite{schoen-hc}, see also Section 7 in~\cite{berthodge}. We recall that these  abelian fourfolds are of Weil type, with an endomorphism which is a (non trivial) 
cube root of $\Id$, and the determinant of the Weil polarization is $1$.

By~\Ref{rmk}{menosei}, we may identify $(J^3(X),L)$ with $(J_{[\sigma]}(\vartheta),\Theta_{[\sigma]}(\vartheta))$, where $[\sigma]$ is the period point of $X$, $\vartheta=(-1,-3,-3)$ and $h=v^{\vee}_1\wedge v^{\vee}_2+ v^{\vee}_3\wedge v^{\vee}_4$ (we adopt the notation introduced in the proof of~\Ref{thm}{sorpresa}). Thus (in the notation introduced in the proof of~\Ref{thm}{sorpresa}) $N=1/3$ and $b=0$. It follows that 
\begin{equation*}
\begin{matrix}
 V_{\CC}\oplus V_{\CC}^{\vee} & \overset{\Psi_0}{\lra} & V_{\CC} \oplus V_{\CC}^{\vee} \\
(v,\ell) & \mapsto & ( 3 g^{-1}(\ell), - g(v))
\end{matrix}
\end{equation*}
defines an endomorphism of $(J^3(X),L)$ such that $\Psi_0\circ \Psi_0=-3\Id$. The endomorphism of $(V_{\CC}\oplus V_{\CC}^{\vee})$ defined by the cube root of $\Id$ given by
 $\Omega:=-(\Id +\Psi_0)/2$ does not map $V\oplus V^{\vee}$ to itself, and hence does not descend to an endomorphism of   $(J^3(X),L)$. 
On the other hand $\Omega$ does descend to an endomorphism of 
\begin{equation*}
Y:=J_{[\sigma]}(\vartheta)/ \{\overline{(v/2,g(v)/2)} \mid v\in V\}.
\end{equation*}
Now $Y$  is one of the abelian fourfolds of Weil type considered by Schoen, and the quotient map $J_{[\sigma]}(\vartheta)\to Y$ has a kernel isomorphic to $\FF_2^4$.

 \bibliography{ref-int-jac-of-kumm-type}
 \end{document}